\documentclass[11pt]{article}
\usepackage[final]{epsfig}
\usepackage[left=3cm,right=3cm, top=2.5cm,bottom=2.5cm,bindingoffset=0cm]{geometry}
\usepackage{graphics}
\usepackage{amsmath}
\usepackage{amsfonts}
\usepackage{latexsym}
\usepackage{amssymb, mathabx}
\usepackage{amsthm}
\usepackage{graphicx}
\usepackage{epstopdf}
\usepackage{multicol,multirow}
\usepackage{hyperref, enumitem}
\usepackage{mathtools}

 \usepackage{relsize}

\usepackage[nottoc]{tocbibind}

\usepackage{amsthm}

\mathtoolsset{showonlyrefs}

\usepackage{tikz}
\usetikzlibrary{positioning}
\usetikzlibrary{decorations.text}
\usetikzlibrary{decorations.pathmorphing}
\usetikzlibrary{decorations.markings}

\usetikzlibrary{arrows} 
\tikzset{
    >=stealth',
    pil/.style={
           ->,
           thick,
           shorten <=2pt,
           shorten >=2pt,}
}
\tikzset{->-/.style={decoration={
  markings,
  mark=at position .7 with {\arrow{>}}},postaction={decorate}}}
  \tikzset{a/.style={decoration={
  markings,
  mark=at position .52 with {\arrow{angle 90}}},postaction={decorate}}}
\tikzset{-<-/.style={decoration={
  markings,
  mark=at position .4 with {\arrow{<}}},postaction={decorate}}}

\makeatletter
\def\thmhead@plain#1#2#3{%
  \thmname{#1}\thmnumber{\@ifnotempty{#1}{ }\@upn{#2}}%
  \thmnote{ {\the\thm@notefont#3}}}
\let\thmhead\thmhead@plain

\makeatother

\newcounter{AppCounter}

\def\restrict#1{\raise-.5ex\hbox{\ensuremath|}_{#1}}

\newtheorem{lemma}{Lemma}[section]
\newtheorem{proposition}[lemma]{Proposition}
\newtheorem{hproposition}[lemma]{``Proposition''}
\newtheorem{remark-definition}[lemma]{Remark-Definition}
\newtheorem{theorem}[lemma]{Theorem}
\newtheorem{corollary}[lemma]{Corollary}

\newtheorem{proposition-conjecture}[lemma]{Proposition-conjecture}

\theoremstyle{definition}
\newtheorem{example}[lemma]{Example}

\newtheorem{definition}[lemma]{Definition}
\newtheorem{remark}[lemma]{Remark}

\newcommand{\proofend}{\hfill$\Box$\bigskip}

\newcommand{\R}{{\mathbb R}}

\newcommand{\G}{\mathcal G}
\newcommand{\B}{  B}
\newcommand{\metric}{\langle \,, \rangle}       
\renewcommand{\div}{\mathrm{div} \,}

\newcommand{\SDiff}{\mathrm{SDiff}}
\newcommand{\DSDiff}{\mathrm{DSDiff}}
\newcommand{\SVect}{\mathrm{SVect}}
\newcommand{\vect}{\Vect}
\newcommand{\svect}{\SVect}
\newcommand{\dvect}{\mathrm{DVect}}
\newcommand{\dgrad}{\mathrm{DGrad}}
\newcommand{\dforms}{\mathrm{D}\Omega^1}
\newcommand{\dexactforms}{\mathrm{D}\Omega_{ex}^1}
\newcommand{\dccforms}{\mathrm{D}\Omega_{cc}^1}
\newcommand{\dcinfty}{{D}\Cont^\infty}
\newcommand{\dsvect}{\mathrm{DSVect}}
\newcommand{\dsvectext}{ {\dsvect'}}

\newcommand{\Ker}{\mathrm{Ker}}
\renewcommand{\Im}{\mathrm{Im}}
\newcommand{\grad}[1]{\nabla #1}
\newcommand{\Grad}{\mathrm{Grad}}
\newcommand{\diff}[1]{{d}  #1}

\newcommand{\T}{{T}}
\newcommand{\Cont}{{C}}
\newcommand{\LieBracket}{ [\, , ] }

\newcommand{\g}{\mathfrak{g}}
\newcommand{\ad}{\mathrm{ad}}
\newcommand{\Ad}{\mathrm{Ad}}
\newcommand{\id}{\mathrm{id}}

\newcommand{\Hom}{\mathrm{H}}

\newcommand{\Diffeo}{\mathrm{Diff}}
\newcommand{\VS}{{\mathrm{VS}}}
\newcommand{\vs}{{\mathrm{VS}}}
\newcommand{\W}{\textnormal{Dens}} 
\newcommand{\Vect}{\mathrm{Vect}}
\newcommand{\Diff}{\textnormal{Diff}} 
\newcommand{\Sheet}{\Gamma}
\newcommand{\Src}{\mathrm{src}}
\newcommand{\Trg}{\mathrm{trg}}
\newcommand{\units}{{\id}}
\newcommand{\Dom}{D}
\newcommand{\chiplus}{\chi^+_\Sheet}
\newcommand{\chiplust}{\chi^+_{\Sheet_t}}
\newcommand{\chimint}{\chi^-_{\Sheet_t}}
\newcommand{\chimin}{\chi_\Sheet^-}
\newcommand{\Domplus}{\Dom_\Sheet^+}
\newcommand{\Dommin}{\Dom_\Sheet^-}
\newcommand{\Dompm}{\Dom_\Sheet^\pm}
\renewcommand{\L}{\mathcal L}

\newcommand{\A}{\mathcal{A}}
\newcommand{\I}{\mathcal{I}}
\renewcommand{\H}{\mathcal{H}}
\newcommand{\F}{\mathcal{F}}
\renewcommand{\P}{\mathcal{P}}
\newcommand{\tanjump}{{jump^\parallel}}
\newcommand{\normjump}{{jump^\bot}}
\newcommand{\Reg}{{\mathcal R}}
\newcommand{\DTN}{{\mathrm{DtN}}}
\newcommand{\NTD}{{\mathrm{NtD}}}

\newcommand{\low}[1]{\raise-.0ex\hbox{$\scriptstyle #1$}}
\newcommand{\high}[1]{\raise.5ex\hbox{$\scriptstyle #1$}}

\newcommand{\normsq}[1]{\langle #1, #1 \rangle}

\renewcommand{\tfrac}{\frac}

\newcommand{\marginnote}[1]
{
}

\newcounter{ai}

\newcounter{bk}

\title {Vortex sheets and diffeomorphism groupoids}

\author{Anton Izosimov\thanks{
Department of Mathematics,
University of Arizona;
e-mail: {\tt izosimov@math.arizona.edu}
} \,
and Boris Khesin\thanks{
Department of Mathematics,
University of Toronto;
e-mail: \tt{khesin@math.toronto.edu}
} }

\date{}

\begin{document}

\maketitle
\begin{abstract}
In 1966 V.Arnold suggested a group-theoretic approach to ideal hydrodynamics in which the motion of an 
inviscid incompressible fluid is described as the geodesic flow of the right-invariant 
$L^2$-metric on the group of volume-preserving diffeomorphisms of the flow domain. Here we propose 
geodesic, group-theoretic, and Hamiltonian frameworks to include fluid flows with vortex sheets. It turns out 
that the corresponding dynamics is related to a certain groupoid of pairs of volume-preserving diffeomorphisms
with common interface. We also develop a general framework for  Euler-Arnold equations 
for geodesics on groupoids equipped with one-sided invariant metrics.
\end{abstract}

\tableofcontents

\section{Introduction} \label{intro}

Vortex sheets are hypersurfaces of discontinuity in fluid velocity with different speed of fluid layers 
on different sides of the hypersurface. 
They naturally appear, e.g., in the flow past an airplane wing \cite{childress2009introduction}.
In this paper we develop geodesic, group-theoretic, and Hamiltonian frameworks for their description.

In 1966 V.~Arnold  proved that the Euler equation for an ideal fluid describes the geodesic flow of a right-invariant metric on the group of volume-preserving 
diffeomorphisms of  the flow domain \cite{Arn66}.
This insight turned out to be indispensable for the study of Hamiltonian properties
and conservation laws in hydrodynamics, fluid instabilities, topological properties of flows, 
as well as a powerful tool for obtaining sharper existence and uniqueness results for Euler-type equations \cite{AK}. 
However, the scope of applicability of Arnold's approach is limited
to systems whose symmetries form a Lie group. At the same time, there are many problems in fluid
dynamics, such as free boundary problems or (discontinuous)
fluid  flows with vortex sheets, whose symmetries should instead be regarded as a groupoid: e.g., 
only those of the maps corresponding to fluid configurations with moving boundary admit composition, for
which the image of one map coincides with the source of the other.

In this paper we  propose a strategy to extend Arnold's framework to Lie groupoids and develop
a  groupoid-theoretic description for incompressible  fluid  flows with vortex sheets, 
i.e., flows whose  velocity field has a jump discontinuity along a hypersurface. It turns out that the 
corresponding configuration space  has a natural groupoid structure. By using this {\it vortex sheet groupoid}
instead of the {\it diffeomorphisms group} in Arnold's description and describing the corresponding
algebroid, we obtain a geometric interpretation for discontinuous fluid flows.  We prove that vortex sheet type
 solutions of the Euler equation are
precisely the geodesics of an $L^2$-type right-invariant metric on the Lie groupoid of discontinuous
volume-preserving diffeomorphisms. 

\medskip

\subsection[{Groupoid framework for vortex sheets}]{Groupoid framework for vortex sheets}

Recall that  the hydrodynamical Euler equation for an  ideal incompressible fluid filling a Riemannian manifold $M$
(possibly, with boundary $\partial M$) is the following evolution law of the velocity field $u$:
\begin{equation}\label{idealEulerIntro}
\partial_t u+\nabla_u u=-\nabla p\,,
\end{equation}
supplemented by the divergence-free condition ${\rm div}\, u=0$ on $M$
and tangency to the boundary, $u\parallel\partial M$. 
Arnold's theorem sheds light on the origin of this equation:

\begin{theorem}   {\rm \cite{Arn66} }
The Euler equation can be regarded as an equation of the geodesic flow on the group $\SDiff(M)$ 
of volume-preserving diffeomorphisms of $M$ with respect to the right-invariant 
 metric given at the identity of the group by the squared $L^2$-norm of the fluid's velocity field (i.e., the fluid kinetic energy\footnote{The 
$L^2$-metric is twice the kinetic energy of the fluid, which leads to a simple time rescaling, and we will not be mentioning this throughout 
the paper.}).
\end{theorem}

This setting assumes sufficient smoothness of the initial velocity field $u$. 
In particular, it does not, generally speaking, describe flows with vortex sheets, i.e. with jump discontinuities 
in the velocity.
On the other hand, it was recently discovered by F.\,Otto and C.\,Loeschcke~\cite{Loesch} 
that the motion of vortex sheets is also 
governed by a geodesic flow, but of somewhat different origin. 
Consider the space $\VS(M)$ of vortex sheets (of a given topological type) in $M$, 
i.e. the space of hypersurfaces which bound fixed volume in $M$.
Define the following (weak) metric  on the space $\VS(M)$. A tangent vector to a point $\Sheet$ 
in the space  of all vortex sheets $\VS(M)$ can be regarded as   a vector field $v$  
attached at the vortex sheet $\Sheet \subset M$ and normal to it. 
Then its square length is set to be
\begin{align}\label{vsmetric}
\normsq{v}_{\low\vs}:=\inf\left\{\normsq{u}_{ L^2} \mid {\div} u=0 \mbox{ and } (u,\nu)\,\nu = v  \mbox{ on } \Sheet\right\}\,
\end{align}
where $\normsq{u}_{L^2} := \int_M (u,u)\,\mu$ is the squared $L^2$-norm of a vector field $u$ on $M$, 
and $\nu$ is the unit normal field to $\Sheet$ (see Definition \ref{def:baseMetric} below).
Then the fluid flow with such vortex sheets satisfies the following variational principle:

\begin{theorem}\label{LoeschThm0}{\rm \cite{Loesch}}
Geodesics with respect to the metric $\metric_{\vs}$ on  the space  $\VS(M)$ 
describe the motion of  vortex sheets in an incompressible flow 
which is globally potential outside of the vortex sheet. 
\end{theorem}

\medskip

To unify these two geodesic approaches, as well as to develop Arnold's approach to cover velocity fields 
with discontinuities, we introduce the Lie  groupoid $\DSDiff(M)$ of volume-preserving diffeomorphisms 
of a manifold $M$ that are discontinuous along a hypersurface. Namely, the elements of  
$\DSDiff(M)$ are  quadruples $(\Sheet_1, \Sheet_2, \phi^+, \phi^-)$, where 
$\Sheet_1, \Sheet_2 \in \VS(M) $ are hypersurfaces (vortex sheets) in $M$ confining the same total volume, while  $\phi^\pm \colon \Dom^\pm_{\Sheet_1} \to \Dom^\pm_{\Sheet_2}$ 
are volume preserving diffeomorphisms between  connected  components of $M\, \setminus \, \Sheet_i$ 
denoted by $\Dom_{\Sheet_i}^+, \Dom_{\Sheet_i}^-$.  
The multiplication of the quadruples in $\DSDiff(M)$ is given by the natural composition of discontinuous diffeomorphisms and is shown in Figure \ref{fig:composition}.
\medskip

 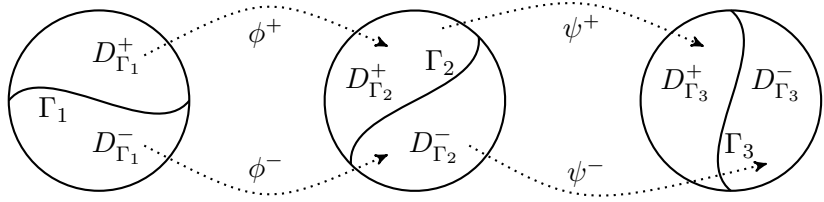
\begin{figure}[t]
\centerline{
\begin{tikzpicture}[thick, scale = 1.2]
 \node  at (0,0) () {
 \begin{tikzpicture}[thick, scale = 1.2]
    \draw (1, 1) ellipse (1cm and 1cm);
     \draw (0,1)  .. controls (0.4,1.5) and  (1.6,0.5) ..  (2,1);
      \node  at (0.5,0.9) () {$\Sheet_1$};
        \node  at (1.2,1.5) () {$\Dom^+_{\Sheet_1}$};
             \node  at (1.2,0.5) () {$\Dom^-_{\Sheet_1}$};
     \end{tikzpicture}
     };
      \node  at (3.5,0) () {
 \begin{tikzpicture}[thick, rotate = 45, scale = 1.2]
    \draw (1, 1) ellipse (1cm and 1cm);
     \draw (0,1)  .. controls (0.4,1.5) and  (1.6,0.5) ..  (2,1);
        \node  at (1.5,1.1) () {$\Sheet_2$};
                \node  at (0.8,1.5) () {$\Dom^+_{\Sheet_2}$};
             \node  at (0.8,0.5) () {$\Dom^-_{\Sheet_2}$};
     \end{tikzpicture}
     };
           \node  at (7,0) () {
 \begin{tikzpicture}[thick, rotate = 90, scale = 1.2]
    \draw (1, 1) ellipse (1cm and 1cm);
     \draw (0,1)  .. controls (0.4,1.5) and  (1.6,0.5) ..  (2,1);
          \node  at (0.5,0.9) () {$\Sheet_3$};
                  \node  at (1.2,1.5) () {$\Dom^+_{\Sheet_3}$};
             \node  at (1.2,0.5) () {$\Dom^-_{\Sheet_3}$};
     \end{tikzpicture}
     };
     \draw [dotted, ->] (0.5,0.5) .. controls (1.85, 1.2) .. (3.2,0.6);
          \draw [dotted, ->] (0.5,-0.5) .. controls (1.85, -1.2) .. (3.2,-0.6);
            \node  at (1.85,0.75) () {$\phi^+$};
                     \node  at (1.85,-0.75) () {$\phi^-$};
                          \draw [dotted, ->] (3.8,0.8) .. controls (5.35, 1.2) .. (6.7,0.6);
            \node  at (5.35,0.8) () {$\psi^+$};
                    \draw [dotted, ->] (4.1,-0.5) .. controls (5.35, -1.2) .. (7.4,-0.7);
                      \node  at (5.4,-0.8) () {$\psi^-$};
\end{tikzpicture}
}
\caption{Elements of the groupoid  $\DSDiff(M)$ and their composition rule.}\label{fig:composition}
\end{figure}

\medskip

The  infinitesimal object corresponding to this Lie groupoid is the Lie algebroid $\dsvect(M) \to \VS(M)$, which 
is the space of ``possible velocities" of the fluid with a vortex sheet.
Given a vortex sheet $\Sheet$, the corresponding velocities are (discontinuous) vector fields on $M$ 
of the form $u = \chiplus u^+ + \chimin u^-\!,$
where $\chiplus, \chimin$ are the indicator functions of the connected components $\Dompm$ 
of $M\, \setminus \,\Sheet$, and $u^\pm$ are smooth divergence-free vector fields on $\Dompm$ such that 
the restrictions of $u^+$ and $u^-$ to $\Sheet$ have the same normal component, see Section \ref{sec:kinematics}.
The map from such  vector fields $u$ to their normal components on $\Sheet$  is the so-called anchor map $\#$
of the corresponding algebroid.
Note that such vector fields discontinuous along $\Sheet$ do not have a Lie algebra structure, 
as the Lie bracket of two such fields will not, in general, be a field with matching normal components on $\Sheet$.
There is, however, a Lie bracket on sections of those fields (explicitly given in Section \ref{sect:algebroid_vf}).

We describe below how to define a right-invariant $L^2$-metric on this groupoid and construct an analog of the geodesic Euler-Arnold equation. Recall that the Euler equation in a manifold $M$ for a fluid flow discontinuous along a vortex sheet $\Sheet\subset M$ has the form:
\begin{align}\label{twoPhaseEuler}
\begin{cases}
\partial_t u^+ + \nabla_{u^+} u^+ = -\grad p^+\!, \\ 
\partial_t u^- + \nabla_{u^-} u^-  = -\grad p^-\!,\\
 \qquad \partial_t \Sheet = \# u\,,
\end{cases}
\end{align}
where $u = \chiplus u^+ + \chimin u^-$ is the fluid velocity, ${\rm div}\, u^\pm=0$, and 
$p^\pm \in \Cont^\infty(\Dompm)$ are functions satisfying the continuity condition
$p^+\vert_{\Sheet} = p^-\vert_{\Sheet}$. These equations naturally arise from the weak form of the Euler equation, see
Appendix \ref{app:weak}. The first main result of the paper is the following

\begin{theorem}{\bf (=Theorem \ref{thmMain2})}
The Euler equation \eqref{twoPhaseEuler} for a fluid flow with a vortex sheet $\Sheet\subset M$
is the groupoid  Euler-Arnold equation  corresponding to the 
$L^2$-metric on the algebroid $\dsvect(M)$. Equivalently, the Euler equation \eqref{twoPhaseEuler} is a geodesic
equation for the right-invariant $L^2$-metric on (source fibers of) the Lie groupoid $\DSDiff(M)$ 
of discontinuous volume-preserving diffeomorphisms.
\end{theorem}

\begin{remark}
One can  see that the standard hydrodynamical Euler equation is a particular case of the above equations 
with a vortex sheet (where $\Sheet$ is empty). In this case the space 
$\VS(M)$ of vortex sheets consists of one point, while the groupoid $\DSDiff(M)$ of diffeomorphism pairs
becomes the group  $\SDiff(M)$ of volume-preserving diffeomorphisms.
\end{remark} 

We note that the geodesics on the groupoid turn out to be weak solutions of the Euler equation 
with vortex sheet initial data, as we show in Appendix \ref{app:weak}. 
Furthermore, this equation can be described within the Hamiltonian framework:

\begin{theorem}{\bf (=Theorem \ref{thmMain})}
The Euler equation \eqref{twoPhaseEuler} for a  flow with a vortex sheet 
written on the dual  $\dsvect(M)^*$ of the algebroid is Hamiltonian with respect to 
the natural Poisson structure on  the dual algebroid and the Hamiltonian function given by the $L^2$ kinetic energy.
\end{theorem}

This theorem is an analog of the Hamiltonian property of the 
Euler-Arnold equation  on the dual to a Lie algebra with respect to the Lie-Poisson structure.

\medskip

Return to the metric properties of the groupoid Euler-Arnold equation.  
Now the metric~\eqref{vsmetric} on vortex sheets appears as a natural projection from the groupoid of diffeomorphism pairs 
to the space of vortex sheets.
Namely, given initial vortex sheet $\Sheet \in \VS(M)$, consider the subset $\DSDiff(M)_{\Sheet}\subset 
\DSDiff(M)$ of pairs of diffeomorphisms with  domains  $\Dom^\pm_{\Sheet}$ (a so-called \textit{source fiber} of the groupoid $\DSDiff(M)$)
and equipped with the right-invariant source-wise $L^2$-metric on $\DSDiff(M)$. 
Then the following statement generalizes Theorem \ref{LoeschThm0}:

\begin{theorem}{\bf (=Theorem \ref{geodDescription})}
For any initial vortex sheet $\Sheet \in \VS(M)$, the groupoid target mapping 
$\Trg \colon (\DSDiff(M)_{\Sheet}, \metric_{L^2}) \to (\VS(M), \metric_\vs)$ is a Riemannian submersion
to the space  $\VS$ of vortex sheets equipped with the metric  $\metric_{\vs}$.
In particular, horizontal geodesics on  $\DSDiff(M)_{\Sheet}$ project to geodesics on vortex sheets $\VS$.
 \end{theorem}
\begin{figure}[t]
\centerline{
\begin{tikzpicture}[thick, scale = 1.2]
\draw [thin] (0,-1.) -- (5,-1.);
  \fill [opacity = 0.05] (-0.2,3) -- (5.2,3) -- (5.2,0.5) -- (-0.2,0.5) -- cycle;
\draw [very thin] (0,0.7) -- (0,2.8);
\draw [very thin] (1,0.7) -- (1,2.8);
\draw [] (2,0.7) -- (2,2.8);
\draw [very thin] (3,0.7) -- (3,1.5);
\draw [very thin] (4,0.7) -- (4,1.5);
\draw [very thin] (5,0.7) -- (5,1.5);
\draw [very thin] (4,2.8) -- (4,2);
\draw [very thin] (5,2.8) -- (5,2);
\draw [very thin] (3,2) -- (3,2.3);
\draw [very thin] (3,2.7) -- (3,2.8);
\draw  [very thick, ->] (2,1.5) -- (2.5,1.5);
   \fill (2,1.5) circle [radius=1.5pt];
      \fill (2,-1.) circle [radius=1.5pt];
      \node  at (1.7,1.5) () {$\id_\Sheet$};
      \draw [thin] (2,1.35) -- (2.15,1.35) -- (2.15,1.5);
           \node []  at (3.8,1.75) () {$u = \chiplus \grad f^+ + \chimin \grad f^-$};
          \node  at (2,-1.2) () {$\Sheet$};
                 \node  at (3.4,-0.75) () {$\xi = \#u := \Trg_*u$};
          \draw  [very thick, ->] (2,-1) -- (2.5,-1);
                   \draw  [->] (2,0.2) -- (2,-0.5);   
                             \node  at (2.3,-0.1) () {$\Trg$};    
                                       \node  at (0.2,-1.25) () {$\VS(M)$};  
                                          \node  at (0.2,0.25) () {$\DSDiff(M)_\Sheet$};      
                                                                                    \node  at (2.95,2.5) () {$\DSDiff(M)_\Sheet^\Sheet$};      
 %
\end{tikzpicture}
}
\caption{Riemannian submersion for the groupoid. Here $\DSDiff(M)_\Sheet^\Sheet$ is the group of discontinuous diffeomorphisms $\SDiff(\Domplus) \times \SDiff(\Dommin)$, and $u$ is a horizontal vector field projecting to $\xi = \#u$. The latter can be regarded as the velocity of the vortex sheet $\Sheet$, and $\langle \xi, \xi \rangle_\VS = \langle u, u \rangle_{L^2}$. }\label{fig:submersion}
\end{figure}
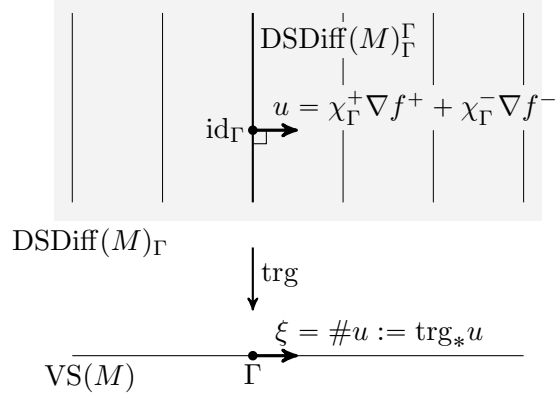

In particular, the submersion property means that  horizontal geodesics in the space $\DSDiff(M)_{\Sheet}$ 
correspond to gradient fields, which implies Theorem \ref{LoeschThm0}.
Note that these gradient fields have the form $u =  \chiplus \grad f^+ + \chimin \grad f^-$ in $M$, see Corollary \ref{cor:horiz} below and Figure \ref{fig:submersion}. This also allows one to define a nonlocal $H^{-1/2}$-type metric
 of hydrodynamical origin (based on Neumann-to-Dirichlet operators) 
 on the space of shapes, and it contrasts previously considered local  $H^s$ metrics on those spaces with $s\ge 0$, 
 see Appendix \ref{app:shapes}.
\begin{remark}The 
smoothness of the groupoid and algebroid is understood below 
in the Fr\'{e}chet $C^\infty$ setting. Similarly, one can consider the setting of Hilbert manifolds modeled on Sobolev $H^s$ spaces for sufficiently large $s$, 
cf. \cite{EbMars70}.\end{remark}
\medskip



\subsection[Examples]{Examples}

First consider solutions of the Euler equation that are irrotational outside of vortex sheets, i.e. those whose 
vorticity includes only singular part supported on $\Sheet$. The corresponding velocity is locally potential 
outside of $\Sheet$. If the velocity field is globally potential (initially and hence for all times) outside $\Sheet$, 
then such a solution corresponds to a horizontal geodesic on  $\DSDiff(M)_{\Sheet}$ 
and to a geodesic on the space $\VS(M)$ of vortex sheets (we call such solutions \textit{pure vortex sheet motions}).
In the case where the velocity field is only locally potential (we call such solutions \textit{irrotational  flows with vortex sheets}) the corresponding solution 
can be understood as a trajectory of a Newtonian system in a magnetic field.

 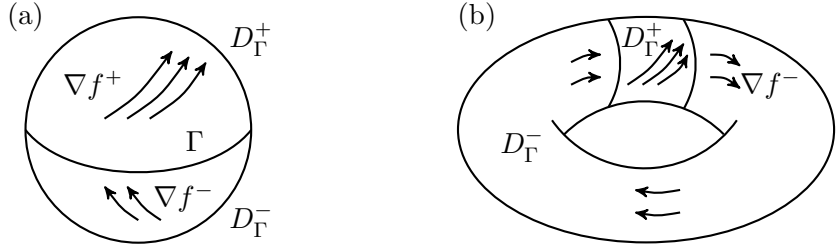
\begin{figure}[t]
\centerline{
\begin{tikzpicture}[thick, scale = 1.5]
\node at (-1,1) () {(a)}; 
\node at (3,1) () {(b)}; 
\node at (1,0.8) () {$\Domplus$}; 
\node at (1,-0.8) () {$\Dommin$}; 
\node at (0,0) (){
\begin{tikzpicture}[thick, scale = 1.5]
    \draw (1, 1) ellipse (1cm and 1cm);
     \draw (0,1)  .. controls (0.4,0.5) and  (1.6,0.5) ..  (2,1);
      \node  at (1.5,0.9) () {$\Sheet$};
        \node  at (0.6,1.4) () {$\grad f^+$};
             \draw [->] (0.7,1.1) .. controls (1,1.3) .. (1.3,1.7);
                    \draw [->] (0.9,1.1) .. controls (1.2,1.3) .. (1.45,1.65);
                           \draw [->] (1.1,1.1) .. controls (1.4,1.3) .. (1.6,1.6);
                               \draw [<-] (0.7,0.5) .. controls (0.85,0.3) .. (1,0.2);
                                     \draw [<-] (0.9,0.5) .. controls (1.05,0.3) .. (1.2,0.2);
                                          \node  at (1.4,0.4) () {$\grad f^-$};   
          \end{tikzpicture}
          };
          \node at (4.5,0) () {
          \begin{tikzpicture}[thick, rotate = 0, scale = 1]
       \draw [->] (0.7,1.1) .. controls (1,1.3) .. (1.3,1.7);
                    \draw [->] (0.9,1.1) .. controls (1.2,1.3) .. (1.45,1.65);
                           \draw [->] (1.1,1.1) .. controls (1.4,1.3) .. (1.5,1.5);
                           \draw [->] (1.8,1.5) .. controls (2,1.5) .. (2.2,1.35);
                             \draw [->] (1.8,1.2) .. controls (2,1.2) .. (2.2,1.05);
                                \draw [->] (-0.05,1.4) .. controls (0.15, 1.5) ..  (0.35,1.5);
                                     \node  at (2.6,1.1) () {$\grad f^-$};
                                   \draw [->] (-0.05,1.1) .. controls (0.15, 1.2) ..  (0.35,1.2);
                                         \draw [<-] (0.8,-0.6) .. controls (1.1,-0.65) ..  (1.4,-0.6);
                                              \draw [<-] (0.8,-0.3) .. controls (1.1,-0.35) ..  (1.4,-0.3);
                                               \node  at (-0.7,0.3) () {$\Dommin$};
                                                \node  at (0.9,1.7) () {$\Domplus$};
 \node  at (0.95,0.5) () {
\begin{tikzpicture}[thick, rotate = -90, scale = 1.25]
    \draw (3,-3) ellipse (1.2cm and 2cm);
    \draw   (3.05,-3.87) arc (225:135:1.2cm);
    \draw   (2.9,-4) arc (-55:55:1.2cm);
        \draw   (1.83,-3.4) arc (120:60:0.93cm);
        \draw   (1.83,-2.6) arc (120:60:0.95cm);
     \end{tikzpicture}
     };
    %
\end{tikzpicture}
          };
\end{tikzpicture}
}
\caption{(a) Pure vortex sheet solutions on a sphere. Their motion is 
 a geodesic on $\VS(S^2)$. (b) Irrotational, but not pure, vortex sheet solutions.
Their motion  can be viewed as a trajectory of a Newtonian system in a magnetic field on  $\VS(T^2)$.
}\label{fig:curve_onS2}
\end{figure}
\begin{example} 
Consider a closed curve $\Sheet$ on a two-dimensional sphere $M=S^2$ as a vortex sheet, see Figure \ref{fig:curve_onS2}(a).
In this case the domains $\Dom^\pm_{\Sheet}$ are simply-connected, and hence irrotational fields in $\Dompm$ are globally potential with harmonic potentials $f^\pm$. The curve motion is defined by the normal to $\Sheet$ vector field, and hence the potentials $f^\pm$ are solutions of the corresponding Neumann problem, see Section  \ref{sect:pure_motion}. The common normal component of $\nabla f^\pm$, i.e., the curve velocity, can be regarded as a tangent vector to the space of vortex sheets.
The corresponding curve motion is a geodesic in the metric $\metric_\vs$
on the space of curves in $S^2$ bounding the same area, see Theorem \ref{LoeschThm0}.

\end{example}
\medskip

\begin{example} 
Assume now that the vorticity has support only on $\Sheet$, but the field is only locally potential (i.e., the corresponding potential is multivalued). 
For instance, consider a two-torus $M=T^2$ with vortex sheet $\Sheet = \Sheet_1\cup \Sheet_2$ being 
the union of two cross-sections, see Figure \ref{fig:curve_onS2}(b), where the velocity field in 
$\Dom^+_\Sheet$ has positive circulation in the meridional direction, while  the velocity field in 
$\Dom^-_\Sheet$ is potential. 
The corresponding irrotational vortex sheet solutions 
can be described as a Hamiltonian motion on the cotangent bundle $T^*\VS(T^2)$, 
which is a natural system with kinetic and potential energy and where the standard symplectic structure 
is twisted by adding a ``magnetic term", a closed 2-form on the base $\VS(T^2)$, see Section \ref{sec:irrot}. 
The geodesic motion itself
would correspond to a pure kinetic energy in the standard symplectic structure on $T^*\VS(T^2)$, 
see Section  \ref{sect:pure_motion}.

\medskip
\end{example}

\medskip

Finally, note that general solutions with vortex sheets include both a smooth part of the vorticity 
and its singular part. As we mentioned, they correspond to arbitrary geodesics in the right-invariant 
 $L^2$-metric on the groupoid $\DSDiff(M)$ and have a Hamiltonian description as well. 
 \begin{figure}[t]
\centerline{
\begin{tikzpicture}[thick, scale = 1.5]
\node at (0,0) (){
\begin{tikzpicture}[thick, scale = 1.5]
    \draw (1, 1) ellipse (1cm and 1cm);
     \draw [dashed] (0,1)  .. controls (0.4,0.5) and  (1.6,0.5) ..  (2,1);
         \draw [] (1,0)  .. controls (1.5,0.4) and  (1.5,1.6) ..  (1,2);
          \node  at (1.05,1.7) () {$\Sheet$};
                   \node  at (0.7,0.5) () {$\omega = c$};
                    \draw [->] (0.8, 0.9) -- (1.1, 1.2);
                          \draw [->] (0.8, 1.1) -- (1.1, 1.4);
                              \node  at (0.65,1.05) () {$u^+$};
                                                  \draw [<-] (1.8, 0.9) -- (1.5, 1.2);
                                                     \draw [<-] (1.8, 1.1) -- (1.5, 1.4);
                                                          \node  at (1.6,1.55) () {$u^-$};
          \end{tikzpicture}
          };
          \draw [->] (1.6,0) -- (2.1,0);
\node at (4,0) (){
\begin{tikzpicture}[thick, scale = 1.5]
    \draw (1, 1) ellipse (1cm and 1cm);
         \draw [] (1,0)  .. controls (1.7,0.4) and  (1.7,1.6) ..  (1,2);
            \draw [dashed] (0.05,0.8)  .. controls (0.4,0.7) and  (1.1,0.8) ..  (1.5, 1.2);
               \draw [dashed] (1.4, 0.4) .. controls (1.6,0.5) and  (1.85,0.6) ..  (1.97,0.8);
                                  \node  at (0.7,0.65) () {$\omega = c$};
                                      \node  at (2.35, 0.7) (A) {$\omega= c$};
          \node  at (1.1,1.7) () {$\Sheet$};
          \end{tikzpicture}
          };          
 \end{tikzpicture}
}
\caption{Smooth levels of the vorticity $\omega$ on a sphere are transported by the flow with velocity~$u^\pm$  and hence getting broken  along the vortex sheet $\Sheet$. 
}\label{fig:sing+smooth}
\end{figure}
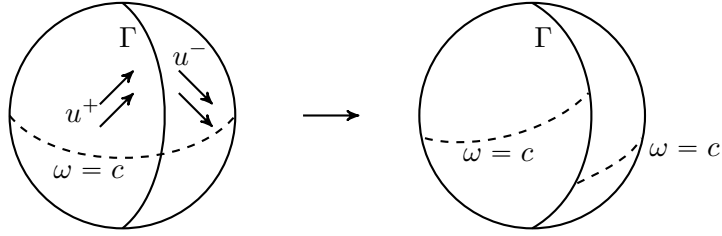
\begin{example}   \label{ex:sing+smooth}
Consider a vortex sheet solution on a two-dimensional sphere  $M=S^2$, where 
initially the singular part of the vorticity, 
the vortex sheet itself, is  a meridian, while level curves of the smooth part of the vorticity
are parallels, i.e. flat horizontal sections, see Figure~\ref{fig:sing+smooth}. 
Then on  different sides of the  meridian the velocity has a jump and it starts moving the frozen-in vorticity 
in different directions. Indeed, the fluid flow in this case is a combination of two 
motions: the smooth part of the vorticity pulls the fluid along horizontal sections, while the singular part of the vorticity rotates eastern and western hemispheres
with respect to each other. 
So the vorticity levels will immediately break once the motion starts. 
This example shows that such a system 
(as well as a generic initial condition with both singular and smooth vorticity parts present) does not admit 
a description within the framework of smooth diffeomorphisms. The vorticity as a whole is not transported by the \textit{smooth} flow
any longer and the introduction of a diffeomorphism groupoid is a necessity!

\end{example}

\begin{remark}
It is also worth mentioning that in 2D the evolution of vortex curves in $\R^2$ is governed by the 
Birkhoff-Rott equation, which is obtained from the Euler equation \eqref{twoPhaseEuler} provided that one can regard the strength of the singular vorticity as a parameter 
along the vortex curve. 
The framework in this paper provides a similar description in the 
setting of a compact ambient manifold (say, sphere or torus) and without assuming
that kind of parametrization. 
\end{remark}

\medskip


\subsection{General groupoid framework for Euler-Arnold-type 
equations }

Since the groupoid-theoretic  framework for Euler-Arnold-type equations 
developed in the paper has a universal character, we provide in Table \ref{table1} 
a dictionary of corresponding notions in the cases of groups and groupoids, as well as their 
hydrodynamical implementation.

\begin{table*}[h]
\scriptsize
\begin{center}
\begin{tabular}{c|c|c|c}
general notion  /  &  group setting / & groupoid setting / &   section for general /  \\
hydro setting     &  ideal fluid on $M$ & flow with vortex sheet &  hydro setting \\
\hline
\hline
   &   &  &  \\
   &  Lie  group $G$   & Lie  groupoid  $\G \rightrightarrows \B$ & Section \ref{sect:groupoid} \\
 configuration space  &   ---------------------   &   ---------------------  &  --------------------- \\
  (positions) &    $\SDiff(M)$    &     $\DSDiff(M)\rightrightarrows \VS(M)$ & \\
   &   volume-preserving   &   groupoid of  & Section \ref{sect:diffeo_pair}  \\
  &  diffeomorphism group & diffeomorphism pairs &  \\
\hline
   &   &  &  \\
   &  Lie  algebra $\frak g$   & Lie  algebroid  $\A \to \B$ & Section \ref{sect:algebroid} \\
   phase space &   ---------------------   &   ---------------------  &  --------------------- \\
   (velocities)   &    ${\SVect}(M)$    &     $\dsvect(M) \to \VS(M)$ & \\
   &   divergence-free   &   algebroid of discontinuous & Section \ref{sect:algebroid_vf}  \\
  &  vector fields & vector fields  &  \\
\hline
   &   &  &  \\
Lie bracket &   commutator of     &   bracket of sections,  & Section \ref{sect:algebroid_vf}  \\
   &  vector fields &  formula \eqref{sectionsBracket} &   \\
   &   &  &  \\
\hline
   &   &  &  \\
   & Lie algebra dual $\frak g^*$  & algebroid dual  $\A ^*\to \B$ & Section \ref{sect:poisson_br} \\
dual space &      ---------------------   &   ---------------------  &  --------------------- \\
   &$ \Omega^1(M) \, / \, \Omega^1_{ex}(M)$  & $\dforms(M, \Sheet) \, / \, \dexactforms(M,\Sheet) $
                                                                                            & Section \ref{sect:dual} \\
   & 1-form cosets  & discontinuous 1-form cosets  &  \\
\hline
  &   &  &  \\
Poisson structure     &  Lie-Poisson bracket on  $\frak g^*$ & Poisson bracket on $\A ^*$  & Section \ref{sect:poisson_br} \\
(Hamiltonian &      ---------------------   &   ---------------------  &  --------------------- \\
 operator)   &  Lie derivative  & Hamiltonian operator  
                        & Section \ref{sect:poisson_vs} \\
  &  ${\rm ad}_u^* =\L_u$ & $\P_{[\alpha]}$  &  \\
\hline
    &   &  &  \\
 inertia operator    &   $I:   \frak g\to\frak g^*$ &   $\I:\A\to \A^*$  & Section \ref{section:algsub} \\
 (kinetic energy)  &  ---------------------   &   ---------------------  &  --------------------- \\
    & metric operator $u\mapsto [u^\flat]$  &  
      $u\mapsto [u^\flat]$  & Section \ref{sect:euler-arn_vs}\\
    &   & on discontinuous vector fields &  \\
\hline
   &   &  &  \\
hydrodynamical   &   $\partial_t [\alpha]+\L_u [\alpha]=0$    &  
 $\partial_t^\Reg [\alpha] + [i_u \diff^\Reg \alpha + \frac{1}{2}\diff^\Reg i_u\alpha]  =0$& Section \ref{sect:euler-arn_vs}\\
 Euler equation  &   &  &  \\
\hline
   &   &  &  \\
base / space of  &  one point   &  $\VS(M)$  & Section \ref{sect:pure_motion}\\
vortex sheets &   &  &  \\
\hline
\end{tabular}
\end{center}
\caption{Group- and groupoid-theoretic  frameworks for Euler-Arnold-type equations }\label{table1}
\end{table*}

This universal groupoid-theoretic  framework can be also applied to include such systems as
flows with a free moving boundary or a rigid body moving in a fluid, as well as the existing algebroid approach to nonholonomic systems with 
symmetry, etc., see e.g. \cite{ML05}. 

\begin{remark}
While the literature on vortex sheets is enormous, here we would like to mention several papers
which might be related to a more geometric point of view, in addition to those mentioned above. 
In \cite{MW83, Kh12}  point vortices and more generally vortex membranes were regarded as singular 
coadjoint orbits of the group of volume-preserving diffeomorphisms. 
They correspond to vorticities with support on codimension-two submanifolds. 
This approach does not quite work for vortex sheets, because the vorticity 
on two sides of the sheet is transported by two different diffeomorphisms, and hence the discontinous flow 
does not need to preserve the orbits of the group coadjoint action, see Example \ref{ex:sing+smooth}.
The paper \cite{LMMR86} gives a Hamiltonian formulation for the incompressible Euler equation 
in the moving boundary case. That approach is based on Hamiltonian reduction of the cotangent bundle 
of a principal bundle, which might be thought of as dual to our groupoid approach, cf. Section \ref{sect:otto}. 
The  groupoid framework proposed below seems to be more natural in this setting, while the derived algebroid 
formulas have a universal character. We will touch on the free boundary problem 
in more detail in a future publication.
Lagrangian formalism for Lie groupoids was discussed in~\cite{We96}. A possibility of using the language of Lie algebroids
in fluid dynamics was also discussed in \cite{Jac12}.
\end{remark}

We expect that this approach can be also  applied to derive  Kelvin-Helmholtz instabilities of vortex sheets
\cite{Mar-Pul}, to study relations to problems of optimal mass transport, 
cf. \cite{KLMP, Villani}, as well as to  use this technique to obtain existence and uniqueness results for the Euler equation with discontinuous initial data, cf. e.g. \cite{Lebeau}.
\medskip

\textbf{Acknowledgements.} We are grateful to 
F.\,Otto for fruitful discussions. 
This research was partially supported by  an NSERC research grant. A.I. would like to thank Max Planck Institute for Mathematics, Bonn, for hospitality and support during some phases of this project.

\medskip

\section{Group setting of ideal hydrodynamics}

\subsection{Geodesic setting for the Euler equation}\label{sect:geod_Euler}
We start with the general setting  of ideal fluid dynamics. 
Let $M$ be an $n$-dimensional Riemannian manifold
with the Riemannian volume form $\mu$ and, possibly, with boundary $\partial M$. 
As we discussed in the introduction, an ideal incompressible fluid filling $M$ moves according to the hydrodynamical Euler equations:
\begin{equation}\label{idealEuler}
\begin{cases}
\begin{aligned}
&\partial_t u+\nabla_u u=-\grad p\,,\\
&{\rm div}\, u=0\quad\text{ and }\quad u\parallel\partial M\,.
\end{aligned}
\end{cases}
\end{equation}
The notation $\nabla_u u$ stands for the Riemannian covariant derivative  of the field $u$ along itself. 
The pressure function $p$ entering the Euler equation 
is defined uniquely modulo an additive constant by the above constraints on the velocity $u$.

\smallskip

Arnold  \cite{Arn66} showed that the Euler equations   can be regarded as an 
equation of the geodesic flow on the group $\SDiff(M):=\{\phi\in \Diffeo(M)~|~\phi^*\mu=\mu\}$ 
of volume-preserving diffeomorphisms of $M$ with respect to the right-invariant $L^2$-metric 
given by the kinetic energy of the fluid. 
In this approach an evolution of the  velocity field $u(t)$ according to the Euler equations is understood as 
an evolution of the vector in the Lie algebra ${\SVect}(M)=\{u\in {\vect}(M) \mid \L_u\mu=0\}$
of divergence-free vector fields on $M$ tangent to the boundary $\partial M$  
(here $\L_u$ stands for the Lie derivative along the field $u$). 
This time-dependent velocity $u(t)$ traces the geodesic 
on the group $\SDiff(M)$ defined by the given initial condition $u(0)=u_0$. 
Note that this setting assumes sufficient smoothness of the initial velocity: it is usually  taken to be 
in the Sobolev $H^s$ space for sufficiently high $s$: $s>\frac n2+1$, where $n=\dim M$, 
cf. \cite{EbMars70}. Alternatively, one can also consider the Fr\'echet setting of 
$C^\infty$-diffeomorphisms, rather than the setting of Hilbert manifolds  of $H^s$-diffeomorphisms.

\begin{remark}\label{rem:EulerDerivation} 
The geodesic interpretation of the Euler equation (\ref{idealEuler}) can be obtained, e.g., for a flat $M$ 
without boundary,  as follows, see
 \cite{Arn66, AK}. Consider the flow $(t,x)\mapsto \phi(t,x)$ defined by the 
velocity field $u(t,x)$, which describes the motion of fluid particles:
$ 
\partial_t \phi(t,x)=u(t,\phi(t,x)), \,\, \phi(0,x)=x 
$
for all $x\in M$ and $t\ge0$. The chain rule for the time derivative immediately gives 
$$
\partial^2_{t} \phi(t,x) 
=(\partial_t u+\nabla_uu)(t,\phi(t,x))\,, 
$$
and hence the Euler equation (\ref{idealEuler}) is equivalent to 
$$ 
\partial^2_{t} \phi(t,x)=-(\nabla p)(t,\phi(t,x))\,, 
$$ 
while the incompressibility condition  $\mathrm{ div }~ u=0$ becomes 
$\det(\partial_x \phi(t,x))=1$ for any $t$. 
The latter form of the Euler equation (for a smooth flow $\phi(t,x)$) 
says that  the  acceleration of the flow at any position $\phi$ is given by the gradient 
of a function (right-translated to $\phi$). Therefore this acceleration  is $L^2$-orthogonal to  the set of 
volume-preserving diffeomorphisms at any point $\phi \in \SDiff(M)$: 
$    \int_M(\nabla p,w)\,\mu = 0 $ 
for all divergence-free vector fields $w$ and all smooth functions $p$ 
on $M$. In other words, any solution 
$\phi(t,.)$ of the Euler equation is  a geodesic line with respect to the induced $L^2$-metric on the group $\SDiff(M)$. 
This $L^2$-metric is the kinetic energy of the fluid flow: for a velocity field $u$
its energy is $E(v)=\frac 12\int_M(u,u)\,\mu$ and it is right-invariant on the group 
of volume-preserving diffeomorphisms. 
Thus the Euler equation defines geodesics in the right-invariant $L^2$-metric on the group  $\SDiff(M)$.
 \end{remark}

\par

\subsection{Hamiltonian framework of the Euler equation}

The geodesic description implies  the following Hamiltonian framework for the Euler equation.  
Consider the (smooth) dual space $\mathfrak g^*={\SVect}^*(M)$ 
to the Lie algebra $\mathfrak g={\SVect}(M)$ of divergence-free vector fields on $M$ (tangent to the boundary).
This dual space  has a natural  description 
as the  space of cosets  $\mathfrak g^*=\Omega^1(M) / \diff \Omega^0(M)$, where $\Omega^k(M)$ is the space of $C^\infty$ $k$-forms on $M$. For a 1-form $\alpha$ on $M$ its coset of 1-forms is
$$
[\alpha]=\{\alpha+df\,|\,\text{ for all } f\in \Cont^\infty(M)\}\in \Omega^1(M) / \diff \Omega^0(M)\,.
$$
The pairing between cosets and divergence-free vector fields is straightforward:
$\langle [\alpha],w\rangle:=\int_M\alpha(w)\,\mu$ for any  field $w\in {\SVect}(M)$.
(Note that in the case of $M$ with boundary, no assumptions on the behaviour of 1-forms or function 
differentials on the boundary are imposed.)
The coadjoint action of the group $\SDiff(M)$ on the dual 
 $\mathfrak g^*$ is given by the change of coordinates in (cosets of) 1-forms on $M$ 
 by means of volume-preserving diffeomorphisms.
 
A Riemannian metric $(\,,)$ on the manifold $M$ allows one to
identify the Lie algebra and  (the smooth part of) its dual by means of the so-called inertia operator:
given a vector field $u$ on $M$  one defines the 1-form $\alpha=u^\flat$ 
as the pointwise inner product with  the velocity field $u$:
$u^\flat(w): = (u,w)$ for all $w\in T_xM$. Note also  that divergence-free fields $u$ correspond to co-closed 1-forms $u^\flat$.
The Euler equation \eqref{idealEuler} rewritten on 1-forms $\alpha=u^\flat$ is
$$
\partial_t \alpha+\L_u \alpha=-dP\,
$$
for  an appropriate function $P$ on $M$.
In terms of the cosets of 1-forms $[\alpha]$, the Euler equation on the dual space $\mathfrak g^*$ looks as follows: 
\begin{equation}\label{1-forms}
\partial_t [\alpha]+\L_u [\alpha]=0\,.
\end{equation}

It follows from the geodesic description that the  Euler equation \eqref{1-forms} on 
$\mathfrak g^*=\SVect^*(M)$ is a Hamiltonian equation with the
Hamiltonian functional  $\H$ given by the fluid's kinetic energy,
$$
{\H}\left([\alpha]\right)=E(u)= \frac 12\int_M(u,u)\,\mu
$$ 
for $\alpha=u^\flat$.
The corresponding Poisson structure is given by  the 
natural linear Lie-Poisson bracket on the dual space $\mathfrak g^*$ 
of the Lie algebra $\mathfrak g$, see details in \cite{Arn66, AK}.
The corresponding Hamiltonian operator is given by the Lie algebra coadjoint action ${\rm ad}^*_u$, 
which  in the case of the diffeomorphism group corresponds to the Lie derivative: ${\rm ad}^*_u=\L_u$.
Its symplectic leaves are coadjoint orbits of the corresponding group $\SDiff(M)$.

\begin{remark} 
According to the Euler equation  \eqref{1-forms} in any dimension  the 
coset of 1-forms $[\alpha]$ evolves by a volume-preserving change of coordinates, 
i.e. during the Euler evolution it remains in the same coadjoint orbit in $\mathfrak g^*$.
Introduce the {\it vorticity 2-form} $\omega:=\diff u^\flat$ as the differential of the 1-form 
$\alpha=u^\flat$ and note  that  the vorticity exact 2-form is well-defined for cosets $[\alpha]$: 
1-forms $\alpha$ in the same coset have equal vorticities $\omega=\diff \alpha$. 
The corresponding Euler equation assumes the vorticity (or Helmholtz) form
\begin{equation}\label{idealvorticity}
\partial_t \omega+\L_u \omega=0\,,
\end{equation}
which means that the vorticity form is transported by (or ``frozen into") the fluid flow (Kelvin's theorem).
The definition of vorticity $\omega$ as an exact 2-form $\omega=\diff u^\flat$ makes sense for a manifold $M$ 
of any dimension. (In 2D  the vorticity 2-form can be identified with the vorticity function, while in 3D 
it can be regarded as the vorticity vector field $\mathrm{curl}\,v$  by means of the relation 
$i_{\mathrm{curl}\,v} \mu=\omega$ for the volume form $\mu$.) 
\end{remark}

\begin{remark} \label{rem:types_vs}
The cases of singular vorticity are of particular importance. 
A vortex membrane (or a vortex filament in 3D) corresponds to a distributional 2-form 
(i.e., {\it  de Rham $(n-2)$-current}) $\omega=C\cdot \delta_N$ supported on a submanifold $N$ 
of codimension 2 in $M$, where the constant $C$ is called 
the strength of the membrane. A vortex sheet can be associated with  a  de Rham $(n-2)$-current
$\omega=df\wedge \delta_\Sheet$  supported on a hypersurface $\Sheet$ (of codimension 1) in $M$, 
where $df $ is an exact 1-form on  $\Sheet$.
Vortex sheets appear in a fluid with a jump discontinuity of the velocity along a hypersurface. 
In this paper we will be particularly concerned with the setting where $\Sheet\subset M$ is homologous to zero,
and it splits the manifold $M$ into two connected parts, 
$ \Dom^+_{\Sheet}\cup \Dom^-_{\Sheet}=M\,\setminus\, \Sheet$, being
 the boundary of each of them (with appropriate orientations). 
For the vorticity 2-form $\omega=\diff u^\flat$  supported on $\Sheet$, $\omega=df\wedge \delta_\Sheet$, the velocity field
 $u$ is irrotational (i.e. curl-free) outside $\Sheet$ and hence locally potential in $M\,\setminus\, \Sheet$. 
 Note that the difference of the velocities on different sides of $\Sheet$ corresponds to the 
 1-form $df$ with help of the metric: $(v^+-v^-)^\flat|_\Sheet=df$, see e.g. \cite{Kh12}.

 If the field $u$ is {\it globally potential} on $M\,\setminus\, \Sheet$, we call 
 the corresponding solution $u(t)$ a {\it pure vortex sheet solution}. 
 The motion of  vortex sheets for such solutions  is the main object of our consideration and it admits 
 a geodesic interpretation in terms of a metric on $\Sheet$'s.
 If  the field $u$ is only {\it locally potential} on $M\,\setminus\, \Sheet$,
 we call the corresponding evolution of velocity $u$ an {\it irrotational vortex sheet solution}. 
This evolution admits a Hamiltonian interpretation on the cotangent bundle to the space
of vortex sheets. Alternatively, it also admits a geodesic interpretation on a certain extension
of the space of $\Sheet$'s, which we will address in a future publication.
For {\it general solutions} $u$ with vortex sheets (and not necessarily vanishing curl) we give a Hamiltonian description, 
as well as the geodesic interpretation on the diffeomorphism groupoid, which however does not reduce to the 
geodesic interpretation on vortex sheets only, see Section~\ref{sect:dyn_vs}.
 \end{remark}

\medskip

\section{Spaces of densities and vortex sheets}

\subsection{Otto calculus on the space of densities}\label{sect:otto}
The Euler equation, being a geodesic equation on the group of volume-preserving diffeomorphisms,
is closely related (in a sense, dual) to
the theory of optimal mass transport, and in particular, to the problem 
of moving one mass (or density) to another while minimizing a certain cost. 
In this section we discuss the relation of metric properties
of the diffeomorphism group and the space of densities to show an analogy with the framework of vortex sheets 
in the next section.

Assume that $M$ is a compact $n$-dimensional Riemannian manifold without boundary and
 consider the group $\SDiff(M)$ of diffeomorphisms preserving the volume form $\mu$
as a subgroup in the group $\Diff(M)$ of all smooth diffeomorphisms of $M$. 
Define a (weak) Riemannian metric on the group $\Diff(M)$ in the following straightforward way: 
given $u,v\in \vect(M)$, the inner product of
two vectors $u\circ\phi,v\circ\phi\in T_\phi\Diff(M)$ 
at any point $\phi\in\Diff(M)$  is
\begin{equation}\label{dmetric}
\langle u\circ\phi,v\circ\phi\rangle_\Diff = \int_M (u,v)\,\phi_* \mu\,.
\end{equation}
This metric is right-invariant with respect to the action of the subgroup
$\SDiff(M)$ of volume-preserving diffeomorphisms,
although it is not right-invariant with respect to the action of the whole group $\Diff(M)$. 

Let $\mu$ be a smooth reference volume form (or density) 
on $M$ of unit total mass, 
and consider the projection $\pi : \Diff(M) \to \W(M)$ of diffeomorphisms 
onto the space $\W(M)$ of (normalized) smooth
densities on $M$. The  diffeomorphism group $\Diff(M)$
is fibered over $\W(M)$ by means of this projection
$\pi$ as follows: the fiber over a volume form $\tilde \mu$  consists of all 
diffeomorphisms $\phi$ that push $\mu$ to $\tilde \mu$,  $\phi_*\mu=\tilde \mu$. (Note that diffeomorphisms from $\Diff(M)$ act transitively on smooth normalized densities, 
according to Moser's theorem.) In other words, 
two diffeomorphisms $\phi_1$ and $\phi_2$ 
belong to the same fiber 
if and only if $\phi_1=\phi_2\circ\varphi$ for some 
diffeomorphism $\varphi$ preserving the volume form $\mu$.

\smallskip

\begin{remark}
{\rm 
The projection $\pi$ can be extended to Borel maps and densities
that are absolutely continuous with respect to the Lebesgue measure.
More precisely, let $\mu$ and $\tilde \mu$ be two measures of the same total volume, and let $\operatorname{dist}(x,y)$
be the (geodesic) distance function on~$M$.
Consider the following \index{mass transport problem} {\it optimal mass
transport problem}: Find a Borel map $\phi:M\to M$ that pushes
the  measure $\mu$ forward to $\tilde \mu$ and attains the minimum
of the $L^2$-cost functional 
$\int_M \operatorname{dist}^2(x,\phi(x))\mu$ among all such maps. 
The minimal cost of transport defines a metric (called the {\it Kantorovich} or 
{\it Wasserstein metric})\index{Kantorovich metric}\index{Wasserstein metric}
$\operatorname{Dist}$ on densities:
\begin{equation}\label{optimal}
{\operatorname{Dist}}^2(\mu, \tilde \mu)
:=\inf_\phi\Big\{\int_M \operatorname{dist}^2(x,\phi(x))\mu~
|~\phi_*\mu=\tilde \mu\Big\}\,.
\end{equation}

It turns out that this mass transport problem admits a unique solution 
(defined up to measure-zero sets), called the optimal map $\tilde\phi$, see, e.g.,  
\cite{Brenier91, Villani}. Furthermore, a 1-parameter family 
of maps $\phi(t)$ joining $\phi(0)=\id$ with  the 
 map $\phi(1)=\tilde\phi$, such that $\phi(t)$ 
pushes $\mu$ to $\mu(t):=\phi(t)_*\mu$
in an optimal way for every $t$, defines a geodesic $\mu(t)$ between  
$\mu$ and $\tilde \mu$ in the  space 
of densities with respect to the metric $\operatorname{Dist}$; 
see \cite{Otto, Villani} for details. 

Here we recall an infinitely smooth (formal) setting of this problem,  while the corresponding 
setting of Sobolev spaces can be found e.g. in \cite{KLMP}.
One can see that the Kantorovich metric $\operatorname{Dist}$  is  
generated by a (weak) Riemannian metric on the space $\W$ of smooth densities~\cite{BB}. Thus both $\Diff$ and $\W$ can be regarded as infinite-dimensional Riemannian 
manifolds.
}
\end{remark}

\begin{proposition}\label{rsub} {\rm{\cite{Otto}}}
The bundle map $\pi:\Diff(M)\to \W(M)$ is a
Riemannian submersion of the metric $\metric_\Diff$ 
on the diffeomorphism group 
$\Diff(M)$ to the metric $\operatorname{Dist}$ 
on the density space $\W(M)$.
The horizontal (i.e., normal to fibers) spaces in the bundle $\Diff(M) \to \W(M)$ 
consist of right-translated gradient fields.
\end{proposition}

Recall that for two Riemannian manifolds $P$ and $B$ a {\it submersion}
\index{submersion}
$\pi: P\to B$ is a smooth map which has a surjective differential and preserves 
lengths of horizontal tangent vectors to~$P$. 
For a bundle $P\to B$ this means that on $P$ there is a distribution 
of horizontal 
spaces orthogonal to fibers and projecting isometrically to the 
tangent spaces to $B$.
Geodesics on $B$ can be lifted to horizontal geodesics in $P$, and the lift is unique for a given initial point in~$P$.

\begin{remark}
In short, the proposition follows from the Hodge decomposition  $\Vect=\SVect\oplus_{L^2} \Grad$
for vector fields on $M$: any vector field $v$ decomposes uniquely into 
the sum $v=w+\nabla p$ of a divergence-free field $w$
and a gradient field $\nabla p$, which are $L^2$-orthogonal to each other:
$\int_M(w,\nabla p)\,\mu=0$. The vertical tangent space at the identity coincides with  $\SVect(M)$,
while the horizontal space is $\Grad(M)$. The vertical space (tangent to a fiber)
at a point $\phi\in \Diff(M)$ consists of $w\circ\phi$, divergence-free vector fields $w$
right-translated by the diffeomorphism $\phi$, while the horizontal space is given by the translated 
gradient fields, $(\nabla p)\circ\phi$.
The $L^2$-type metric $\metric_\Diff$ on  horizontal 
spaces for different points of the same
fiber projects isometrically to one and the same metric
on the base, due to the $\SDiff$-invariance of the metric. Now the proposition follows from the observation 
that the above metric ${\operatorname{Dist}}$ is Riemannian and generated by the $L^2$ metric 
on gradients, see \cite{BB}.
\end{remark}


One of the main properties of a Riemannian submersion is 
the following feature of geodesics:

\begin{corollary}\label{horizontal}
Any geodesic initially tangent to a horizontal space on the full 
diffeomorphism group $\Diff(M)$ remains horizontal, 
i.e.,  tangent to the gradient distribution on this group. 
There is a one-to-one correspondence between geodesics on the base $\W(M)$ 
starting at the density $\mu$ and horizontal geodesics in $\Diff(M)$ starting 
at the identity diffeomorphism $\id$.
\end{corollary}

\begin{remark}\label{rem:symp_red}
The Riemannian submersion property for the fibration $\Diff(M) \to \W(M)$
can be put in the framework of symplectic reduction.
For a principal bundle $\pi:P\to B$ with the structure group $G$ 
the symplectic reduction of the cotangent bundle $T^*P$ over the $G$-action
gives the cotangent bundle $T^*B=T^*P/\!/G$. If the bundle $P$ is equipped with a $G$-invariant Riemannian 
metric $\metric_P$  it induces the metric $\metric_B$
on the base $B$. The Riemannian submersion of $P$ 
to the base $B$, equipped with the metrics $\metric_P$
and $\metric_B$ respectively, is the result of 
the symplectic reduction with respect to the $G$-action. Thus 
Proposition \ref{rsub} can be viewed within the  framework of symplectic reduction  with  
$P=\Diff(M)$, $G=\SDiff(M)$, and $B=\W(M)$.
\end{remark}
\medskip


\subsection{Calculus on the space of vortex sheets}\label{sec:vscalculus}

Below we show that the above framework  has an analog for the space of vortex sheets.
We first consider a preliminary model of description of vortex sheets, emphasizing similarities with the 
description above, and adjust it later.

As we mentioned, vortex sheets correspond to a jump discontinuity of the fluid velocity along a hypersurface
in a flow domain. Below we assume that the  velocity is gradient in the complement to the hypersurface.
For a compact connected closed manifold $M$ equipped with a volume form $\mu$ let $\Sheet_0 \subset M$
be  a cooriented closed embedded (but not necessarily connected) smooth hypersurface splitting $M$ 
into two connected parts, $\Dom^+_{\Sheet_0} \!\cup \Sheet_0 \cup\Dom^-_{\Sheet_0} =M$. (The case of more connected components, 
as well as the case of $M$ with boundary, require only minor adjustments.)
Denote by $\VS(M)$  the space of  vortex sheets that are images of $\Sheet_0$ under volume-preserving 
diffeomorphisms of the ambient manifold $M$:
$\VS(M):=\{ \Sheet=\phi(\Sheet_0)~|~\phi\in\SDiff(M)\}$. 
Then by construction the group $\SDiff(M)$ is fibered over $\VS(M)$: the fiber $\F_{\Sheet_0}$
over $\Sheet_0\in \VS(M)$ is the subgroup 
$\SDiff(M)_{\Sheet_0}$ of volume-preserving diffeomorphisms of $M$ mapping $\Sheet_0 \subset M$ into itself, 
while the fiber $\F_{ \Sheet}$ over a hypersurface $ \Sheet$ 
consists of all such volume-preserving diffeomorphisms of $M$ that map $\Sheet_0$ 
to $ \Sheet$. In particular, the map $\SDiff(M)\to \VS(M)$ is surjective. 
The vertical tangent space (at the identity) is the Lie subalgebra of all divergence-free vector fields 
on $M$ tangent to $\Sheet_0$, while at any other point the vertical space consists of those very vector fields
right-translated by a volume-preserving diffeomorphism moving $\Sheet_0$ to $ \Sheet$.
The tangent space to $\VS(M)$ at $\Sheet$ can be regarded as the space of normal vector fields $v$ 
on $\Sheet$ with an additional ``zero-mean constraint":   the $(n-1)$-form $i_v\mu$ has 
zero integral over $\Sheet$, which corresponds to the volume conservation inside $\Sheet$ for any its 
infinitesimal motion.

\begin{definition}\label{def:baseMetric}
Define the following (weak) \textit{metric  on the space $\VS(M)$ of vortex sheets}.\footnote{Weakness of metric means that the  topology induced by the metric is weaker than the $H^k$-topology 
considered on sufficiently smooth vortex sheets.}
Let $v\in T_\Sheet\VS(M)$ be a vector field attached at a vortex sheet $\Sheet \subset M$ and normal to it. 
Then its square length is
\begin{equation}
\normsq{v}_{\vs}:=\inf\left\{\normsq{u}_{L^2} \mid \div u=0 \mbox{ and } (u,\nu)\,\nu = v  \mbox{ on } \Sheet\right\},
\end{equation}
where $u\in \SVect(M)$ is a smooth 
vector field in $M$,  and $\nu$ is the unit normal field to $\Sheet$.
\end{definition}

Recall that the group $\SDiff(M)$ is also endowed with a natural right-invariant $L^2$-metric. At the group identity
it is given by the $L^2$-inner product on divergence-free vector fields:
$\langle u,u\rangle_{L^2}=\int_M(u,u)\mu$, where $u$ is a vector field at $\id\in \SDiff(M)$.
For any diffeomorphism $\phi\in \SDiff(M)$ we set by right-invariance
$\langle u\circ\phi ,u\circ\phi\rangle_{L^2}:=\int_M(u\circ\phi (x),u\circ\phi (x))_{\phi(x)}\mu$, where we
 take into account that  $\phi^*\mu = \mu$ for a volume-preserving diffeomorphism $\phi$.
While vortex sheets cannot be described solely within the setting of diffeomorphism groups, as the motion of 
particles is discontinuous near $\Sheet$, their heuristic description can be given as follows.

\begin{hproposition}  The bundle projection 
$(\SDiff(M), \metric_{L^2})\to (\VS(M),  \metric_{\vs})$
is a Riemannian submersion.
\end{hproposition}

The fact that this is a Riemannian submersion is essentially built in Definition \ref{def:baseMetric} of the metric 
on the space of vortex sheets $\VS(M)$.  This heuristic statement implicitly means that vortex sheets are moved by smooth diffeomorphisms, which turns out to be not the case.
However, such a geometric viewpoint leads to interesting 
consequences: geodesics on the base $\VS(M)$ are obtained by projecting horizontal geodesics on
$\SDiff(M)$. Recall that all geodesics on $\SDiff(M)$  describe incompressible fluid motions
and are given by solutions of the Euler equation, see Remark \ref{rem:EulerDerivation}.
Assuming that for a given vector $v$ there is a ``horizontal" vector $u$ realizing the infimum in the definition 
above, i.e. a smooth divergence-free vector  field $u$ on $M$, having the prescribed normal component $v$ on 
a vortex sheet $\Sheet$ we come to the following corollary.

\begin{corollary}\label{LoeschThm} {\rm{\cite{Loesch}}}
Geodesics with respect to the metric $\metric_{\vs}$ on  the space  $\VS(M)$ 
describe the motion of vortex sheets in an incompressible flow.
\end{corollary}

However, the rigorous setting is more complicated. The reason is that, for connected $\Sheet$, there are no horizontal vectors for the  
bundle $\SDiff(M)\to \VS(M)$ in the smooth setting: no smooth divergence-free vector field is 
$L^2$-orthogonal to all divergence-free vector fields tangent to $\Sheet$. (Indeed, divergence-free fields are 
$L^2$-orthogonal to gradients off  $\Sheet$, so for smooth fields in the orthogonal complement 
their smooth potentials must be globally defined harmonic functions on $M$, i.e. 
constants. For disconnected $\Sheet$ there is at most a finite-dimensional space of horizontal vectors.) To describe vortex sheets we need to consider
vector fields admitting jumps on a hypersurface, and hence maps which are diffeomorphisms 
on the interior and exterior of the hypersurface. 
Furthermore, in the Hamiltonian setting the vorticity is singular in the case of vortex sheets, 
and should belong to a certain completion of a smooth dual space to the Lie algebra $\SVect(M)$.
While the geometric picture above gives a hint of the framework to follow, the group description is not valid
 for  vortex sheets and we need to consider a groupoid.

\begin{remark}
We will see that while general solutions with vortex sheets can be described as geodesics of 
a certain right-invariant metric on a groupoid, for special solutions this metric can be reduced (as a submersion)
to a much smaller space, the space of vortex sheets. Namely, {\it potential} vortex sheets solutions 
(cf. Remark \ref{rem:types_vs}) can be described as geodesics  in the above metric on the space $\VS(M)$.
For {\it irrotational} vortex sheets solutions we present the corresponding Hamiltonian formulation, see
Section \ref{sec:irrot}.

One should also mention that the metric $\metric_{\vs}$ on  the space  $\VS(M)$ 
can be used to define a natural hydrodynamical metric on the space of shapes (of the same volume) 
and hence the corresponding optimal transport problem, see Appendix \ref{app:shapes}. 
\end{remark}


\medskip

\section[Generalities on Lie groupoids and algebroids]{Generalities on Lie groupoids and Lie algebroids}
In this section we briefly recall basic facts about Lie groupoids and Lie algebroids (details can be found, e.g., in \cite{dufour2006poisson}). For the sake of exposition we assume that all objects in this section are finite-dimensional (except for subsection \ref{section:algsub}).

\label{section:algsub}
\subsection{Lie groupoids}\label{sect:groupoid}

\begin{definition}
A \textit{groupoid} $\G \rightrightarrows \B$ is a pair of sets, $\B$ (the set of \textit{objects}) and $\G$ (the set of \textit{arrows}), endowed with the following structures:
\begin{enumerate} \item There are two maps $\Src,\Trg \colon \G \to \B$, called the \textit{source} map and the \textit{target} map.
\item There is a partial binary operation $(g, h) \mapsto gh$ on $\G$, which is defined for all pairs $g,h \in \G$ such that $\Src(g) = \Trg(h)$, and has the following properties.

\begin{enumerate}
\item  The source of the product is the source of the arrow applied first: $
\Src(gh) = \Src(h)$, while the target of the product is the target of the arrow applied second: $\Trg(gh) = \Trg(g).$
\item Associativity: $g(hk) = (gh)k$ whenever any of these expressions is well-defined.
\item Identity: for each $x \in \B$, there is an element $\id_x \in \G$ such that $\Src(\id_x) = \Trg(\id_x) = x$ and for every $g \in \G$ one has $\id_{\Trg(g)} \cdot  g =  g\cdot \id_{\Src(g)} = g$.
\item Inverse: for each $g \in \G$, there is an element $g^{-1} \in \G$ such that $\Src(g^{-1}) = \Trg(g)$, $\Trg(g^{-1}) = \Src(g)$, 
$g^{-1}g = \id_{\Src(g)}$, and $gg^{-1} = \id_{\Trg(g)}.$
\end{enumerate}
\end{enumerate}
\end{definition}

In what follows, we often use the term \textit{groupoid} referring to the set of arrows $\G$. If $\G \rightrightarrows \B $ is a groupoid, we say that $\G$ is a groupoid over $\B$.

A groupoid $\G \rightrightarrows \B$ is called a \textit{Lie groupoid} if $\G, \B$ are manifolds, the source and target maps are submersions, and the maps $(g, h) \mapsto gh$, $x \mapsto \id_x$, and $g \mapsto g^{-1}$ are 
smooth. (The domain of the multiplication map is $\{(x,y) \in \G \times \G \mid \Src(x) = \Trg(y)\}$. The requirement that $\Src$ and $\Trg$ are submersions guarantees that this set is a submanifold of  $\G \times \G$, so smoothness of multiplication is well-defined.)

\begin{example}\label{groupoids}
\begin{enumerate}[label=(\alph*)] \item Any Lie group $G$ is a Lie groupoid over a point.
\item For any smooth manifold $\B$, the set $\G := \B \times \B$ is a Lie groupoid over $\B$, called \textit{the pair groupoid}. The source and the target are defined by $\Src(x,y) = x$, $\Trg(x,y) = y$, while the product is given by
$
(y,z)(x,y) := (x,z).
$
\item
Let $\B$ be a smooth manifold, and let $G$ be a Lie group acting on $\B$. Then one can define the so-called \textit{action} (or \textit{transformation}) Lie groupoid $G \ltimes \B  \rightrightarrows \B$. The points of $G \ltimes \B$ are triples $(x,y,g)$, where $x,y \in \B$, $g \in G$, and $gx= y$. The source map is given by $\Src(x,y,g) := x$, the target is $\Trg(x,y,g) := y$, and the multiplication is defined by
$
(y,z,h)(x,y,g) := (x,z,hg)\,.
$ 

\end{enumerate}
\end{example}
In the remaining part of this subsection we introduce several standard notions related to groupoids.
\begin{definition}
A groupoid $\G \rightrightarrows \B$ is called \textit{transitive} if for any $x, y\in \B$ there exists $g \in \G$ such that $\Src(g) = x$ and $\Trg(g) = y$. 
\end{definition}
\begin{example}
An action groupoid  $G \ltimes \B$ is transitive if and only if the $G$-action on $\B$ is transitive. For applications in this paper we will consider only transitive groupoids.

\end{example}
%
\begin{definition}
Let $\G \rightrightarrows \B$ be a groupoid.  Then the \textit{source fiber}  $\G_x$ of $\G$ corresponding to $x \in \B$ is the set  $\G_x := \{ g \in \G \mid \Src(g) = x\}$.
The \textit{isotropy group} $\G_x^x$ of $\G$ corresponding to $x \in \B$ is the set $\G_x^x := \{ g \in \G \mid \Src(g) = \Trg(g) = x\}$. (This is indeed  a group under the operation induced from the groupoid multiplication.)
\end{definition}
\begin{example}
For an action groupoid $G \ltimes \B$, any source fiber is canonically identified with the group $G$, while the isotropy group corresponding to $x \in \B$ is the isotropy subgroup (the stabilizer) of $x$ under the $G$-action.
\end{example}

\medskip


\subsection{Lie algebroids}\label{sect:algebroid}
The infinitesimal object associated with a Lie groupoid is a \textit{Lie algebroid}.
\begin{definition}
A \textit{Lie algebroid} $\A$ over a manifold $\B$ is a vector bundle $\A \to \B$ endowed with a Lie bracket $\LieBracket$ on $\Cont^\infty$-smooth sections and a vector bundle morphism $\# \colon \A \to \T \B$, called the \textit{anchor map}, such that for any two $\Cont^\infty$-sections $\zeta, \eta$ of $\A$ and any smooth function $f \in \Cont^\infty(\B)$, one has the following version of the Leibniz rule:
$$
[\zeta,f\eta] = f[\zeta,\eta] + ({\#\zeta} \cdot  f) \eta\,.
$$
\end{definition}
\begin{remark}
Here and in what follows $ {\#\zeta}\cdot  f$ stands for the derivative of the function $f$ along the vector 
field $\#\zeta$ on $B$.  It also follows from the Leibniz rule and the Jacobi identity for the bracket that 
the anchor map induces a Lie algebra homomorphism from sections of $\A$ to vector fields on $B$:
$
\#[\zeta,\eta] = [\#\zeta, \#\eta]$, where the bracket on the right-hand side is the standard Lie bracket of vector fields. 
\end{remark}
 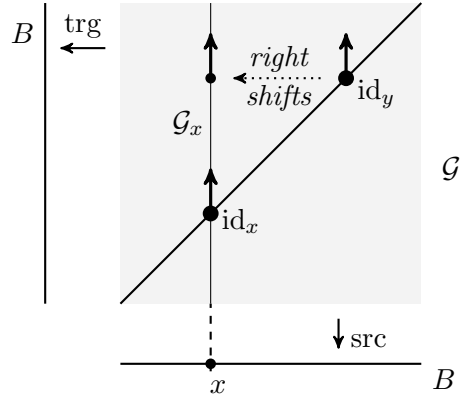
\begin{figure}[t]
\centerline{
\begin{tikzpicture}[thick, scale = 2]
     \fill [opacity = 0.05] (0,0) -- (0,2) -- (2,2) -- (2,0) -- cycle;
     \draw (2,-0.4) -- (0,-0.4);
     \draw (-0.5,0) -- (-0.5,2);
       \draw (0,0) -- (2,2);
            \fill (0.4+0.2,0.4+0.2) circle [radius=1.5pt];
 \draw  [thin] (0.4+0.2,0) -- (0.4+0.2,2);         
     \draw  [dashed] (0.4+0.2,-0.45) -- (0.4+0.2,-0);    
            \node  at (0.6+0.2,0.35+0.2) () {$\id_x$};
             \fill (1.3+0.2,1.3+0.2) circle [radius=1.5pt];
            \node  at (1.5+0.2,1.2+0.2) () {$\id_y$};
             \fill (0.4+0.2,1.3+0.2) circle [radius=1pt];
                 \fill (0.4+0.2,-0.4) circle [radius=1pt];
                      \node  at (0.45+0.2,-0.55) () {$x$};
                         \node  at (0.3+0.15,1.2) () {$\G_x$};
                             \node  at (2.2,0.9) () {$\G$};
 \draw  [dotted, <-] (0.55+0.2,1.3+0.2) -- (1.15+0.2,1.3+0.2);
            \node  at (0.85+0.2,1.4+0.23) () {\it right};
            \node  at (0.85+0.2,1.2+0.17) () {\it shifts};
 \draw  [very thick, ->] (0.4+0.2,0.4+0.2) -- (0.4+0.2,0.7+0.2);
 \draw  [very thick, ->] (0.4+0.2,1.3+0.2) -- (0.4+0.2,1.6+0.2);
 \draw  [very thick, ->] (1.3+0.2,1.3+0.2) -- (1.3+0.2,1.6+0.2);
  \draw  [->] (1.45,-0.1) -- (1.45,-0.3);
            \node  at (1.65,-0.25) () {$\rm src$};
    \draw  [->] (-0.1,1.7) -- (-0.4, 1.7);
            \node  at (-0.25,1.85) () {$\rm trg$};
            \node  at (2.15,-0.5) () {$B$};
            \node  at (-0.65,1.8) () {$B$};
\end{tikzpicture}
}
\caption{A transitive groupoid $\G \rightrightarrows \B$ is depicted as a square. The
vertical projection is the source map $\Src \colon \G \to B$, the  horizontal
projection is the target map  $\Trg \colon \G \to B$, while horizontal arrows are right translations. A section of the algebroid is a
collection of vertical vectors attached to the diagonal $\Src = \Trg$.
}
\label{fig:algebroid}
\end{figure}
\begin{definition}
\textit{The Lie algebroid $\A \to \B$ corresponding to a Lie groupoid $\G \rightrightarrows \B$} is constructed in the following way.
The fiber of $\A$ over $x \in B$ is the tangent space to the source fiber $\G_x$ at the point $\id_x$. The 
{\it anchor map} on this fiber is defined as the differential of the target map $\Trg \colon \G_x \to \B$, while the 
{\it bracket on sections} is defined as follows. Every section of $\A$ can be uniquely extended to 
a right-invariant vector field on $\G$ tangent to source fibers, and the correspondence between such vector 
fields and sections of $\A$ is a vector space isomorphism (see Figure \ref{fig:algebroid}). This allows one to define the bracket of sections 
of $\A$ as the Lie bracket of the corresponding right-invariant vector fields (which is again  
a right-invariant vector field tangent to source fibers, and, therefore, corresponds to a section of $\A$).
\end{definition}

\begin{remark}\label{transAlegbroid}
For a transitive groupoid $\G$, the Lie algebra of right-invariant vector fields tangent to source fibers is isomorphic via restriction to the Lie algebra of vector fields on a fixed source fiber $\G_x$ invariant with respect to the right action of the vertex group $\G^x_x$. (See Figure \ref{fig:algebroid}. In this figure, a vector field tangent to a source fiber $\G_x$ is represented as a collection of vertical vectors tangent to the vertical line $\Src = x$.) Therefore, in the transitive case one can also define the bracket in the algebroid $\A$ as the bracket of $\G_x^x$-invariant vector fields on $\G_x$, and $\A$ is isomorphic to the algebroid $\T\G_x \, / \, \G_x^x$.
\end{remark}

\begin{example}\label{algebroids}For Lie groupoids of Example \ref{groupoids}, the corresponding algebroids are:
\begin{enumerate}[label=(\alph*)] 
\item
The Lie algebra $\g$ of the group $G$, considered as a Lie algebroid over a point.
\item The tangent bundle $\T B$ of the manifold $B$. The corresponding bracket on sections is the standard Lie bracket of vector fields, while the anchor map is the identity.
\item
The \textit{action Lie algebroid} $\g \ltimes B$, where $\g$ is the Lie algebra of the group $G$. As a vector bundle, $\g \ltimes B$ is a trivial bundle over $B$ with fiber $\g$. The anchor map $\g \ltimes B \to \T B$ is defined for an element $(u,x) \in \g \ltimes B$ (where $u \in \g$ and $x \in \B$) by $\#(u, x) = \rho_u({x}) $, with $\rho_u$ being the infinitesimal generator of the $G$-action corresponding to $u \in \g$. The bracket of sections is given by \begin{equation}\label{actionBracket}
[\zeta,\eta](x) := [\zeta(x), \eta(x)] + {\#\zeta}(x) \cdot \eta- {\#\eta}(x) \cdot  \zeta\,,
\end{equation}
where the derivatives $\#\zeta(x) \cdot  \eta$, $\#\eta(x) \cdot  \zeta$ are defined by identifying sections of $\g \ltimes B $ with $\g$-valued functions on $B$. 
\end{enumerate}
\end{example}
In the remaining part of this subsection we introduce several standard notions related to Lie algebroids.
\begin{definition}
A {\it Lie algebroid} $\A \to \B$ is called {\it transitive} if the anchor map is surjective.
\end{definition}
One can prove that the Lie algebroid associated with a transitive Lie groupoid is transitive. (The converse is also true provided that the base $\B$ is connected.)

\medskip

Further, we define an \textit{isotropy algebra}.
Let $x \in B$, and let $ u,  v \in \Ker \#_x$ be elements of the kernel of the anchor map restricted to the fiber of $\A$ over $x$. Let also $\hat u, \hat v$ be arbitrary smooth sections of~$\A$ such that $\hat u(x) = u$ and $\hat v(x) = v$. Then one has the following result.
\begin{proposition}\label{prop:isoalgebra}
The value $[\hat u, \hat v](x)$ depends only on $u ,v$, but not on the choice of the extensions $\hat u, \hat v$. So, the formula $[u,v] := [\hat u, \hat v](x)$ gives a well-defined bracket on $\Ker \#_x$. The so-defined bracket turns the space $\Ker \#_x$  into a Lie algebra. Furthermore, if there is a Lie groupoid $\G$ associated with the algebroid $\A$, then this algebra is the Lie algebra of the isotropy group $\G_x^x$.
\end{proposition}
\begin{definition}
The algebra $\Ker \#_x$  is called the \textit{isotropy algebra} at the point $x$. 
\end{definition}

Further, we define the \textit{isotropy representation} of a Lie groupoid. Let $\G  \rightrightarrows \B$ be a Lie groupoid, and let $g \in \G$. Let also $x := \Src(g)$, $y := \Trg(g)$. Then we have a group homomorphism
$
\Phi_g \colon \G_{x}^{x} \to \G_{y}^{y}
$
given by
$
\Phi_g(h) := ghg^{-1},
$
and the corresponding homomorphism of Lie algebras $$\Ad_g \colon \Ker \#_{x} \to \Ker \#_{y}\,,$$ where $\#$ is the anchor map for the algebroid $\A$ corresponding to the groupoid $\G $.
The collection of operators $\{ \Ad_g \mid g \in \G\}$ defines a representation of the Lie groupoid $\G$ on the corresponding bundle of isotropy algebras $\{ \Ker \#_x \mid x \in \B\}$ (the latter is indeed a vector bundle provided that the Lie algebroid $\A$ is transitive, or, more generally, if the anchor map has constant rank). This representation is called the \textit{isotropy representation} of a Lie groupoid. It generalizes the notion of the adjoint representation of a Lie group.\par
Now, recall that one can use the adjoint representation of a group to define the bracket in the corresponding Lie algebra. Likewise, the isotropy representation of a Lie groupoid $\G$ can be used to define the bracket $[\zeta,\eta]$ of sections $\zeta, \eta$ of the corresponding Lie algebroid $\A$, provided that one has $\#\eta = 0$. Namely, one has the following result.

\begin{proposition}
Let $\G  \rightrightarrows \B$ be a Lie groupoid, and let  $\A \to \B$ be the corresponding algebroid. Let also $\zeta,\eta$ be sections of $\A$. Further, assume that $\eta$ belongs to the isotropy algebra at every point: 
$\#\eta= 0$.
Let $g_t$ be any smooth curve in a fixed source fiber $\G_x \subset \G$ such that $g_0 = \id_x$, and the tangent vector to $g_t$ at $\id_x$ is $\zeta(x)$.
Then
\begin{align}\label{isotropyRepr}
[\zeta,\eta](x) = \left.\frac{d}{dt}\right\vert_{t = 0}\!\!\!\! \Ad^{-1}_{g_t}( \eta(\Trg(g_t)))\,.
\end{align}
\end{proposition}
\begin{remark}
We have
$
\Ad^{-1}_{g_t} (\eta(\Trg(g_t))) \in \Ker \#_x
$
for every $t$, so the derivative in \eqref{isotropyRepr} is a well-defined element of $\Ker \#_x$.

When $\G$ is a Lie group $G$, the condition $\#\eta = 0$ becomes trivial, while formula~\eqref{isotropyRepr} becomes the relation $\ad = -\diff (\Ad)$ between the adjoint representations of $G$ and the adjoint representation of the corresponding Lie algebra $\g$. The minus sign is due to the fact that we have defined the Lie bracket using right-invariant vector fields instead of left-invariant ones. 
\end{remark}


\subsection{Lie algebroids and Poisson vector bundles}\label{sect:poisson_br}
Recall that the dual space $\g^*$ of any Lie algebra $\g$ carries a natural linear Poisson structure (by definition, a Poisson structure on a vector space is \textit{linear} if the Poisson bracket of two linear functions is again a linear function). 
Conversely, given a vector space $V$ with a linear Poisson structure, its dual space $V^*$ has a natural Lie algebra structure. This duality between Lie algebras and ``Poisson vector spaces'' extends to the vector bundles setting. The corresponding dual objects are {Lie algebroids} and \textit{Poisson vector bundles}.

\begin{definition}
A \textit{Poisson vector bundle} $ E \to B$ is a vector bundle whose total space $E$ is endowed with a fiberwise linear Poisson structure, that is a Poisson structure such that the bracket of any two fiberwise linear functions is again a fiberwise linear function.
\end{definition}
\begin{remark}
It also follows from this definition and the Leibniz rule for the Poisson bracket that the bracket of a fiberwise linear function with a fiberwise constant function is a fiberwise constant function, while the brackets of fiberwise constant functions vanish.
\end{remark}
\begin{example}
The two basic examples of Poisson vector bundles are a vector space endowed with a linear Poisson structure (which is a Poisson vector bundle over a point), and the cotangent bundle of a manifold. These Poisson vector bundles are dual to Lie algebroids of Examples \ref{algebroids}(a) and \ref{algebroids}(b) respectively. \end{example}
For general Lie algebroids, one has the following result.
\begin{proposition} 
The dual bundle $\A^* \to B$ of any Lie algebroid $\A \to B$ has a natural structure of a Poisson vector bundle. The Poisson structure on $\A^*$ is uniquely determined by requiring that for arbitrary fiberwise linear functions $\zeta ,\eta$ and an arbitrary fiberwise constant function~$f$, one has
\begin{equation}\label{DualPoissonBracket}
\{ \zeta,\eta\} := [\zeta,\eta]\,, \quad \{\zeta, f\} :={\#\zeta} \cdot f\,.
\end{equation}
Here we identify fiberwise linear functions on $\A^*$ with sections of $\A$, and fiberwise constant functions on $\A^*$ with functions on the base $B$.\par
Conversely, given a Poisson vector bundle, its dual has a natural Lie algebroid structure. The bracket of sections and the anchor are defined by the same formulas \eqref{DualPoissonBracket} understood as the definitions of the right-hand sides.
\end{proposition}
In what follows, we will need an explicit formula for the Poisson structure on the dual of a Lie algebroid. We first define the \textit{Lie algebroid differential}:
\begin{definition}
For an arbitrary $1$-form $\xi$ on $\A$ (i.e., a section of $\A^*$) its \textit{algebroid differential} $\diff_\A \xi$ is a $2$-form on $\A$ given on arbitrary $\zeta, \eta$ lying in one fiber of $\A$ by the formula
\begin{align}\label{algDiff}
\diff_{\A}\xi \left(\zeta,\eta\right) := -\xi([\hat \zeta,\hat \eta]) + {\#\zeta} \cdot  (\xi(\hat \eta)) - {\#\eta} \cdot  (\xi(\hat \zeta))\,,
\end{align}
where $\hat \zeta, \hat \eta$ are arbitrary smooth sections of $\A$ extending $\zeta, \eta$.
\end{definition}
\begin{example}
When $\A = \T B$ is the tangent bundle,  $\diff_\A$ is the de Rham differential. When $\A$ is a Lie algebra, considered as an algebroid over a point, $\diff_\A$ is the Chevalley-Eilenberg differential.
\end{example}

\begin{proposition}\label{poissonBracketFormula} {\rm\cite{boucetta2011}} Let $\A$ be a Lie algebroid. Then,   for any $ \xi \in \A^*$ 
and for any smooth functions $f, g \in \Cont^\infty(\A^*)$, one has
\begin{align}\label{PoissonExplicit}
\{f, g\}(\xi) = -\diff_{\A} \hat \xi(\diff_\xi^Ff, \diff_\xi^Fg\vphantom{ \hat \xi}) + {\#\diff_\xi^Ff} \cdot (g \circ \hat \xi)- {\#\diff_\xi^Fg} \cdot  (f \circ  \hat \xi)\,,
\end{align}
where $\hat \xi $ is an arbitrary section of $\A^*$ extending $\xi$, 
 and $\diff_\xi^Ff$, $\diff_\xi^Fg$ are fiberwise differentials of $f$ and $g$ at $\xi$ (i.e. differentials restricted to the tangent space of the fiber of $\xi \in \A^*$), regarded as elements of $\A$.
\end{proposition}

\begin{remark}
This formula can be used as a definition in the infinite-dimensional case. Although for general infinite-dimensional algebroids  it is not even clear why this expression is well-defined, 
 we prove it below by obtaining an explicit formula in the setting of diffeomorphism groupoids  and vortex sheets.
\end{remark}

\begin{remark}\label{PVBasRS}
For a Lie algebroid $\A$ of a Lie groupoid  $\G \rightrightarrows \B$, the Poisson structure on the dual bundle can also be defined as follows. Functions on $\A^*$ can be identified with right-invariant functions on the \textit{source-wise cotangent bundle} $ \T^*_{sw}\G:= \bigcup_{x \in \B} \T^* \G_x$. Such right-invariant functions form a Poisson subalgebra with respect to the canonical Poisson bracket on the cotangent bundle, which gives rise to a Poisson bracket on $\A^*$. In other words, the Poisson manifold $\A^*$ is obtained from the Poisson manifold $ \T^*_{sw}\G$ by means of Hamiltonian reduction with respect to the right $\G$-action on $ \T^*_{sw}\G$.\par
Furthermore, in the transitive case, $\A^*$ is simply the quotient of the cotangent bundle $\T^*\G_x$ of an arbitrary source fiber $\G_x$ by the Hamiltonian right action of the vertex group $\G_x^x$ (cf. Remark \ref{transAlegbroid}).
\end{remark}

\medskip
\subsection[Euler-Arnold equations on Lie algebroids]{Euler-Arnold equations on Lie algebroids}\label{section:algsub}
Let $\A \to B$ be a  (finite- or infinite-dimensional) Lie algebroid, and let $\I \colon \A \to \A^*$ be an invertible 
bundle map. ({In the infinite-dimensional case one needs to consider the smooth dual bundle $\A^*$,
similarly to consideration of smooth duals of infinite-dimensional Lie algebras, cf.~\cite{AK}. In the hydrodynamical 
setting we define this smooth dual in detail in Section \ref{sec:kinematics}.})
We call such $\I$ an \textit{inertia operator}. An inertia operator $\I$ defines a metric on $\A$ given by
$$
 \langle u,v\rangle_\A := \langle \I(u), v \rangle \,
$$
for any $u ,v$ in the same fiber of $\A$.
Since the inertia operator $\I$ is invertible, we also get a dual metric on $\A^*$:
$$
\langle \xi_1, \xi_2 \rangle_{\A^*} :=   \langle \I(\xi_1), \xi_2 \rangle = \langle  \I^{-1}(\xi_1), \I^{-1}(\xi_2) \rangle_\A
$$
for any $\xi_1, \xi_2$ in the same fiber of $\A^*$. 
Define also a function $\H \in \Cont^\infty( \A^* )$ by
\begin{equation}\label{Hamiltonian}
\H(\xi) := \frac{1}{2}\langle \xi, \xi \rangle_{\A^*}\quad \forall \,\xi \in \A^*.
\end{equation}

\begin{definition}
The Hamiltonian equation associated with the Poisson structure on $\A^*$ and
the function $\H$ is called the \textit{groupoid Euler-Arnold equation} corresponding to the metric~$\metric_\A$.
\end{definition}
\begin{example}
When $\A$ is a Lie algebra, we obtain the standard notion of an Euler-Arnold equation on a Lie algebra dual. When $\A = \T B$ is the tangent bundle of $B$, the Euler-Arnold equation is the geodesic equation for the metric $\metric_\A$. 
\end{example}
\begin{remark}\label{rem:idofmetrics}
In the case when the algebroid $\A$ is associated with a certain Lie groupoid $\G$, solutions of the Euler-Arnold equation can be interpreted as geodesics of a right-invariant source-wise (i.e. defined only for vectors tangent to source fibers) metric on $\G$. Indeed, the one-to-one correspondence between sections of $\A$ and right-invariant vector fields on $\G$ tangent to source fibers  gives rise to a one-to-one correspondence between metrics on $\A$ and source-wise right-invariant metrics on $\G$. So, given a metric $\metric_\A$ on $\A$, we can consider the geodesic flow of the corresponding source-wise right-invariant metric $\metric_\G$ on $\G$ . This geodesic flow can be considered as a dynamical system on the source-wise cotangent bundle $ \T_{sw}^*\G:= \bigcup_{x \in \B} \T^* \G_x$, and under the Hamiltonian reduction with respect to the right $\G$-action (see Remark \ref{PVBasRS}), the solutions of this system descend to solutions of the above-defined groupoid Euler-Arnold equation.
\par
Furthermore, in the transitive case we have a one-to-one correspondence between metrics on $\A$ and metrics on any source fiber $\G_x$ invariant under the right $\G_x^x$-action, while the groupoid Euler-Arnold flow on $\A^*$ can be viewed as the reduction of the geodesic flow on $\G_x$ by means of the right $\G_x^x$-action.
\end{remark}
Further, we show that an Euler-Arnold equation on a transitive algebroid $\A \to \B$ always gives rise to a certain geodesic flow on the base $\B$. Indeed, let $\A \to \B$ be a Lie algebroid.
Then, since the anchor map $\# \colon A \to \T \B$ is an algebroid morphism (i.e., it preserves the bracket 
and the anchor), the dual map $\#^* \colon \T^*\B \to \A^*$ is Poisson. ({In the infinite-dimensional 
case, one needs to define $\T^*\B$ in such a way that its image under $\#^*$ belongs to the regular 
dual $\A^*$.}) Note that if, moreover, the algebroid $\A$ is transitive, then $\#^*(\T^* \B)$ is a symplectic leaf 
in $\A^*$. Indeed, if $\A$ is transitive, then  the Poisson map $\#^*$ is injective, while  the image of 
an injective Poisson map of a symplectic manifold is always symplectic.

\begin{proposition}\label{prop:sub}{\rm (cf. Proposition \ref{rsub})}
Let $\A \to \B$ be a transitive Lie algebroid, and let $\metric_\A$ be a positive-definite metric on $\A$ 
for an invertible inertia operator $\I \colon \A \to \A^*$. Assume also\footnote{Note that this property is automatic in the finite-dimensional case.} that for this metric $\metric_\A$  there  is an orthogonal decomposition 
$\A = \Ker \# \oplus (\Ker \#)^\bot$.  Then the following is true:
\begin{enumerate}
\item The pullback of the groupoid Euler-Arnold flow corresponding to the metric $\metric_\A$ from the symplectic leaf $\#^*(\T^*\B)$ to $\T^*\B$ is the geodesic flow for a certain metric $\metric_\B$ on $\B$. Explicitly, for any $x \in \B$ and any $\zeta, \eta \in \T_x\B$, the metric $\metric_B$ reads \begin{align}\label{inducedMetric}
\langle \zeta, \eta \rangle_\B = \langle \#^{-1}(\zeta),  \#^{-1}(\eta)\rangle_\A\,,
\end{align}
where $\#^{-1} \colon \T\B \to (\Ker \#)^\bot$ is the inverse for the restriction of the anchor map to $(\Ker \#)^\bot$.
\item The anchor map $\#  \colon (\A, \metric_\A) \to (\T\B, \metric_\B)$ is a Riemannian submersion of vector bundles, meaning that its restriction 
$ \#\vert_{(\Ker \#)^\bot}  \colon ((\Ker \#)^\bot, \metric_\A) \to (\T\B, \metric_\B)$ is an isometry.
\item  Assume, in addition, that the algebroid $\A$ corresponds to a certain transitive groupoid $\G$. Then, for every $x \in \B$, the target mapping $\Trg \colon (\G_x, \metric_\G) \to (\B, \metric_\B)$ is a Riemannian submersion. (Here the metric $\metric_\G$ on $\G_x$ is defined using the identification between metrics on $\A$ and right-invariant source-wise metrics on $\G$, see Remark \ref{rem:idofmetrics}.)
\end{enumerate}
\end{proposition}

\begin{proof}
A straightforward computation shows that the metric on $\T^*\B$ dual to~\eqref{inducedMetric} 
is the pull-back of the metric $\metric$ on $\A^*$ by the map $\#^*$. But this means that the Hamiltonian of the geodesic flow for the metric $\metric_B$ is the pull-back of the Euler-Arnold Hamiltonian, and, since the mapping $\#^*$ is Poisson, the same is true for the flows. This proves the first statement. Further, the second statement follows directly from formula \eqref{inducedMetric}, while the third statement follows from the second one and right-invariance. Thus, the proposition is proved.
\end{proof}

\begin{remark}
Another approach to the proof is the following. The source fiber $\G_x$ is a total space of a principal $\G^x_x$ bundle over~$\B$, with $\G^x_x$ acting on the fiber $\G_x$ by multiplication on the right. 
So, the proof follows from the symplectic reduction  outlined in Remark~\ref{rem:symp_red} combined with Remark~\ref{rem:idofmetrics}.
\end{remark}

\begin{example}\label{groupoidOverDensities} 
Let $M$ be a Riemannian manifold. Consider the natural transitive action of its diffeomorphisms group $\Diff(M)$ on the space $\W(M)$ of densities on $M$ of unit total mass, and let $\Diff(M) \ltimes \W(M)$ be the corresponding action groupoid (see Example \ref{groupoids}(c)). Define a metric on the corresponding action algebroid $\vect(M) \ltimes \W(M)$ by setting
$$
\langle u, v \rangle_{L^2} := \int_M (u,v)\mu
$$
for $u, v$ lying in the fiber of $\vect(M) \ltimes \W(M)$ over $\mu \in \W(M)$. (Recall that the fibers of $\vect(M) \ltimes \W(M)$ are identified with the Lie algebra $\vect(M)$, see Example \ref{algebroids}(c).) Then,  according to Remark \ref{rem:idofmetrics}, for any $\mu \in \W(M)$, there is a corresponding metric on the source fiber $(\Diff(M) \ltimes \W(M))_\mu = \Diff(M)$ invariant under the right action of the isotropy group $(\Diff(M) \ltimes \W(M))_\mu^\mu =  \SDiff(M)=\{ \phi \in \Diff(M) \mid \phi^*\mu = \mu\}$. This metric is exactly~\eqref{dmetric}.  Thus, Proposition \ref{rsub} is a special case of Proposition \ref{prop:sub}.
\end{example}

\medskip

\section{Discontinuous calculus}\label{sec:dcalculus}
\subsection{Basic definitions and properties}\label{sec:dtf}
Let $M$ be a compact connected oriented manifold without boundary endowed with a volume form $\mu$, and let $\Sheet \subset M$ be a smooth embedded hypersurface splitting $M$ in two parts, $\Domplus$ and $\Dommin$. In this section we define various spaces of tensor fields discontinuous across $\Sheet$. All such tensor fields will be assumed to be of the form $\chiplus \xi^+ + \chiplus \xi^-$, where $\chi^\pm_\Sheet$ are characteristic functions of the domains $\Dom^\pm_\Sheet$, and the fields $\xi^\pm$ are $\Cont^\infty$-smooth in $\Dompm$ up to the boundary. We introduce the following spaces:
\begin{gather*}
\vphantom{\dforms(M,\Sheet) }\dcinfty(M, \Sheet) := \{ \chiplus f^+ + \chimin f^- \mid f^\pm \in \Cont^\infty(\Dom^\pm_\Sheet)\} \qquad \mbox{(\textit{discontinuous functions}),}\\
\vphantom{\dforms(M,\Sheet) }\dvect(M,\Sheet) := \{ \chiplus u^+ + \chimin u^- \mid u^\pm \in \Vect(\Dom^\pm_\Sheet)\} \qquad  \mbox{(\textit{discontinuous vector fields})},\\
\dforms(M,\Sheet) := \{ \chiplus \alpha^+ + \chimin \alpha^- \mid \alpha^\pm \in \Omega^1(\Dom^\pm_\Sheet)\} \qquad \mbox{(\textit{discontinuous $1$-forms})}.
\end{gather*}
We also introduce the subspace of \textit{divergence-free discontinuous vector fields}
\begin{align*}
\dsvect(M,\Sheet)  := \{ \chiplus u^+ + \chimin u^- \in \dvect(M,\Sheet)   \mid \div u^\pm = 0\,,\, u^+\vert_{\Sheet} - u^-\vert_{\Sheet} \mbox{ is tangent to } \Sheet\}
\end{align*}
and the subspace of \textit{exact discontinuous $1$-forms}
\begin{align*}
\dexactforms(M,\Sheet) := \{ \chiplus \diff f^+ + \chimin \diff f^- \mid f^+\vert_{\Sheet} = f^-\vert_{\Sheet}\}\,.
\end{align*}

 \begin{remark}
One can show that the space $\dexactforms(M,\Sheet)$ consists of precisely those $1$-forms 
$\alpha \in\dforms(M,\Sheet)$ which are exact as de Rham currents on $M$, i.e. those currents 
which have vanishing pairing with smooth closed $(n-1)$-forms on $M$. (Equivalently, 
this space is the $L^2$-closure of exact forms inside $\dforms(M, \Sheet)$ for any Riemannian metric on $M$.)
In particular, 
the $(n-1)$-current $\alpha =  \chiplus \diff f^+ + \chimin \diff f^-\!$, where $f^+\vert_{\Sheet} = f^-\vert_{\Sheet}$, is the de Rham differential of the $n$-current $f = \chiplus f^+ + \chimin f^-$.\par
Similarly, the space $\dsvect(M,\Sheet) $ consists of 
those fields which have vanishing pairing with smooth exact $1$-forms on $M$, i.e., those vector fields $v$ for which $i_v \mu$ is a closed $1$-current. (These are precisely those vector fields $v \in \dvect(M, \Sheet)$ which are divergence-free in a weak sense, see Corollary \ref{dense}.)
 \end{remark}
In the presence of a metric on $M$, we also introduce the dual versions of the above subspaces:
\begin{gather*}
\dccforms(M,\Sheet) := \{u^\flat \mid u \in \dsvect(M)_\Sheet\} = \{ \chiplus \alpha^+ + \chimin \alpha^- \mid \diff^* \alpha^\pm = 0\,,\,  \alpha^+(\nu) = \alpha^-(\nu)\} \,,\\
\dgrad(M,\Sheet) := \{\alpha^\sharp \mid \alpha \in \dexactforms(M,\Sheet) \}= \{ \chiplus \nabla f^+ + \chimin \nabla f^- \mid f^+\vert_{\Sheet} = f^-\vert_{\Sheet}\}\,.
\end{gather*}
Here $d^*$ is the adjoint of the de Rham $d$ operator, and $\nu$ is the unit normal field on $\Sheet$.\par
\smallskip
Further, we define differential operations on discontinuous fields in the following way. Let $\mathcal D$ be a differential operator on $M$ acting on tensor fields of certain type (functions, vector fields, or $1$-forms). Let also $\xi = \chiplus \xi^+ + \chimin \xi^-$ be such a tensor field discontinuous across $\Gamma$. 
Then we define the regularized version $\mathcal D^\Reg$ of the operator $\mathcal D$ by setting
$$\mathcal D^\Reg \xi := \chiplus \mathcal D\xi^+ + \chimin \mathcal D\xi^-\!.$$
For example, the regularized {differential} of a discontinuous form $\chiplus \alpha^+ + \chimin \alpha^- \in \dforms(M, \Sheet)$ is
$$
\diff^\Reg (\chiplus \alpha^+ + \chimin \alpha^-) := \chiplus \diff\alpha^+ + \chimin \diff\alpha^-\!.
$$
(Compare this with the actual differential
$
\diff (\chiplus \alpha^+ + \chimin \alpha^-) = \diff^\Reg (\chiplus \alpha^+ + \chimin \alpha^-) +\delta_\Sheet \wedge (\alpha^+ - \alpha^-)
$
containing the ``singular part" as well.)
 In a similar way, we define regularized differential operators taking several arguments.
For example, for discontinuous vector fields $\chiplus u^+ + \chimin u^-, \chiplus v^+ + \chimin v^- \in \dvect(M, \Sheet)$, we have
$$
[\chiplus u^+ + \chimin u^-\!,\, \chiplus v^+ + \chimin v^-]^\Reg := \chiplus [u^+, v^+] + \chimin [u^-, v^-]\,.
$$
(One should distinguish this bracket from the actual Lie bracket of discontinuous vector fields, which is a distribution.) 
\par
Similarly, for any tensor field $\xi_t = \chiplust \xi^+_t + \chimint \xi^-_t$ depending smoothly on a parameter $t$ and discontinuous across a $t$-dependent hypersurface~$\Sheet_t$ we define
$$
{\partial_t^\Reg \xi_t}{} := \chiplust {\partial_t \xi^+_t}{}  + \chimint{\partial_t \xi^-_t}{} .
$$

Below we will need the following two useful formulas for operations with discontinuous objects. For $f =   \chi^+_{\Sheet}f^+ +  \chi^-_{\Sheet}f^- \in \dcinfty(M, \Sheet)$, let  
 $jump(f) := f^+\vert_{\Sheet} - f^-\vert_\Sheet$.

\begin{lemma}\label{derIntDiscFnLemma}
Let $\Sheet_t \in \VS(M)$ be a family of hypersurfaces parametrized by $t \in \R$, and let $f_t \in  \dcinfty(M, \Sheet_t)$ be a smooth family of functions discontinuous across $\Sheet_t$. Then
\begin{equation}\label{derIntDiscFn}
\frac{d}{dt} \int_M f_t\mu \,=\, \int_M \frac{d^\Reg f_t}{dt}\mu \,+\, \int_{\Sheet_t} jump( f_t)\frac{d\Sheet_t}{dt}\,.
\end{equation}
\end{lemma}
\begin{remark}
Since $\Sheet_t$ is a family of unparametrized hypersurfaces, the derivative $d\Sheet_t / dt$ is a section of the normal bundle ${N}\Sheet := (\T M)\vert_{\Sheet} \, / \, \T \Sheet$ of $\Sheet$, which can be identified, by means of the volume form $\mu$, with the canonical bundle (see Lemma \ref{lemma:vstanspace} below). 
\end{remark}
\begin{proof}[Proof of Lemma \ref{derIntDiscFnLemma}] It follows from 
$$
\frac{d}{dt} \int_{D_{\Sheet_t}^\pm} f_t\mu \,=\, \int_{D_{\Sheet_t}^\pm}\frac{d f_t^\pm}{dt}\mu \,+ \,\int_{\partial D_{\Sheet_t}^\pm } f^\pm_t \frac{d}{dt}\left(\partial D_{\Sheet_t}^\pm\right),
$$
where $\partial D_{\Sheet_t}^\pm = \pm \Sheet_t$.
\end{proof}

\begin{lemma}\label{intDerLemma} Let $f \in \dcinfty(M, \Sheet)$, and let $v \in \dsvect(M, \Sheet)$. Then
$$
\int_M \left(\L^\Reg_v f\right)\mu = \int_\Sheet jump(f)\,i_v \mu\,.
$$
(Notice that since $v \in \dsvect(M, \Sheet)$, we have $(i_{v^+}\mu)\vert_\Sheet = (i_{v^-} \mu)\vert_\Sheet$, so the restriction of $i_v\mu$ to $\Sheet$ is well-defined).
\end{lemma}
\begin{proof}
The proof is achieved by writing the left-hand side as the sum of integrals over $\Dompm$ and applying the Stokes formula.
\end{proof}

\medskip
\subsection{Singular Hodge decomposition}
In this section we present the Hodge decomposition in the setting of discontinuous forms and vector fields.
 \begin{proposition}[(Singular Hodge decomposition)]\label{Hodge}
There exist orthogonal (with respect to the $L^2$-inner product) decompositions
\begin{equation}\label{SHD}
\dforms(M, \Sheet) =\dccforms(M,\Sheet) \,\oplus_{L^2} \, \dexactforms(M, \Sheet)\,
\end{equation}
and
\begin{equation}\label{SHD2}
 \dvect(M, \Sheet) =\dsvect(M,\Sheet) \,\oplus_{L^2} \,  \dgrad(M, \Sheet)\,.
\end{equation}
\end{proposition}

\begin{proof}
These decompositions are metric-dual to each other, so it suffices to prove \eqref{SHD2}. 
The orthogonality follows from a straightforward application of the Stokes formula and divergence-free condition. 
Indeed, for $v \in \dsvect(M,\Sheet)$ and $w = \chiplus \nabla f^+ + \chimin \nabla f^-  \in \dgrad(M, \Sheet)$
the pairing $\langle v,w \rangle_{L^2} $ reduces to 
$\int_{\Sheet}   (f^+  i_{v^+}\mu- f^-  i_{v^-}\mu)$, which vanishes due to relations $(i_{v^+}\mu)\vert_{\Sheet} =(i_{v^-}\mu)\vert_{\Sheet}$ and  $f^+ = f^-$ on $\Sheet$.

Now we need to show that the sum of $ \dsvect(M,\Sheet)$ and $\dgrad(M, \Sheet)$ is the whole space 
$\dvect(M, \Sheet)$. Let $u =  \chiplus u^+ +  \chimin u^- \in  \dvect(M, \Sheet) $. 
Using Hodge decomposition for manifolds with boundary, write $u^\pm$ as $u^\pm =    v^\pm  + \nabla f^\pm,$
 where the vector fields $ v^\pm \in \Vect(\Dom_\Sheet^\pm)$ are divergence-free and tangent to $\Sheet$, while $f^\pm \in \Cont^\infty(\Dom_\Sheet^\pm)$ are functions. Further, let  $g:= f^+\vert_\Sheet - f^-\vert_\Sheet \in \Cont^\infty(\Sheet)$. Then, by Theorem \ref{dlp} from Appendix \ref{app:layer}, there exist harmonic functions $h^\pm  \in \Cont^\infty(\Dompm)$ such that the normal derivatives of $h^+$ and $h^-$ at $\Sheet$ coincide, while $h^+\vert_{\Sheet} - h^-\vert_{\Sheet} = g$. (The function $h = \chiplus h^++ \chimin h^-$ is known as a \textit{double layer potential}.)
%
 Using these functions, write $u$ as
  $$
  u = (\chiplus ( v^+ + \nabla h^+ ) +  \chimin ( v^- + \nabla h^- )) + ( \chiplus ( \nabla f^+ - \nabla h^+ ) +  \chimin ( \nabla f^- - \nabla h^- ))\,.
  $$
Then the first bracket is in  $\dsvect(M,\Sheet)$, while the second bracket is in $ \dgrad(M,\Sheet)$. Thus, the proposition is proved.
\end{proof}
 \begin{corollary}\label{dense}\label{dgradgrad}
\begin{enumerate} \item The space $\dsvect(M, \Sheet)$ is precisely the $L^2$-closure of $\svect(M)$ inside $\dvect(M, \Sheet)$.
\item The space $\dgrad(M, \Sheet)$ is  the $L^2$-closure of  $\Grad(M)$  inside $\dvect(M, \Sheet)$. \end{enumerate}
\end{corollary}
\begin{proof} We prove $i)$ only, as the other one is similar.
We have $\overline{\svect(M)} = \Grad(M)^\bot$ (where the bar stands for the $L^2$-closure and  $\bot$ stands for the orthogonal complement in $L^2$), so the inclusion $\dsvect(M, \Sheet) \subset  \overline{\svect(M)} $ directly follows from Proposition~\ref{Hodge}. Therefore, it is suffices to show that   $ \overline{\svect(M)}  \cap \dvect(M, \Sheet) \subset \dsvect(M, \Sheet)$. 
To that end, we first note that $\overline{\Grad(M)} = \svect(M)^\bot$, so by Proposition~\ref{Hodge} we have  $\overline{\Grad(M)} \supset \dgrad(M, \Sheet)$. Therefore, given $u \in  \overline{\svect(M)} =  \Grad(M)^\bot$, we have that  $u  \in \Grad(M, \Sheet)^\bot$. So, if $u$ is also in $\dvect(M, \Sheet)$, then $u \in \dsvect(M, \Sheet)$ by Proposition~\ref{Hodge}, as desired.
\end{proof}

\medskip
\section{Kinematics of vortex sheets}\label{sec:kinematics}
In this section, $M$ is a compact connected oriented manifold without boundary endowed with a volume form $\mu$. 

\subsection{The Lie groupoid of discontinuous diffeomorphisms}\label{sect:diffeo_pair}
In this subsection, we define the Lie groupoid $\DSDiff(M)$ of volume-preserving diffeomorphisms discontinuous 
along a hypersurface.
 This groupoid (or, more precisely, any of its source fibers) can be viewed as the configuration space of a fluid with an immersed vortex sheet.\par
The base of the groupoid $\DSDiff(M)$ is, by definition, the space $\VS(M)$ of  vortex sheets~$\Sheet$ confining  diffeomorphic
domains of the same total volume (see Section \ref{sec:vscalculus}).
The elements of  $\DSDiff(M)$ are volume-preserving diffeomorphisms discontinuous along a hypersurface, i.e. quadruples $(\Sheet_1, \Sheet_2, \phi^+, \phi^-)$ where $\Sheet_1, \Sheet_2 \in \VS(M)$, and $\phi^\pm \colon \Dom^\pm_{\Sheet_1} \to \Dom^\pm_{\Sheet_2}$ are volume preserving diffeomorphisms. Here, 
as above, $\Dom_{\Sheet_i}^+, \Dom_{\Sheet_i}^-$ are connected  components of $M\, \setminus \, \Sheet_i$. 
The source and the target of $ (\Sheet_1, \Sheet_2, \phi^+, \phi^-)$ are, by definition, $\Sheet_1$ 
and $\Sheet_2$ respectively. The multiplication in $\DSDiff(M)$ is given by composition of 
discontinuous diffeomorphisms (see Figure \ref{fig:composition}):
$$
(\Sheet_2, \Sheet_3, \psi^+, \psi^-)(\Sheet_1, \Sheet_2, \phi^+, \phi^-) := (\Sheet_1, \Sheet_3, \psi^+\phi^+, \psi^-\phi^-)\,.
$$

 
 \begin{proposition}\label{frechet}
 	$\DSDiff(M) \rightrightarrows \VS(M)$  is a transitive Lie-Fr\'echet groupoid.
 \end{proposition}
 \begin{proof}
Verification of groupoid axioms is straightforward, while transitivity follows from the definition of $\VS(M)$: for any $\Sheet, \tilde \Sheet \in \VS(M)$ there exists a \textit{smooth} volume-preserving diffeomorphism $\phi \in \SDiff(M)$ taking $\Sheet$ to $ \tilde \Sheet$. This diffeomorphism can be regarded as an element $ (\Sheet, \tilde \Sheet, \phi\vert_{\high \Domplus},  \phi\vert_{\high \Dommin})\in \DSDiff(M)$ with source $\Sheet$ and target $\tilde \Sheet$.\par
To prove the Lie property, we first show that  $\VS(M)$  is a Fr\'echet manifold. It is well known that the ambient space $\Sigma(M)$ of all hypersurfaces in $M$ (with no restriction on the volume) is a Fr\'echet manifold, see, e.g., \cite{Hamilton}, Example 4.1.7. The corresponding local charts are constructed by identifying surfaces close to a given $\Sheet \subset \Sigma(M)$ with sections of the normal bundle $N\Sheet$ of $\Sheet$ in the vicinity of the zero section. 
In the presence of a volume form on $M$, those sections can be identified with top-degree forms on $\Sheet$, while the identification between the neighborhood of $\Sheet$ and the neighborhood of the zero section in $N\Sheet$ can be made in such a way that $\Sheet, \tilde \Sheet \subset \Sigma(M)$ bound the same volume if and only if the top-degree form on $\Sheet$ representing $\tilde \Sheet$ has zero integral over $\Sheet$. For such a choice of charts on $\Sigma(M)$, its subset $\VS(M)$ is locally described as a closed subspace of top-degree forms on $\Sheet$ with zero mean. Therefore, $\VS(M)$ is a Fr\'echet submanifold of $\Sigma(M)$.\par
Further, we show that $\DSDiff(M)$ can be endowed with a Fr\'echet manifold structure in such a way that the source and target map are submersions. (Recall that a map between Fr\'echet manifolds is called a \textit{submersion} if it can be locally represented as the projection of a direct product to one of the factors.) Moreover, the manifold structure we define has the property that the mapping $\pi \colon \DSDiff(M) \to \VS(M) \times \VS(M)$ given by $\phi \mapsto (\Src(\phi), \Trg(\phi))$ is a fiber bundle. To that end, choose a reference vortex sheet $\Sheet_0 \in \VS(M)$. Then the vertex group $ \pi^{-1}(\Sheet_0, \Sheet_0)$ is the direct product $\SDiff(\Dom_{\Sheet_0}^+) \times \SDiff(\Dom_{\Sheet_0}^+)$. Each of the factors is a Lie-Fr\'echet group (see \cite{EbMars70}), so $ \pi^{-1}(\Sheet_0, \Sheet_0)$ is a Lie-Fr\'echet group as well. Further, for any $\Sheet \in \VS(M)$, there exists a (non-canonical) smooth volume-preserving diffeomorphism $\phi_\Sheet \in \SDiff(M)$ such that $\phi_\Sheet(\Sheet_0) = \Sheet$. This gives rise to bijections
\begin{gather}\Phi_{\Sheet_1, \Sheet_2} \colon \pi^{-1}(\Sheet_1, \Sheet_2) \to \pi^{-1}(\Sheet_0, \Sheet_0)\\
\Phi_{\Sheet_1, \Sheet_2}(\psi) := \phi_{\Sheet_2}^{-1} \circ \psi \circ \phi_{\Sheet_1} \end{gather}
between fibers of $\pi$.
Furthermore, for any $\Sheet \in \VS(M)$ there exists its neighborhood $O(\Sheet)$ in which the mapping $\phi_\star \colon O(\Sheet) \to \SDiff(M)$ given by $\tilde \Sheet \mapsto \phi_{\tilde \Sheet}$ can be chosen to be smooth (cf.~\cite{Hamilton}, Example 4.3.6). Then the corresponding maps $\Phi_{\Sheet_1, \Sheet_2}$ provide local identifications $$\pi^{-1}(O(\Sheet_1) \times O(\Sheet_2)) \simeq O(\Sheet_1) \times O(\Sheet_2) \times \pi^{-1}(\Sheet_0, \Sheet_0)\,.$$
Covering $\DSDiff(M)$ by open sets of the form $O(\Sheet_1) \times O(\Sheet_2) \times \pi^{-1}(\Sheet_0, \Sheet_0)$, we endow it with a Fr\'echet manifold structure such that $\pi \colon \DSDiff(M) \to \VS(M) \times \VS(M)$ is a fiber bundle. (Note that the transition maps are given by group operations in $\pi^{-1}(\Sheet_0, \Sheet_0)$ and hence smooth.) It follows that the source and target maps are fiber bundles as well, and hence submersions.\par
Finally, notice that with our local trivializations the groupoid operations (multiplication and inversion) in $\DSDiff(M)$ boil down to group operations in $\pi^{-1}(\Sheet_0, \Sheet_0)$, while the unit map $\Sheet \mapsto \id_\Sheet$ reads $\id_\Sheet = (\Sheet, \Sheet, \id) \in O(\Sheet) \times O(\Sheet) \times \pi^{-1}(\Sheet_0, \Sheet_0)$. Therefore, $\DSDiff(M) \rightrightarrows \VS(M)$ is indeed a Lie-Fr\'echet groupoid, as desired. 
 \end{proof}
  \begin{remark}One can also consider the groupoid $\DSDiff(M)$ in the category of Hilbert manifolds modeled on Sobolev $H^s$ spaces for sufficiently large $s$, 
similarly to, e.g., \cite{EbMars70} or Remark 3.3 in~\cite{KLMP}. 
\end{remark}

\begin{remark}
Since $\DSDiff(M)$ is a Lie-Fr\'echet groupoid, it follows that the corresponding algebroid is well-defined as a Fr\'echet vector bundle over $\VS(M)$ with a bracket and anchor on smooth sections. We describe this algebroid in detail in the next section.
\end{remark}
 
 \begin{remark}\label{subgroupoid}
 Note that if we impose an additional requirement that the maps $\phi^\pm$ in the definition of the groupoid $\DSDiff(M)$ are restrictions of  one 
 and the same smooth volume-preserving diffeomorphism $\phi \in \SDiff(M)$, then we obtain the definition 
 of the action groupoid $ \SDiff(M) \ltimes \VS(M)$ (see Example~\ref{groupoids}(c)) corresponding to the natural 
 action of $\SDiff(M)$ on $\VS(M)$. So, the groupoid $\DSDiff(M)$ comes with an action subgroupoid 
 $ \SDiff(M) \ltimes \VS(M)$ of smooth volume-preserving diffeomorphisms. We will see below that the 
 groupoid $\DSDiff(M)$ inherits some properties of the action groupoid $  \SDiff(M) \ltimes \VS(M)$. 
 In particular, the brackets in the algebroids corresponding to these groupoids look similarly.
 \end{remark}

 \medskip
 
 \subsection{The Lie algebroid of discontinuous vector fields}\label{sect:algebroid_vf}
 In this subsection we describe the Lie algebroid
$\dsvect(M) \to \VS(M)$
corresponding to the Lie groupoid $\DSDiff(M)$. This algebroid serves as the space of velocities for a fluid with an immersed vortex sheet.

\begin{theorem}\label{algFibers}
The fibration $ \dsvect(M) \to  \VS(M)$ can be equipped with the structure of a Lie algebroid 
corresponding to the groupoid $\DSDiff(M)$ as follows:
\begin{enumerate}\item
The fiber of $\dsvect(M)$ over $\Sheet \in \VS(M)$ is the space $\dsvect(M, \Sheet)$ which consists
 of discontinuous vector fields on $M$ of the form 
$u = \chiplus u^+ + \chimin u^-$, where 
 $u^\pm \in \SVect(\Dompm)$ are such that $u^+\vert_{\Sheet} - u^-\vert_{\Sheet}$ is tangent to $\Sheet$ (in other words, the normal component of $u$ on $\Sheet$ is continuous, see Section \ref{sec:dtf}). 
 \item  The anchor map  $
 \# \colon \dsvect(M, \Sheet) \to \T_\Sheet \VS(M)\,
 $ is given by the projection of $u^+\vert_{\Sheet}$ or, equivalently, $u^-\vert_{\Sheet}$ to the normal bundle ${N}\Sheet := (\T M)\vert_{\Sheet} \, / \, \T \Sheet$  (these projections coincide by the previous statement).
 \item Let $U, V$ be sections of $\dsvect(M)$.
 Then their algebroid bracket is
 \begin{equation}\label{sectionsBracket}
[U,V](\Sheet) = [U(\Sheet), V(\Sheet)]^\Reg +{\#U(\Sheet)}\cdot V - {\#V(\Sheet)}\cdot U\,.
\end{equation}

 \end{enumerate}
\end{theorem}
\begin{remark}\label{connectionInExtBundle} The derivative ${\#U(\Sheet)}\cdot V$ is a (discontinuous) vector field defined by
$$
{\#U(\Sheet)}\cdot V :=\left.\frac{ \hphantom{a}d^\Reg}{d t}\right\vert_{t = 0}\!\!\!\!\!{ V(\Sheet_t)}\,,
$$
where $\Sheet_t$ is any smooth curve in $\VS(M)$ with $\Sheet_0 = \Sheet$ and the tangent vector at $\Sheet$ 
given by ${\#U(\Sheet)}$. 
The derivative $\#U(\Sheet)\cdot V$ does not have to lie in $\dsvect(M, \Sheet)$, but belongs to the bigger
space  
$$
\dsvectext(M,\Sheet)  := \{   \chiplus u^+ + \chimin u^- \mid u^\pm \in \SVect(\Dom^\pm_\Sheet)\}
$$
of divergence-free vector fields with no conditions on $\Sheet$.
(The normal component of $\#U(\Sheet)\cdot V$ at $\Sheet$ does not have to be continuous.)
 The so-defined derivative 
can be viewed as a covariant derivative $\nabla_{\#U} V$ in the extended bundle $\dsvectext(M)  \to \VS(M)$ whose fiber over $\Sheet \in \VS(M)$ is the space $\dsvectext(M, \Sheet)$, cf.~\cite{Hamilton}, Example 4.5.4.
(We use the notation $\#U(\Sheet)\cdot V $ to emphasize its  relation to the formula from Example \ref{algebroids}(c).)

\end{remark}
\begin{remark}\label{remark:compensation}
Note that the first term on the right-hand side of \eqref{sectionsBracket} is not an element of $\dsvect(M, \Sheet)$. Indeed, for two vector fields whose normal components are continuous across $\Gamma$, their (regularized) Lie bracket $[\,,]^\Reg$ does not necessarily have this property. However, the last two terms also have discontinuous normal components that 
turn out to compensate the discontinuity of the first term, as we will see below.
  \end{remark}


\begin{remark}\label{subAlgBracket}
The Lie algebroid $\dsvect(M)$ contains a subalgebroid $ \svect(M) \ltimes \VS(M)$ of smooth divergence-free vector fields, corresponding to the subgroupoid $\SDiff(M) \ltimes \VS(M)$ of smooth volume-preserving diffeomorphisms (see Remark \ref{subgroupoid}). Note that since  $\SDiff(M) \ltimes \VS(M)$ is an action groupoid,  $\svect(M) \ltimes \VS(M)$ is an action algebroid, and the corresponding bracket automatically has form~\eqref{sectionsBracket} (cf. Example \ref{algebroids}(c)). The non-trivial part of the third statement of Theorem~\ref{algFibers} is that the bracket has the same form \textit{on the whole} Lie algebroid $\dsvect(M)$, 
despite the fact that $\dsvect(M)$ is \textit{not} an action algebroid. (In particular, the fibers of  $\dsvect(M)$ are not closed under the Lie bracket of vector fields (see Remark \ref{remark:compensation}) and hence do not have any natural Lie algebra structure.)
\end{remark}

\proof[Proof of Theorem~\ref{algFibers}]
We begin with the first statement. By definition, the fiber  of $\dsvect(M)$ over $\Sheet$ consists of tangent vectors at $\units_\Sheet \in \DSDiff(M)$ to curves of the form $(\Sheet, \Sheet_t, \phi_t^+, \phi_t^- )$, where $\Sheet_0 = \Sheet$ and $\phi_0^\pm = \id$. 
The  tangent vector to such a curve is a pair of divergence-free vector fields
$$
u^\pm =\left.\frac{ \hphantom{}d}{d t}\right\vert_{t = 0}\!\!\!\!\! \phi^\pm_t \,\,\in \,\SVect(\Dompm)\,.
$$
Also note that, by definition, both $\phi_t^+$ and $\phi_t^-$ map the surface $\Sheet$ to the same surface $\Sheet_t$. In other words, we have
\begin{equation}\label{boundaryRelation}
\phi_t^+\vert_{\Sheet} = \phi_t^-\vert_{\Sheet} \circ \psi_t\,,
\end{equation}
where $\psi_t \in \Diff(\Sheet)$ is a diffeomorphism of the sheet $\Sheet$, and $\psi_0 = \id$. Differentiating \eqref{boundaryRelation} with respect to $t$ at $t = 0$, we get
$$
u^+\vert_{\Sheet} = u^-\vert_{\Sheet} +\left.\frac{ \hphantom{}d}{d t}\right\vert_{t = 0}\!\!\!\!\! \psi_t\,,
$$
which means that $u^+\vert_{\Sheet} - u^-\vert_{\Sheet}$ is tangent to $\Sheet$, as desired.
\par
Conversely, given any pair of divergence-free vector fields $u^\pm \in \SVect(\Dompm)$ such that $u^+\vert_{\Sheet} - u^-\vert_{\Sheet}$  is tangent to $\Sheet$, one can construct a curve $\phi_t$ in the source fiber $\DSDiff(M)_\Sheet$ whose tangent vector at $\units_\Sheet$ coincides with $\chiplus u^+ + \chimin u^-$. So, the fiber of $\dsvect(M)$ over $\VS(M)$ is indeed the space $\dsvect(M, \Sheet)$.

To prove the second statement we need the following. 

\begin{lemma}\label{lemma:vstanspace}The tangent space  $\T_\Sheet \VS(M)$ is the space of sections of the normal bundle ${N}\Sheet$ (equivalently, the space of top-degree forms on $\Sheet$) having zero mean.  \end{lemma}
The proof of this lemma is achieved by differentiating a family $\Sheet_t$ at $t=0$. Its motion is determined by the normal 
vector field, while the zero mean condition follows from the conservation of volumes of  $\Dompm$ (cf. \cite{Hamilton}, Example 4.5.5). The equivalence between sections of $N\Sheet$ and top-degree forms of $\Sheet$ is provided by the map $v \mapsto (i_v \mu)\vert_\Sheet$.

 Now, we compute the anchor map. Let $u =  \chiplus u^+ + \chimin u^- \in \dsvect(M)_\Sheet$. Consider a curve $\phi_t = (\Sheet, \Sheet_t, \phi_t^+, \phi_t^- )$ whose tangent vector at $\units_\Sheet$ is $u$. Then, by definition of the anchor map for the algebroid of a Lie groupoid, we have 
 $$
 \#u = \left.\frac{ \hphantom{}d}{d t}\right\vert_{t = 0}\!\!\!\!\! \Trg(\phi_t) =\left.\frac{ \hphantom{}d}{d t}\right\vert_{t = 0}\!\!\!\!\! \Sheet_t \,. 
 $$
 Note that, as an unparametrized surface, $\Sheet_t$ coincides with $\phi_t^+(\Sheet)$, so the latter derivative is equal to the normal component of $$
\left.\frac{ \hphantom{}d}{d t}\right\vert_{t = 0}\!\!\!\!\! \phi_t^+(\Sheet) = u^+\!, 
 $$
 which means that $\#u$ is exactly the normal component of $u^+$ (equivalently, $u^-$), as desired.\par
Finally, we prove the third statement. First, notice that any section $U$ of $\dsvect(M)$ can be written (nonuniquely) as $U = U_{sm} + U_{tan}$, 
where $U_{sm}$ is a section of $\SVect(M) \ltimes \VS(M)$ (i.e., for every $\Sheet \in \VS(M)$ 
the vector field $U_{sm}(\Sheet)$ is divergence-free  and smooth), while $U_{tan}$ satisfies $\#U_{tan} = 0$ 
(i.e., for every $\Sheet \in \VS(M)$ the vector field $U_{tan}(\Sheet)$ is tangent to $\Sheet$).
(The existence of $U_{sm}$ is equivalent to the ability to trace any infinitesimal motion of $\Sheet$ by a smooth
divergence-free field in $M$.) 
Therefore, since both the left- and the right-hand sides of  \eqref{sectionsBracket} are skew-symmetric and additive in $U$ and $V$, it suffices to prove  this formula in the following two cases:
\begin{enumerate}\item $U, V$ are sections of $\SVect(M) \ltimes \VS(M)$; 
\item $\#V = 0$. \end{enumerate}
The first case follows from  $\SVect(M) \ltimes \VS(M)$ being an action algebroid (see Remark \ref{subAlgBracket}), so we only need to consider the case $\#V = 0$. To deal with this case, we use the following lemma.
\begin{lemma}\label{AdIsPushForward}
For the Lie groupoid $\DSDiff(M)$, the adjoint operator $\Ad_\phi \colon \Ker \#_{\Src(\phi)} \to  \Ker \#_{\Trg(\phi)}$ is given by pushforward:
$
\Ad_\phi u = \phi_*u\,.
$
\end{lemma}
\begin{proof}
The proof follows the lines of that for the corresponding statement about the group $\SDiff(M)$, see, e.g., 
Theorem 3.11 of \cite{AK}, as the adjoint action is a volume-preserving change of coordinates on $\Dompm$.
\end{proof}

Now, let $\phi_t$ be any smooth curve in the source fiber $\DSDiff(M)_\Sheet$ such that $\phi_0 = \units_\Sheet$, and the tangent vector to $\phi_t$ at $\units_\Sheet$ is $U(\Sheet)$. Then, using formula~\eqref{isotropyRepr} and Lemma~\ref{AdIsPushForward}, we have
\begin{align*}
[U,V](\Sheet) =\left.\frac{ \hphantom{}d}{d t}\right\vert_{t = 0}\!\!\!\!\! \phi_t^*\left(V(\Trg(\phi_t))\right)   &= [\left.\frac{ \hphantom{}d}{d t}\right\vert_{t = 0}\!\!\!\!\! \phi_t\,, V(\Trg(\phi_0))]^\Reg +  \left.\frac{ \hphantom{a}d^\Reg}{d t}\right\vert_{t = 0}\!\!\!\!\! V(\Trg(\phi_t))  \\
&= [U(\Sheet), V(\Sheet)]^\Reg + {\#U(\Sheet)}\cdot V\,,
\end{align*}
which, due to the condition $\#V = 0$, is equivalent to~\eqref{sectionsBracket}. Thus, the theorem is proved.
\proofend

%

 \begin{remark}\label{rem:normcot}
 In what follows, we identify the normal bundle $N\Sheet$ with the canonical bundle $\Lambda^{n-1}(\T^*\Sheet)$, where $n:=\dim M$ (i.e. $n-1:=\dim \Sheet$). With this identification, the tangent space $ \T_\Sheet \VS(M)$ is the space $\Omega^{n-1}_0(\Sheet)$ of top-degree forms on $\Sheet$ with zero mean, while the anchor map is given by
 $$
 \#u= (i_u \mu)\vert_{\Sheet} =  (i_{u^+} \mu)\vert_{\Sheet} =  (i_{u^-} \mu)\vert_{\Sheet} \quad \forall \, u \in \dsvect(M, \Sheet)\,.
 $$
 \end{remark}
 
 \begin{corollary}\label{cor:isotropy}
 The isotropy algebra $\Ker \#_\Sheet$ for the Lie algebroid $\dsvect(M)$ consists of vector fields of the form $u =  \chiplus u^+ + \chimin u^-$, where $u^\pm \in \SVect(\Dompm)$ are tangent to $\Sheet$. The Lie bracket on $\Ker \#_\Sheet$ is the regularized Lie bracket of vector fields.
 \end{corollary}

Notice that $\Ker \#_\Sheet$ is exactly the Lie algebra of the vertex group $\SDiff(\Domplus) \times \SDiff(\Dommin)$ (cf. Proposition \ref{prop:isoalgebra}).




\medskip

 \subsection{The tangent space to the algebroid of discontinuous fields}

Here we describe the tangent space to $\dsvect(M)$. Since $\dsvect(M)$ is the space of possible velocities 
for a fluid with a vortex sheet, its tangent space is, in a sense, the space of possible accelerations. 
In what follows we use this description to show that the velocity uniquely determines the pressure 
terms $p^\pm$ in \eqref{twoPhaseEuler}.

\begin{definition}\label{def:tanVectToDSVECT} 
Let $u_t =  \chiplust u_t^+ + \chimint u_t^- \in \dsvect(M)$ be a smooth 
curve in the space $\dsvect(M)$.
Then the \textit{tangent vector} to $u_t$ at $t = t_0$ is a pair
$
(v,\xi) \in \dsvectext(M,{\Sheet_{t_0}}) \oplus  \T_{\Sheet_{t_0}}\VS(M)
$ defined by
\begin{equation}\label{tanVectToDSVECT}
v:=  \left.\frac{ \hphantom{a}d^\Reg}{d t}\right\vert_{t = t_0}\!\!\!\!\! u_t\,, 
\quad \xi:=\left.\frac{ \hphantom{}d}{d t}\right\vert_{t = t_0}\!\!\!\!\! \Sheet_t \,.
\end{equation}
Note that the $v$-component of the tangent vector is the covariant derivative of $u_t$ along the curve $\Sheet_t$,
where we regard $u_t$ as a section of the algebroid over $\Sheet_t$, cf.  Remark \ref{connectionInExtBundle}.
\end{definition}

 Let  $\tanjump \colon  \dsvect(M,\Sheet) \to \Vect(\Sheet)$ be the map assigning to each 
 $u \in \dsvect(M,\Sheet)$ the jump of its tangential component at $\Sheet$:
 \begin{equation}\label{tanJumpOperator}
 \tanjump(\chiplus u^+ + \chimin u^-) := (u^+ - u^-)\vert_{\Sheet}\,.
  \end{equation}
 Let also $\normjump \colon  \dsvectext(M,\Sheet) \to \T_\Sheet \VS(M) \simeq \Omega_0^{n-1}(\Sheet)$ be the map assigning to each $v \in \dsvectext(M, \Sheet)$ the jump of its normal component at $\Sheet$:
 \begin{equation}\label{normJumpOperator}
 \normjump(\chiplus v^+ + \chimin v^-) := (i_{(v^+ - v^-)} \mu) \vert_{\Sheet}\,.
 \end{equation}

\begin{proposition}\label{tanSpaceDes}
The (Fr\'{e}chet) tangent space to $\dsvect(M)$ at a point $u \in \dsvect(M, \Sheet)$ is a vector subspace of $ \dsvectext(M,\Sheet) \oplus  \T_\Sheet \VS(M) $ consisting of pairs $(v, \xi)$ satisfying 
\begin{align}\label{tangentSpaceEq}
\normjump(v) = \L_{\tanjump(u)}\xi\,.
\end{align}
Here $\xi$ is understood as an $(n-1)$-form with zero mean on $\Sheet$, cf. Remark \ref{rem:normcot}.
\end{proposition}
\begin{remark}
Recall that the bundle $\pi \colon \dsvectext(M)^{} \to \VS(M)$ is equipped with a natural connection, see 
Remark \ref{connectionInExtBundle}. As any connection in a vector bundle, it defines a splitting
$$
\T_u\dsvectext(M)^{} \simeq \dsvectext(M, \pi(u))^{}_{} \oplus  \T_{\pi(u)} \VS(M)\,.
$$
In Proposition \ref{tanSpaceDes} we described the subspace $\T_u\dsvect(M)^{}  \subset \T_u\dsvectext(M)^{} $ in terms of this splitting.
\end{remark}
\begin{proof}[Proof of Proposition \ref{tanSpaceDes}]
Let $u_t = \chiplust u_t^+ + \chimint u_t^- \in \dsvect(M)$ be a curve with $u_ 0 = u$. We prove that its tangent vector $(v,\sigma)$ at $u$, given by \eqref{tanVectToDSVECT}, satisfies \eqref{tangentSpaceEq}. 
Let $\phi_t$ be a curve in $\SDiff(M)$ such that $\phi_{0} = \id$ and $(\phi_t)_*\Sheet = \Sheet_{t}$. Then the vector field $\phi_t^*u_t$ belongs to $\dsvect(M, \Sheet)$ for every $t$, and thus so does its time derivative\begin{align}\label{derDiff}
\left.\frac{ \hphantom{}d}{d t}\right\vert_{t = 0}\!\!\!\!\!  \left( \phi_t^*u_t \right) =    [w, u]^\Reg + v\,, 
\end{align}
where
$
w:= {\partial_t  \phi_t}$ at $t = 0$. 
Applying $\normjump$ to both sides of \eqref{derDiff} 
and using that the left-hand side is in $\dsvect(M, \Sheet)$, we get 
\begin{align}\label{derDiffJump}
\begin{aligned}
\normjump(v) &= \normjump[u,w]^\Reg = (i_{[u^+\!, w]} \mu)\vert_{\Sheet} -  (i_{[u^-\!, w]} \mu)\vert_{\Sheet} \\ &=  ([\L_{u^+}, i_{w}]\mu)\vert_{\Sheet} -([\L_{u^-},i_{w} ]\mu)\vert_{\Sheet}   = \L_{\tanjump(u)}((i_w\mu)\vert_{\Sheet})\,,
\end{aligned}
\end{align}
where we used the standard formula  $i_{[u,w]} = [\L_u, i_w]$, and  divergence-free conditions $\L_{u^+} \mu = \L_{u^-} \mu = 0$. 
Now it suffices to notice that for $t = 0$ we have $ {\partial_t\Sheet_t}{} = (i_w\mu)\vert_{\Sheet}$, so \eqref{derDiffJump} is equivalent to \eqref{tangentSpaceEq}. 
\par
Conversely, given any pair $(v ,\xi)$ satisfying \eqref{tangentSpaceEq}, the curve
$$
u_t := \exp(tw)_*(u + t(v+[w, u]^\Reg))\,,
$$
where $w \in \svect(M)$ is any divergence-free vector field on $M$ whose normal component at $\Sheet$ is $\xi$, has $(v, \xi)$ as its tangent vector at $u$. Thus, the space of tangent vectors to $\dsvect(M)$ at $u$ coincides with the solution space of \eqref{tangentSpaceEq}, as desired.
\end{proof}
\medskip


 \subsection{The dual algebroid and its tangent space}\label{sect:dual}
 In this subsection, we describe the dual of the Lie algebroid $\dsvect(M)$. This space can be viewed as the space of momenta (or the space of circulations) for a fluid with an immersed vortex sheet.\par
As the dual of $\dsvect(M)$, we consider the ``smooth dual bundle'' defined as follows. Let 
 $$\dforms(M) \,:=\, \bigcup\nolimits_{\Sheet \in \VS(M)} \dforms(M,\Sheet)\,.$$ This is a vector bundle over $\VS(M)$. Similarly, let
 $$\dexactforms(M)\, := \,\bigcup\nolimits_{\Sheet \in \VS(M)}\dexactforms(M,\Sheet)\,.$$
 This is a subbundle in $\dforms(M)$.
 \begin{definition}
 The \textit{smooth dual bundle} $\dsvect(M)^*$ is the quotient
 \begin{align}\label{dualFiber}
 \dsvect(M)^* := \dforms(M) \, / \, \dexactforms(M)\,.
\end{align}
The pairing between a coset $[\alpha] \in  \dsvect(M, \Sheet)^* =  \dforms(M, \Sheet) \, / \,  \dexactforms(M,\Sheet)$ and a vector field $u  \in   \dsvect(M)_\Sheet$ is given by the formula
\begin{align}\label{cosetAction}
\langle[\alpha], u\rangle := \int_M (i_u \alpha)  \mu\,, 
\end{align}
where $\alpha$ is any representative of the coset $[\alpha]$. 
\end{definition}

The value of $\langle[\alpha], u\rangle$ does not depend on the choice of the representative $\alpha \in [\alpha]$ because the pairing between $\dexactforms(M,\Sheet)$ and $\dsvect(M, \Sheet)$ vanishes by Proposition \ref{Hodge}.
So, any  coset $[\alpha] \in  \dforms(M, \Sheet) \, / \,  \dexactforms(M,\Sheet)$  gives rise to a well-defined linear functional on the space $\dsvect(M, \Sheet)$. For $[\alpha] \neq 0$ this functional is non-zero by the following proposition.

\begin{proposition}\label{prop:ccRepr}
For any choice of a Riemannian metric on $M$, any coset $[\alpha] \in \dforms(M, \Sheet) \, / \,  \dexactforms(M,\Sheet)$ has a unique co-closed representative $\alpha \in \dccforms(M, \Sheet)$.
\end{proposition}
\begin{proof}
This immediately follows from decomposition \eqref{SHD}.
\end{proof}
\begin{corollary}\label{cor:nonVanish}
Any non-zero coset $[\alpha] \in \dforms(M, \Sheet) \, / \,  \dexactforms(M,\Sheet)$ defines a non-zero functional on the space $\dsvect(M, \Sheet)$.
\end{corollary}
\begin{proof}
Indeed, let $[\alpha] \in  \dforms(M, \Sheet) \, / \,  \dexactforms(M,\Sheet)$ be non-zero, and let $\alpha_{cc} \in [\alpha]$ be the co-closed representative (with respect to arbitrary Riemannian metric on $M$ compatible with the volume form $\mu$). Then $\alpha_{cc}^\# \in \dsvect(M, \Sheet)$, and $\langle[\alpha], \alpha_{cc}^\#\rangle = \langle \alpha_{cc}^\#, \alpha_{cc}^\# \rangle_{L^2} > 0$, so the functional defined by $[\alpha]$ is non-zero.
\end{proof}

It follows that the smooth dual $ \dsvect(M, \Sheet)^* =  \dforms(M, \Sheet) \, / \,  \dexactforms(M,\Sheet)$ is indeed a subspace of the (continuous) dual space to $\dsvect(M, \Sheet)$. (Another important property is that this subspace ``separates points'', i.e. for any non-zero $u \in \dsvect(M, \Sheet)$ there exists $[\alpha] \in \dsvect(M, \Sheet)^*$ such that $\langle [\alpha], u\rangle \neq 0$. This is equivalent to saying that $\dsvect(M, \Sheet)$ injects into the dual of its smooth dual, which is needed for the Poisson bracket on $\dsvect(M, \Sheet)^*$ to be well-defined.)

 \medskip

 In the next section we present a Poisson bracket on the dual bundle $\dsvect(M)^*$. For this we describe the tangent and cotangent spaces to this dual. Start with the tangent space. 
 
 \begin{definition}
Let $\alpha_t =  \chiplust \alpha_t^+ + \chimint \alpha_t^- \in \dforms(M)$ be a smooth 
curve. Then the \textit{tangent vector} to $\alpha_t$ at $t = t_0$ is the pair
\begin{equation}\label{tanVectToForms}
\left.\frac{ \hphantom{a}d^\Reg}{d t}\right\vert_{t = t_0}\!\!\!\!\!  \alpha_t \,\in\, \dforms(M, \Sheet_{t_0})\,, 
\quad \left.\frac{ \hphantom{}d}{d t}\right\vert_{t = t_0}\!\!\!\!\!  \Sheet_t \,\in\,  \T_{\Sheet_{t_0}}\VS(M)\,
\end{equation}
(cf. Definition \ref{def:tanVectToDSVECT}).
\end{definition}
 Note that two curves $ [\alpha]_t,  {[\beta]}_t$  in $\dsvect(M)^*$ are {tangent} to each other at $t= t_0$ if and only if they admit lifts to $ \dforms(M)$ with the same tangent vector at $t = t_0$.
\begin{proposition}
Let $[\alpha]_t$ be a curve in $\dsvect(M)^*$. Consider an arbitrary lift $\alpha_t$ of this curve to $ \dforms(M)$. Then the coset of  ${\partial_t^\Reg \alpha_t}$
in $ \dforms(M) \, / \,  \dexactforms(M)$  at $t = t_0$ depends only on $\alpha_{t_0}$, not on the whole lift $\alpha_t$.
\end{proposition}
\begin{proof}
Without loss of generality assume that $t_0 = 0$. 
 Let $\alpha_t$, $\tilde \alpha_t$ be two different lifts starting at $\alpha_0$. Then, for the $1$-form $\gamma_t := \tilde \alpha_t - \alpha_t $, we have 
 $
\gamma_t  \in  \dexactforms(M, \Sheet_t)
 $ and $\gamma_0 = 0$. Let $\phi_t$ be any smooth curve in $\Diff(M)$ such that $\phi_0 = \id$ and $(\phi_t)_*\Sheet_0 = \Sheet_t$. Then
 $$
\left.\frac{ \hphantom{}d}{d t}\right\vert_{t = 0}\!\!\!\!\! \phi_t^*\gamma_t  =\left.\frac{ \hphantom{a}d^\Reg}{d t}\right\vert_{t = 0}\!\!\!\!\! \gamma_t =\left.\frac{ \hphantom{a}d^\Reg}{d t}\right\vert_{t = 0}\!\!\!\!\! \tilde \alpha_t - \left.\frac{ \hphantom{a}d^\Reg}{d t}\right\vert_{t = 0}\!\!\!\!\!  \alpha_t  \,.
 $$
 (Here we used that $\phi_0 = \id$ and $\gamma_0 = 0$.) Furthermore, we have $\phi_t^*\gamma_t \in \dexactforms(M, \Sheet_0)$, so the same holds for its time derivative, and the result follows.
 \end{proof}
   \begin{corollary}
 Let $[\alpha] \in \dsvect(M, \Sheet)^*$. Then any choice of $\alpha \in \dforms(M, \Sheet)$ representing the coset $[\alpha]$ gives rise to a splitting
 \begin{equation}\label{tanSpaceSplitting}
 \T_{[\alpha]}\dsvect(M)^* \simeq \dsvect(M, \Sheet)^*_{}\, \oplus \, \T_{\Sheet}\VS(M)\,,
 \end{equation}
 depending on $\alpha$.
 \end{corollary}
 \begin{proof}
 Let $[\alpha]_t$ be a curve in $\dsvect(M)^*$ lifting the curve $\Sheet_t$ in $\VS(M)$ and such that $[\alpha]_0 = \alpha$. Then to this curve one can associate a pair
$$\left.\frac{ \hphantom{a}d^\Reg}{d t}\right\vert_{t = 0}\!\!\!\!\!  [\alpha]_t \,\in\,  \dsvect(M, \Sheet_{t})^*, \quad\left.\frac{ \hphantom{}d}{d t}\right\vert_{t = 0}\!\!\!\!\! \Sheet_t \,\in\,  \T_{\Sheet}\VS(M)\,,$$
where 
$$\frac{ \hphantom{a}d^\Reg}{d t} [\alpha]_t:= [\frac{d^\Reg \alpha_t}{dt}]\,,$$
and 
$\alpha_t$ is an arbitrary lift  of the curve  $[\alpha]_t$  to $\dforms(M)$ such that $\alpha_{0}  = \alpha$. It is easy to see that this correspondence gives rise to an isomorphism between $\T_{[\alpha]}\dsvect(M)^*$ and $\dsvect(M, \Sheet)^*\, \oplus \, \T_{\Sheet}\VS(M)$.
 \end{proof}

 
  \subsection{Poisson bracket on the dual algebroid}\label{sect:poisson_vs}
 In this section we show that formula \eqref{PoissonExplicit} gives a well-defined Poisson bracket on $\dsvect(M)^*$. For this we need to describe the cotangent space to  $\dsvect(M)^*$ and we  start by defining the cotangent space to the base,
  $\T^*_\Sheet \VS(M)$.
 \begin{definition}
The \textit{smooth cotangent space} $\T^*_\Sheet \VS(M)$ is the space $\Cont^\infty(\Sheet) \, / \, \R$ of functions on $\Sheet$ modulo constants. The pairing between a coset $[f] \in   \Cont^\infty(\Sheet) \, / \, \R $ and a top degree form $\xi \in  \T_\Sheet \VS(M) = \Omega_0^{n-1}(\Sheet)$ (where $n = \dim M$), which is
an element of the corresponding tangent space, is given by
  $$
  \langle [f], \xi \rangle := \int_\Sheet f\xi\,.
  $$
  (The right-hand side is independent on the choice of a representative $f \in [f]$ thanks to the zero mean condition on $\xi$.)
  \end{definition}
  Now we define the cotangent space to $\dsvect(M)^*$ by dualizing splitting \eqref{tanSpaceSplitting}.
  
   \begin{definition}
 Let $[\alpha] \in \dsvect(M)^*_\Sheet$. Then the \textit{smooth cotangent space} to $\dsvect(M)^*$ at $[\alpha]$ is
 \begin{equation}\label{cotanSpaceSplitting}
 \T^*_{[\alpha]}\dsvect(M)^* := \dsvect(M, \Sheet)_{}\, \oplus \, \T^*_{\Sheet}\VS(M)\,,
\end{equation}
 where the second summand is the smooth cotangent space. One can see that
this defines the same space for any choice of the representative $\alpha$, although the isomorphism $\T^*_{[\alpha]}\dsvect(M)^* \simeq \dsvect(M, \Sheet)_{}\, \oplus \, \T^*_{\Sheet}\VS(M)$ depends on the choice of $\alpha$.
 \end{definition}
 Further, we define the notion of a differentiable function on $\dsvect(M)^*$. Roughly speaking, a function is differentiable if it has a differential belonging to the smooth cotangent space.
 \begin{definition}
 A function $\F \colon \dsvect(M)^* \to \R$ is \textit{differentiable} if there exists a section $\delta F$ of the smooth cotangent bundle  $\T^*_{}\dsvect(M)^*$
such that for any smooth curve $[\alpha]_t$ in $\dsvect(M)^*$ one has
 $$
\tfrac{d}{dt}\left(\F({[\alpha_t]})\right) = \langle  \delta \F({[\alpha]_t}),\left({\partial_t^\Reg [\alpha]_t}{}, {\partial_t\Sheet_t}{}  \right) \rangle\,.
 $$
 \end{definition}
 Using splitting \eqref{cotanSpaceSplitting}, we decompose $\delta F({[\alpha]})$ for $[\alpha] \in \dsvect(M, \Sheet)^*$ into the fiber and base parts: 
  $$
\delta \F({[\alpha]}) = (
 {\delta^F \F}({[\alpha]}),
  {\delta^B  \F}({[\alpha]})
  ),$$
  where
  $$
 {\delta^F \F}({[\alpha]})\in \dsvect(M, \Sheet)\,, \quad     {\delta^B  \F}({[\alpha]}) \in \T^*_{\Sheet}\VS(M) = \Cont^\infty(\Sheet) \, / \, \R\,.
 $$
 
 \begin{theorem}\label{thm:bracket}
 Let  $\F_1, \F_2 \colon \dsvect(M)^* \to \R$ be differentiable functions. Then their Poisson bracket reads
 \begin{equation}
   \{\F_1, \F_2\} = \P(\delta \F_1, \delta \F_2)\,,
 \end{equation}
 where the value of the Poisson tensor $\P$ on two cotangent vectors $(v_1, [f_1]), (v_2, [f_2]) \in \T^*_{[\alpha]}\dsvect(M)^* = \dsvect(M, \Sheet)\, \oplus \, \T^*_{\Sheet}\VS(M)$ at a point $[\alpha] \in \dsvect(M)^*_\Sheet$ is
\begin{equation}\label{pbf}
 \P_{[\alpha]}\left((v_1, [f_1]), (v_2, [f_2])\right)
  =
  -\int_M \diff^\Reg \alpha(v_1, v_2)\mu  - \int_\Sheet f_1 i_{v_2} \mu + \int_\Sheet f_2 i_{v_1} \mu\,.
\end{equation}
Here $\alpha_{} \in \dforms{}(M, \Sheet)$ is  the  representative of the coset $[\alpha]$ used to define splittings \eqref{tanSpaceSplitting}, \eqref{cotanSpaceSplitting}.
 \end{theorem}
 
\begin{remark}
 Equivalently, this bracket can be written in the form, similar to  a Lie-Poisson bracket with boundary terms:
\begin{align}\label{pbf2}
 \begin{aligned}
 \P_{[\alpha]}((v_1, [f_1]),\, &(v_2, [f_2])) = \int_M\alpha_{}([ v_1, v_2 ]^\Reg)\,\mu  \\
 &+ \, \int_\Sheet \left( jump(i_{v_1}\alpha)-f_1 \right) i_{v_2}\mu\, 
 -\, \int_\Sheet \left( jump(i_{v_2}\alpha)-f_2 \right) i_{v_1}\mu\,.
\end{aligned}
\end{align}
\end{remark}
 \begin{proof}[Proof of Theorem \ref{thm:bracket}]
 Formulas \eqref{pbf} and   \eqref{pbf2} are equivalent to each other. To see this, rewrite the first term in  \eqref{pbf2} using the formula $i_{[v_1, v_2]} = [\L_{v_1}, i_{v_2}]$ and then rewrite the integrals of jumps using Lemma~\ref{intDerLemma}. So, it suffices to derive formula   \eqref{pbf2}.

  Formula \eqref{PoissonExplicit} for the Poisson bracket in the dual of an algebroid combined with formula~\eqref{sectionsBracket} for the bracket of sections of $\dsvect(M)$  gives
\begin{align}\label{PB1}
  \{\F_1, \F_2\}([\alpha]) = \left( \int_M\alpha(\left[  {\delta^F \F_1}{ }([\alpha]),   {\delta^F \F_2}{ }([\alpha])\right]^\Reg)\,\mu \right) +  S_{12} - S_{21}\,, 
\end{align}
where $S_{ij}$ is given by
\begin{align}\label{sij}
S_{ij} :=  \#U_i(\Sheet) \cdot \left(\F_j ( A)\right) - \#U_i(\Sheet)\cdot \langle A ,U_j \rangle  + \int_M \alpha( \#U_i(\Sheet) \cdot U_j) \,\mu\,,
\end{align}
 the section $A$ of $\dsvect(M)^*$ is an arbitrary extension of $[\alpha]$, and $U_i$ is a section of $\dsvect(M)$ given by
$
U_i:=  {\delta^F \F_{i}}(A)
$.
\par
To compute the expression $S_{ij}$, take any curve $ \Sheet_i(t) \in \VS(M)$ such that $\Sheet_i(0) = \Sheet$, and the tangent vector to $\Sheet_i(t)$ at $\Sheet$  is $\#U_i(\Sheet)$. 
Then
\begin{align}\label{sijsum1}
\begin{aligned}
 \#U_i(\Sheet)\, \cdot &\left(\F_j ( A)\right) =\left.\frac{ \hphantom{}d}{d t}\right\vert_{t = 0}\!\!\!\!\!  \F_j(A(\Sheet_i(t))) \\ 
  &=  \langle U_j(\Sheet), \left.\frac{ \hphantom{a}d^\Reg}{d t}\right\vert_{t = 0}\!\!\!\!\! A(\Sheet_i(t))\rangle + \langle \vphantom{\left.\frac{ \hphantom{a}d^\Reg}{d t}\right\vert_{t = 0}\!\!\!\!\!  [\alpha_i] } {\delta^B  \F_j}{ }([\alpha]), \#U_i(\Sheet) \rangle.
 \end{aligned}
\end{align}
Further, let $\alpha_i(t)$ be any curve in $\dforms(M)$ lifting $A(\Sheet_i(t))$ and such that $\alpha_i(0) = \alpha$. Then, using Lemma \ref{derIntDiscFnLemma}, we get 
\begin{gather}
\#U_i(\Sheet)\cdot \langle A ,U_j \rangle = \left.\frac{ \hphantom{}d}{d t}\right\vert_{t = 0}\!\!\! \langle A(\Sheet_i(t)) ,U_j(\Sheet_i(t)) \rangle = \left.\frac{ \hphantom{}d}{d t}\right\vert_{t = 0} \int_M (i_{U_j(\Sheet_i(t))} \alpha_i(t)) \mu
 \\ =  \langle U_j(\Sheet),  \left.\frac{ \hphantom{a}d^\Reg}{d t}\right\vert_{t = 0}\!\!\!\!\!  A(\Sheet_i(t)) \rangle  + \int_M \alpha( \#U_i(\Sheet) \cdot U_j) \,\mu  + \int_\Sheet jump( \alpha ({\delta^F \F_j}{ ([\alpha])}) ) \#U_i(\Sheet)\,.
\end{gather}
Substituting this, along with \eqref{sijsum1}, into \eqref{sij},  and then plugging the resulting formula for $S_{ij}$ into~\eqref{PB1}, one gets  \eqref{pbf2}, as desired.  \end{proof}

 \begin{corollary}
 The  Hamiltonian operator
 $$
 \P_{[\alpha]}^\# \colon \T^*_{[\alpha]}\dsvect(M)^* \to  \T_{[\alpha]}\dsvect(M)^*
 $$
 corresponding to the Poisson bracket on $\dsvect(M)^* $  is given by
 \begin{equation}\label{HamOperator}
 (v, [f]) \mapsto (-[i_v\diff^\Reg \alpha] - \#^*[f], \#v)\,,
\end{equation}
 where $\#^* \colon \T^* \VS(M) \to \dsvect(M)^*$ is the dual of the anchor map, explicitly given by
 \begin{align}\label{dualAnchor}\#^*[f] := \left\{\diff^\Reg h  \mid h \in \dcinfty(M),  jump(h) = f\right\}\end{align}
 (cf. Proposition \ref{dualAnchorProp} below).
 \end{corollary}

 \begin{proof}
By definition, we have
$$
 \langle (w, [g]) ,  \P_{[\alpha]}^\#(v, [f])\rangle =  \P_{[\alpha]}\left((v, [f]), (w, [g])\right)=  -\int_M d^\Reg \alpha(v,w)\mu - \int_\Sheet f i_w \mu + \int_\Sheet g i_v \mu\,.
$$
Take any $h \in \dcinfty(M)$ with $jump(h) = f$. Then, by Lemma \ref{intDerLemma}, we have
$$
\int_\Sheet f i_w \mu = \int_M (i_w\diff^\Reg h)\mu\,,
$$
so we end up with 
$$
 \langle (w, [g]) ,  \P_{[\alpha]}^\#(v, [f])\rangle =  \int_M i_w ( -i_vd^\Reg \alpha - \diff^\Reg h)\mu  + \int_\Sheet g i_v \mu\,.
$$
The result follows.
\end{proof}
 \medskip
  

\section{Dynamics of vortex sheets}\label{sect:dyn_vs}
In this section, $M$ is a compact connected oriented manifold without boundary endowed with a Riemannian metric $(\,,)$ and the corresponding Riemannian volume form $\mu$.

\subsection{Evolution of vortex sheets an an algebroid Euler-Arnold equation}\label{sect:euler-arn_vs}
The $L^2$ product of vector fields associated with the metric $(\,,)$ on $M$ defines a metric $\metric_{L^2}$ on the Lie algebroid $\dsvect(M)$.
\begin{proposition}
\begin{enumerate}
\item
The inertia operator $\I$ associated with the $L^2$-metric $\metric_{L^2}$ on $\dsvect(M)$  takes values in the smooth dual $\dsvect(M)^*$. For $u \in \dsvect(M, \Sheet)$, one has
$
\I(u) = [u^\flat]\,
$,
where $u^\flat$ denotes the $1$-form dual to $u$ with respect to the Riemannian metric $(\,,)$ on $M$, and $[u^\flat]$ stands for the coset of $u^\flat$ in $ \dforms(M, \Sheet) \, / \, \dexactforms(M,\Sheet)$.
\item The inertia operator $\I \colon \dsvect(M) \to  \dsvect(M)^*$ is an isomorphism of vector bundles.
\end{enumerate}
\end{proposition}\label{prop:invertibility}

\begin{proof}
By definition of the inertia operator, for $u ,v \in \dsvect(M, \Sheet)$, one has
$$
\langle \I(u), v \rangle = \langle u, v \rangle_{L^2} =  \int_{M} (u, v) \mu =  \int_{M} i_v u^\flat \mu\,.
$$
This means that the functional $\I(u)$ coincides with the functional represented by the coset of $u^\flat \in \dforms(M, \Sheet)$, proving the first statement. Further, the inertia operator $\I$ has an inverse given by $[\alpha] \to \alpha^\#$, where $\alpha \in [\alpha]$ is the co-closed representative (see Proposition \ref{prop:ccRepr}). So, $\I$ is an isomorphism of vector bundles, as desired.
\end{proof}
Since the inertia operator is invertible, we also obtain an $L^2$-metric on  $\dsvect(M)^*$, and the corresponding Euler-Arnold Hamiltonian
$$
\mathcal H([\alpha]) := \tfrac{1}{2}\normsq{ \alpha}_{L^2},
$$
where $\alpha \in [\alpha]$ is the co-closed representative.

\begin{theorem}\label{thmMain}
The Euler-Arnold equation corresponding to the $L^2$-metric on $\dsvect(M)$ written in terms of a coset $[\alpha] \in \dsvect(M)^*_\Sheet$ reads
\begin{align}\label{twoPhaseEulerCosets}
\begin{cases}
\partial_t^\Reg [\alpha] + [i_u \diff^\Reg \alpha + \tfrac{1}{2}\diff^\Reg i_u\alpha]  =0\,,\\
\qquad \partial_t \Sheet = \# u\,,
\end{cases}
\end{align}
where $\alpha \in [\alpha]$ is the co-closed representative, and $u = \alpha^\#$ is the corresponding divergence-free fluid velocity field. 
It is a Hamiltonian equation on the algebroid dual $\dsvect(M)^*$ with respect to the natural Poisson structure described above and the energy Hamiltonian function~$\mathcal H$. 
\end{theorem}

\begin{remark}
Note that singular Hodge decomposition \eqref{SHD} gives us a way to choose a canonical representative $\alpha \in \dforms(M, \Sheet)$ in every coset $[\alpha] \in \dsvect(M)^*_\Sheet$. Therefore, the derivative $\partial_t^\Reg [\alpha]$, which, in the absence of a metric, depends on the choice of a lift $\alpha$, is now well-defined.

Note also that in the absence of a vortex sheet, the equation \eqref{twoPhaseEulerCosets}
is equivalent to $\partial_t [\alpha] + [i_u \diff \alpha]  =0$, and therefore to the Euler equation \eqref{1-forms}:
$\partial_t [\alpha]+\L_u [\alpha]=0$.

\end{remark}

\begin{proof}[Proof of Theorem \ref{thmMain}]
It suffices to compute $\delta \H([\alpha])$ and apply the Hamiltonian operator. Let $[\alpha]_s$ be an arbitrary smooth curve in $\dsvect(M)^*$ with $[\alpha]_{s=0} = [\alpha]$, and let $\Sheet_s$ be its projection to $\VS(M)$. Let also $\alpha_s \in [\alpha]_s$ be the co-closed representative. Applying Lemma~\ref{derIntDiscFnLemma}, we get
\begin{align}
\frac{d}{ds}\restrict{s = 0}\mathcal H([\alpha]_s) \,&=\,  \frac{1}{2}\left.\frac{ \hphantom{}d}{d s}\right\vert_{s = 0} \int_{M} (\alpha_s, \alpha_s)\,\mu 
 &=  \langle\left.\frac{ \hphantom{a}d^\Reg}{d s}\right\vert_{s = 0}\!\!\!\!\!  [\alpha]_s\,, u\, \rangle  + \langle\, \frac{1}{2}\,
jump(\alpha, \alpha), \left.\frac{ \hphantom{}d}{d s}\right\vert_{s = 0}\!\!\!\!\!\Sheet_s\, \rangle\,,
\end{align}
meaning that
$$
 {\delta^F \H}{}([\alpha]) = u, \quad {\delta^B \H}{}([\alpha]) = \frac{1}{2}\,[jump(\alpha, \alpha)]\,.
$$
Here we use splitting \eqref{cotanSpaceSplitting} coming from the choice of a co-closed representative  $\alpha \in \dforms(M, \Sheet)$ in every coset $[\alpha] \in \dsvect(M, \Sheet)^*$. 
Now, to get \eqref{twoPhaseEulerCosets}, it suffices to apply the Hamiltonian operator \eqref{HamOperator} and notice that
$
\#^*[jump(\alpha, \alpha)] = [\diff^\Reg i_u\alpha]\,.
$
Thus, the theorem is proved.
\end{proof}

\begin{theorem}\label{thmMain2}
The Euler-Arnold equation corresponding to the $L^2$-metric on $\dsvect(M)$ written in terms of the fluid velocity field $u := \I^{-1}([\alpha]) \in \dsvect(M)$ reads
\begin{align}\label{twoPhaseEuler2}
\begin{cases}
\partial_t u^+ + \nabla_{u^+}u^+ = -\grad p^+\!, \\ 
\partial_t u^- + \nabla_{u^-}u^- = -\grad p^-\!,\\
 \qquad \partial_t \Sheet = \# u\,,
\end{cases}
\end{align}
where $p^\pm \in \Cont^\infty(\Dompm)$ are functions satisfying 
\begin{align}\label{presCont2} p^+\vert_{\Sheet} = p^-\vert_{\Sheet}\,. \end{align}
 The functions $p^+$, $p^-$ are defined uniquely modulo a common additive constant by the consistency conditions for \eqref{twoPhaseEuler2} and the constraint \eqref{presCont2}. 
\end{theorem}


\begin{proof}
By definition of the derivative $\partial_{t}^\Reg [\alpha]$, the first of equations \eqref{twoPhaseEulerCosets} is equivalent to the condition
$$
\partial_t^\Reg \alpha + i_{u}\diff^\Reg \alpha   + \tfrac{1}{2}\diff^\Reg i_u\alpha  \,\in\, \dexactforms(M, \Sheet)\,,
$$
for the co-closed representative $\alpha \in [\alpha]$. Equivalently, this can be written as
$$
\partial_t^\Reg \alpha + \L_{u}^\Reg \alpha   - \tfrac{1}{2}\diff^\Reg i_u\alpha  \,\in\, \dexactforms(M, \Sheet)\,.
$$
%
Taking the metric dual vector field and applying the formula
 $ (\L_u u^\flat - \tfrac{1}{2}\diff (u,u))^\sharp = \nabla_u u$, 
 we get
\begin{equation}\label{twoPhaseEulerDgrad}
 \partial_t^\Reg u + \nabla^\Reg_{u}u \in \dgrad(M, \Sheet)\,,
\end{equation}
 which is equivalent to the first two equations in  \eqref{twoPhaseEuler2} supplemented by condition \eqref{presCont2}. Thus, the first statement of the theorem is proved.
\par
To prove that the functions $p^\pm$ are uniquely determined by $u$, rather than by its time derivative, we need to show that the projection of $\partial_t^\Reg u \in \dvect(M,\Sheet)$ to $\dgrad(M, \Sheet)$ can be expressed in terms of $u$. 
To that end, notice that $\partial_t^\Reg u$ belongs to $\dsvectext(M, \Sheet)$, and we know the jump of its normal component by 
Proposition~\ref{tanSpaceDes}. On the other hand, the map $\normjump$ defines an isomorphism between the quotient 
$\dsvectext(M, \Sheet)\,/\,\dsvect(M, \Sheet)$ and the space of possible jumps $\T_\Sheet \VS(M)$. 
Hence we know the coset of  $\partial_t^\Reg u$ in $\dsvectext(M, \Sheet)\,/\,\dsvect(M, \Sheet)$, 
which means that we know its projection to $\dgrad(M, \Sheet)$, as desired.

Somewhat more explicitly, one has  the decomposition 
$\dsvectext(M, \Sheet) = \dsvect(M, \Sheet) \oplus_{L^2} (\dsvectext(M, \Sheet) \cap \dgrad(M, \Sheet))$. 
The latter space is exactly the space of single layer potentials, cf. Appendix \ref{app:layer}. By Theorem \ref{slp}
one can reconstruct the harmonic potentials $s^\pm$ (for a divergence-free vector field) 
continuous on $\Sheet$ and satisfying a given jump condition. 
Thus the $\dgrad(M, \Sheet)$ component of  $\partial_t^\Reg u$ is uniquely determined 
by the jump of the normal component, while the sum $\partial_t^\Reg u + \nabla^\Reg_{u}u$ defines the required 
$p^\pm$.
\end{proof}

Recall that for a fluid velocity field $u$, the corresponding vorticity is the $2$-form $\omega := \diff u^\flat$. For a vector field $u \in \dsvect(M)$, the vorticity is a de Rham current given by
$$
\omega = \diff (\chiplus \alpha^+ + \chimin \alpha^-) =  \chiplus \omega^+ + \chimin\omega^- + (\alpha^+ - \alpha^-) \wedge \delta_\Sheet\,,
$$
where $\omega^\pm = \diff \alpha^\pm$. 
\begin{corollary}[(Singular Kelvin's theorem)]\label{cor:singKelvin}
For a fluid with an immersed vortex sheet, the regular part of the vorticity is transported by the flow:
$$
\partial_t \omega^\pm + \L_u \omega^\pm = 0\,.
$$
\end{corollary}
\begin{proof}
Take the exterior derivative of both sides in the first of equations \eqref{twoPhaseEulerCosets}.
\end{proof}
In particular, the regular part of the vorticity remains in the same diffeomorphism class during the Euler-Arnold evolution. Note that the latter property in fact holds for any Hamiltonian system on $\dsvect(M)^*$ and is related to the structure of symplectic leaves. 
%
%
 
\medskip
  
  \subsection{Pure vortex sheet motions as geodesics on the space of hypersurfaces}
  \label{sect:pure_motion}
  Now, we apply Proposition \ref{prop:sub} to obtain a geodesic description of pure vortex sheet motions.\par
  
    \begin{proposition}\label{dualAnchorProp}
    Let $[f] \in  \T^*_\Sheet \VS(M)$ (recall that the latter space is $\Cont^\infty(\Sheet) \, / \, \R $). Then its image under the map $\#^* \colon \T^*_\Sheet \VS(M) \to \dsvect(M, \Sheet)^*$ is given by 
   \begin{align}\label{dualAnchor2}\#^*[f] := \left\{\diff^\Reg h  \mid h \in \dcinfty(M),  jump(h) = f\right\}.\end{align}
    \end{proposition}
    \begin{proof}
    Let $[f] \in  \T^*_\Sheet \VS(M) $, and let $u \in \dsvect(M, \Sheet)$. Then
    $$
    \langle \#^*[f], u \rangle = \langle [f], \#u \rangle = \int_\Sheet f i_u\mu\,.
    $$
    By Lemma \ref{intDerLemma}, for any $h \in  \dcinfty(M)$ such that  $jump(h) = f$, we can rewrite the latter integral as
  $$
  \int_M (\L_u^\Reg h)\mu =   \int_M (i_u \diff^\Reg h)\mu\,.
  $$
  The latter is exactly the coset on the right-hand side of \eqref{dualAnchor2} paired with $u$, as desired.
    \end{proof}
     By Proposition \ref{prop:ccRepr}, every coset \eqref{dualAnchor2} has a unique co-closed representative $\alpha \in\dccforms(M, \Sheet)$. Explicitly, it reads
     $
     \alpha = \chiplus\diff f^+  + \chimin\diff f^- 
     $,
    where the functions $f^\pm \in \Cont^\infty(\Dompm)$ are harmonic, have equal normal derivatives at $\Sheet$ and satisfy $f^+\vert_{\Sheet} - f^-\vert_{\Sheet} = f$ (these functions $f^\pm$ can be found as the solution of the double layer potential problem, see Theorem \ref{dlp} in Appendix A). The metric dual vector field $v: = \I^{-1}(\#^*[f])$ thus has the form $v = \chiplus\grad f^+  + \chimin\grad f^-  $. This means that the symplectic leaf $\#^*(\T^*\VS(M)) \subset \dsvect(M)^*$ is metric dual to velocity fields of pure vortex sheet motions.

   \begin{theorem}\label{geodDescription}
   Consider the vortex sheet algebroid $\dsvect(M)\to \VS(M)$, equipped with the $L^2$-metric. Then the following holds:  
    \begin{enumerate}
    \item Pure vortex sheets evolve along geodesics of a metric $\metric_{\VS}$ on $\VS(M)$ 
    obtained as the projection of the   $L^2$-metric on $\dsvect(M)$. Explicitly, for
    a tangent vector $\xi$ to the base  this metric reads
    \begin{align}\label{vsExplicit}
     \langle \xi, \xi \rangle_{\vs} = \langle  \#^{-1}(\xi),  \#^{-1}(\xi) \rangle_{L^2} = \int_{D^+_\Sheet}(\grad f^+, \grad f^+)\,\mu 
+  \int_{D^-_\Sheet}(\grad f^-, \grad f^-)\,\mu\,,
\end{align}
where $\Delta f^\pm=0$ in $D^\pm_\Sheet$ and the normal component of $\nabla f^\pm$ at $\Sheet$ is $\xi$.
Equivalently, 
$$
\langle \xi,\xi\rangle_{\vs}= \langle (\NTD^+ + \NTD^-)\xi, \xi \rangle\,,
$$
where the Neumann-to-Dirichlet operators $\NTD^\pm$ 
on the domains $D^\pm_\Sheet$  are regarded as maps $\T_\Sheet \VS(M) \to \T^*_\Sheet \VS(M)$.

    \item The anchor map $\#  \colon (\dsvect(M), \metric_{L^2}) \to (\T\VS(M), \metric_\vs)$ is a Riemannian submersion of vector bundles.
        \item For any $\Sheet \in \VS(M)$, the target mapping $\Trg \colon (\DSDiff(M)_{\Sheet}, \metric_{L^2}) \to (\VS(M), \metric_\vs)$ is a Riemannian submersion, see Figure \ref{fig:submersion}. Here $ \metric_{L^2}$ is the restriction of the right-invariant source-wise metric on $\DSDiff(M)$ corresponding to the $L^2$-metric on $\dsvect(M)$.
    \end{enumerate}
    \end{theorem}

    \begin{proof}
    This theorem follows from Proposition \ref{prop:sub} employed in the setting of the algebroid $\dsvect(M)$
    of discontinuous vector fields. To apply 
    Proposition \ref{prop:sub} we need to show that $\dsvect(M) = \Ker \# \oplus (\Ker \#)^\bot$. Take any $u \in \dsvect(M, \Sheet)$. Consider the solutions $f^\pm \in \Cont^\infty(\Dompm)$ of the Laplace equation with Neumann boundary conditions given by $\#u$. Then one has
  $$
    u = ({u - \chiplus \grad f^+ - \chimin \grad f^-})+ ({ \chiplus \grad f^+ + \chimin \grad f^-})
 $$
    with the first bracket being in $\Ker \# $ and the second bracket being in $(\Ker \#)^\bot$, as desired.
(Note that the above means the decomposition of the space $\dsvect(M, \Sheet)$ as ``fields tangent to $\Sheet$'' + ``gradients of double layer potentials,'' which is similar to the one described in the proof of Theorem \ref{thmMain2}.)
    \end{proof}
    Recall that $\DSDiff(M)_{\Sheet_0}$ is the configuration space of a fluid with an immersed vortex sheet. The motion of the fluid follows the geodesics of the $\metric_{L^2}$-metric. Pure vortex sheets thus correspond to horizontal (with respect to the target mapping) geodesics.\par
    
    The following result shows that the first statement of Theorem \ref{geodDescription} is equivalent to Corollary \ref{LoeschThm}.
    \begin{proposition}
    The metric $\metric_\vs$ constructed above coincides with the metric provided by Definition~\ref{def:baseMetric}.
    \end{proposition}
    
    \begin{proof}
    Denote the metric from Definition~\ref{def:baseMetric} by $\metric'_\vs$. For $\xi \in \T_\Sheet\VS(M)$, that metric reads
    $$
        \langle \xi, \xi \rangle'_\vs = \inf \left\{{\langle u, u \rangle_{L^2} \mid  u \in \svect(M), \,\# u = \xi} \right\}
        $$
        (here we regard $\svect(M)$ as a subspace of $\dsvect(M, \Sheet)$).
        But by Corollary \ref{dense} the subspace $\svect(M)$ is $L^2$-dense in $\dsvect(M, \Sheet)$, which allows 
        one to rewrite the definition of the metric as
            $$
        \langle \xi, \xi \rangle'_\vs = \inf \left\{{\langle u, u \rangle_{L^2} \mid  u \in \dsvect(M, \Sheet), \,\# u = \xi} \right\}.
        $$
        Furthermore, we have
        $$
\left\{   u \in \dsvect(M, \Sheet) \mid \# u = \xi\right\} = \#^{-1}(\xi) + \Ker \#_\Sheet\,,
        $$
        where    $\#^{-1}$ is the inverse of the restriction of the anchor map to $ (\Ker \#)^\bot$, and $ \Ker \#_\Sheet$ stands for the kernel of the anchor map at the fiber $\dsvect(M)_\Sheet$. So,
    $$
       \langle \xi, \xi \rangle'_\vs=  \inf \left\{{\langle u, u \rangle_{L^2} \mid  u \in  \#^{-1}(\xi) + \Ker \#_\Sheet} \right\} = \langle  \#^{-1}(\xi),  \#^{-1}(\xi) \rangle_{L^2} =     \langle \xi, \xi \rangle_{\vs} \,,
    $$
    where in the second equality we used that $ \#^{-1}(\xi) \in  (\Ker \#_\Sheet)^\bot$.
    \end{proof}


  \medskip

\subsection{Irrotational flows with vortex sheets as a Newtonian system in a magnetic field} \label{sec:irrot}
  In this section we consider flows with vortex sheets which are irrotational outside of the vortex sheet. In terms of the vector field $u = \chiplus u^+ + \chimin u^- \in \dsvect(M)$ this means that its smooth parts $u^\pm \in \svect(\Dompm)$ are irrotational (i.e. locally potential or, equivalently, harmonic), while in terms of the dual coset $[\alpha] := \I(u) \in \dsvect(M)^*$ this means that  $\diff \alpha^\pm = 0$ for some (equivalently, any) form $ \chiplus\alpha^+  +\chimin \alpha^-  \in [\alpha]$. As follows from Corollary~\ref{cor:singKelvin}, the space of such cosets is invariant under the Euler-Arnold flow \eqref{twoPhaseEulerCosets}.
Moreover, assume that there is a canonical way to identify the cohomology groups $\Hom^1(\Dompm, \R)$ for different $\Gamma$'s (this happens when the Gauss-Manin connection in the bundles $\Hom^1(D^\pm_\Sheet, \R) \to \VS(M)$ has trivial monodromy).

\begin{proposition}
The cohomolgy classes of $1$-forms $\alpha^\pm$ in $\Hom^1(\Dompm, \R)$ are invariants of the Euler-Arnold flow. In other words, in the irrotational case the Euler-Arnold flow \eqref{twoPhaseEulerCosets} on $\dsvect(M)^*$ restricts to the affine subbundle
$$ 
\dsvect(M)^*_{\theta^\pm} := \{[\chiplus\alpha^+  +\chimin \alpha^- ] \mid \diff \alpha^+ = \diff \alpha^- = 0, [\alpha^+] = \theta^+, [\alpha^-] = \theta^-\} \subset \dsvect(M)^*\,,
$$
where $\theta^\pm \in \Hom^1(\Dompm, \R)$ are fixed cohomology classes. 
\end{proposition}

\begin{proof}
It is easy to see from \eqref{twoPhaseEulerCosets} that the integrals of $\alpha^\pm$ over closed cycles are  dynamically invariant, provided that these cycles stay away from $\Sheet$. 
\end{proof}

Notice also that $\dsvect(M)^*_{\theta^\pm} = A + \#^*(\T^*\VS(M))$ for any section $A$ of $\dsvect(M)^*_{\theta^\pm}$. Furthermore, the metric allows us to choose the section $A$ in a canonical way: for any $\Sheet \in \VS(M)$ there exists a unique form $ \alpha^{\bot}_\Sheet := \chiplus\alpha_\Sheet^+  +\chimin \alpha_\Sheet^-$, where the $1$-forms $\alpha_\Sheet^\pm$ are harmonic, belong to the cohomology classes $\theta^\pm$, and vanish in the normal direction to $\Sheet$. 
Note that the form  $ \alpha^{\bot}_\Sheet$ chosen in such a way is $L^2$-orthogonal
to the space $\#^*(\T^*\VS(M))$. 
We then set $A(\Sheet)$ to be the coset of $\alpha^{\bot}_\Sheet$. This gives an identification $\dsvect(M)^*_{\theta^\pm} \simeq \#^*(\T^*\VS(M))$. The latter image 
$\#^*(\T^*\VS(M))$ can be identified with $\T^*\VS(M)$ itself, since
the map $\#^*$ is injective. 

It turns out that with this identification the Euler-Arnold flow on $\dsvect(M)^*_{\theta^\pm}$ can be described as a Newtonian system in a magnetic field. The corresponding potential $\mathcal P \colon \VS(M) \to \R$ is given by $$\mathcal P(\Sheet) := \frac{1}{2}\langle \alpha^{\bot}_\Sheet,  \alpha^{\bot}_\Sheet\rangle_{L^2}\,,$$ while the magnetic term is defined as follows. Take any $\xi \in \T_\Sheet \VS(M)$. Then, since  $  \alpha^{\bot}_\Sheet = \chiplus\alpha_\Sheet^+  + \chimin\alpha_\Sheet^-  $, where the $1$-forms $\alpha_\Sheet^\pm$ are closed and belong to the fixed cohomology classes, it follows that the derivative $\xi \cdot \alpha^{\bot}_\Sheet$ of $\alpha^{\bot}_\Sheet$ with respect to $\Sheet$ in the direction $\xi$ has the form $\chiplus\diff f^+  + \chimin\diff f^- $, which means that
  $
 [ \xi \cdot \alpha^{\bot}_\Sheet] \in \#^*(\T_\Sheet^*\VS(M)) \simeq \T_\Sheet^*\VS(M)
  $. 
  Then, for any $\xi_1, \xi_2 \in \T_\Sheet \VS(M)$, we set
  $$
  \Omega(\xi_1, \xi_2) := \langle  [\xi_1 \cdot \alpha^{\bot}_\Sheet], \xi_2 \rangle - \langle  [\xi_2 \cdot \alpha^{\bot}_\Sheet], \xi_1 \rangle \,.
  $$
  One can check that this skew-symmetric $2$-form on $\VS(M)$ is closed. Thus, it defines a symplectic structure on $\T^*\VS(M)$ by the formula
  $
  \Omega_{mag} := \Omega_{can} + \pi^* \Omega,
  $
  where $\Omega_{can}$ is the canonical symplectic form on $\T^*\VS(M)$ and
$\pi \colon \T^*\VS(M) \to \VS(M)$ is the canonical projection.
  \begin{theorem}
 The evolution of a  fluid with a vortex sheet defined by an irrotational
initial vector field outside of the sheet is a Hamiltonian system on $\T^*\VS(M)$ with symplectic structure given by $ \Omega_{mag} := \Omega_{can} + \pi^* \Omega$ and the Hamiltonian function given by $\H:= \mathcal K + \pi^*\mathcal P$, where $\mathcal K$ is the kinetic energy corresponding to the metric $\metric_\vs$, and $\mathcal P$ is the potential term defined above. 
  \end{theorem}
  \begin{remark}
  In the particular case when the flow is globally potential outside of 
the vortex sheet, this is Theorem \ref{geodDescription}. In this case, $\mathcal P = 0$ and $\Omega = 0$. 
In the general case of locally potential flow with a vortex sheet, the fluid evolution in this case can be regarded as motion on $\VS(M)$ under the influence of potential $\mathcal P$ and magnetic field $\Omega$.
Note that a mechanical analog of this system is a charged particle $q$ moving in a magnetic field depending on $q$ in the presence of a potential field, i.e. a Hamiltonian system with Hamiltonian of the form ``kinetic energy'' + ``potential energy'' in the twisted symplectic structure.
 \end{remark}
  \begin{remark}
  In the case when the monodromy of the Gauss-Manin connection in $\Hom^1(D^\pm_\Sheet, \R) \to \VS(M)$ is non-trivial, one should pass to a certain covering  $\widetilde{\VS}(M) \to \VS(M)$ trivializing the monodromy.
  \end{remark}
 \begin{figure}[t]
\centerline{
\begin{tikzpicture}[thick, scale = 1]
          \node at (4.5,0) () {
          \begin{tikzpicture}[thick, rotate = 0, scale = 0.8]
                                     \node  at (1, -0.4) () {$u^- = 0$};
                                       \node  at (1, 2.3) () {$u^+ = \grad \phi$};
 \node  at (0.95,0.5) () {
\begin{tikzpicture}[thick, rotate = -90, scale = 1]
    \draw (3,-3) ellipse (1.2cm and 2cm);
    \draw   (3.05,-3.87) arc (225:135:1.2cm);
    \draw   (2.9,-4) arc (-55:55:1.2cm);
        \draw   (1.83,-3.4) arc (120:60:0.93cm);
        \draw   (1.83,-2.6) arc (120:60:0.95cm);
           \draw [->]   (1.9,-3.1) arc (120:60:0.7cm);
                \draw [->]   (1.9,-2.8) arc (120:60:0.7cm);
     \end{tikzpicture}
     };
    %
\end{tikzpicture}
          };
\end{tikzpicture}
}
\caption{An irrotational steady solution with a vortex sheet. This solution delivers the global maximum of the corresponding potential $\mathcal P$ on $\VS(M)$.
}\label{torus2}
\end{figure}
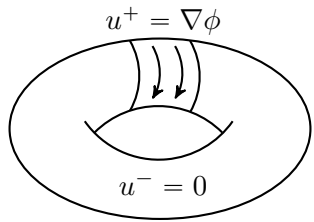
\begin{example}
Consider a flat torus with vortex sheets being two meridians and dividing the torus into two parts, where the irrotational flow is zero in one of the parts and constant $\grad \phi$ parallel to the sheets in the other, see Figure \ref{torus2}.
Such a flow is only locally potential, as the would-be potential function $\phi$ is linear and multivalued on the torus. 
This is a steady solution of the Euler equation. In general, steady solutions
with vortex sheets correspond to critical points of the potential $\mathcal P$ in the irrotational case. It is easy to show that the above case of plane parallel flow on
the torus corresponds to the global maximum of $\mathcal P$.
On the other hand, in the pure vortex sheets case there cannot be any steady solutions (as there are no geodesics consisting of a single point).
\end{example}
 
    \medskip

 The proof of the above theorem is based on a detailed study of symplectic leaves of
 $\dsvect(M)^*$. 
This will be a subject of our forthcoming publication.

  
\medskip

\section{Appendix \ref{app:layer}: Layer potentials on Riemannian manifolds}
\refstepcounter{AppCounter}
\label{app:layer}
In this appendix we establish existence of single and double layer potentials on compact  Riemannian manifolds. Although this result is definitely known to experts, we were not able to find a proof of the general case in  literature, so we sketch the proof here.
\begin{theorem}[(On the single layer potential)]\label{slp}
Let $M$ be a compact Riemannian manifold without boundary, and let $\Sheet \subset M$ be a closed hypersurface splitting $M$ into two parts $\Domplus$, $\Dommin$. Let also $\nu$ be the unit normal field to $\Sheet$, and let $f \in \Cont_0^\infty(\Sheet)$ be a smooth function on $\Gamma$ with zero mean. Then there exist smooth up to $\Sheet$ harmonic functions $s^\pm \in \Cont^{\infty}(\Dompm)$ which satisfy
\begin{equation}
s^+\vert_{\Sheet} = s^-\vert_{\Sheet}\,,\quad
(\grad s^+ - \,\grad s^-, \nu) = f\,.
\end{equation}
Such functions $s^+$, $s^-$ are unique up to a common additive constant. ({The function $s$ equal to $s^+$ in $\Domplus$ and $s^-$ in $\Dommin$ is called a \rm{single layer potential}.})
\end{theorem}

\begin{proof} 
 In terms of the function $g := s^\pm\restrict{\Sheet}$, the problem can be reformulated as
\begin{equation}\label{sumdtn}
 (\DTN^+ + \DTN^-)g = f\,
\end{equation}
where $\DTN^\pm \colon \Cont^\infty(\Sheet) \to \Cont^\infty(\Sheet)$ are Dirichlet-to-Neumann maps associated with the domains $\Dompm$. (We have a sum instead of a difference because an outward normal for $\Domplus$ is an inward normal for $\Dommin$.) Since functions $s^\pm$ can be reconstructed from $g$ uniquely (as harmonic functions with boundary values $s^\pm\vert_{\Sheet} = g$), it suffices to show that given $f$ the equation \eqref{sumdtn} has a unique, up to an additive constant, solution $g$.\par

It is straightforward to verify that the Dirichlet-to-Neumann map is a non-negative formally self-adjoint 
operator whose kernel consists of constant functions. Therefore, $$
\Ker\, (\DTN^+ + \DTN^-) = \Ker\, \DTN^+ \cap \Ker\, \DTN^- \!= \R\,. 
$$
This proves the uniqueness part. To prove existence, we use that each of the Dirichlet-to-Neumann maps 
$\DTN^\pm$ is an elliptic pseudo-differential operator with principal symbol $S \colon \T^*\Sheet \to \R$ 
given by the Riemannian length $S(\xi) = \sqrt{(\xi, \xi)}$ (see \cite{LU89}, Proposition 1.2).
It follows that $\DTN^+ + \DTN^-$ is a  pseudo-differential operator with  principal symbol $2\sqrt{(\xi, \xi)}$ 
and hence elliptic. 
Now recall that for any elliptic formally self-adjoint pseudo-differential operator $\mathcal D \colon C^\infty(X) \to C^\infty(X)$ on a compact Riemannian manifold $X$ one has
$
\Im \, \mathcal D = (\Ker\,\mathcal D )^\bot 
$
(see, e.g., \cite{Taylor}, Chapter 7, Section 10). Therefore,
$$
\Im \,  (\DTN^+ + \DTN^-) =  (\Ker\,(\DTN^+ + \DTN^- ) )^\bot =\Cont_0^\infty(\Sheet).
$$
Thus, \eqref{sumdtn} is solvable for any $f$ with zero mean, as desired.
\end{proof}

\begin{theorem}[(On the double layer potential)]\label{dlp}
Let $M$ be a compact Riemannian manifold without boundary, and let $\Sheet \subset M$ be a closed hypersurface splitting $M$ into two parts $\Domplus$, $\Dommin$. Let also $\nu$ be the unit normal field to $\Sheet$, and let $g \in \Cont^\infty(\Sheet)$ be arbitrary. Then there exist  smooth up to $\Sheet$ harmonic functions $h^\pm \in \Cont^{\infty}(\Dompm)$ which satisfy
\begin{equation}
h^+\vert_{\Sheet} - h^-\vert_{\Sheet} = g\,,\quad
(\grad h^+, \nu) = (\grad h^-, \nu)\,.
\end{equation}
Such functions $h^+$, $h^-$ are unique up to a common additive constant. ({The function $h$ equal to $h^+$ in $\Domplus$ and $h^-$ in $\Dommin$ is called a \rm{double layer potential}.})
\end{theorem}

\begin{proof}
By taking any harmonic functions $g^\pm \in \Cont^{\infty}(\Dompm)$ with $g^+\vert_{\Sheet} - g^-\vert_{\Sheet} = g$, we reduce the question to the single layer potential problem in terms of the functions $s^\pm := g^\pm - h^\pm$.
\end{proof}
\medskip


\section{Appendix \ref{app:weak}: Lie algebroid geodesics as weak solutions of the Euler equation}
\refstepcounter{AppCounter}
\label{app:weak}
In this appendix we prove that solutions of  the algebroid Euler-Arnold equation~\eqref{twoPhaseEuler2} can be regarded as weak solutions of the Euler equation \eqref{idealEulerIntro}. And conversely, weak solutions with vortex sheet type discontinuity solve equations~\eqref{twoPhaseEuler2}. We believe that this result is well known to experts, at least in some particular cases. However, we were not able to find the proof of the general case in the literature, so we provide it here.

   \begin{definition}\label{weak}{\rm (cf. e.g. \cite{shnirelman1997nonuniqueness})}
  A time-dependent vector field $u$ on $M$ is said to be a \textit{weak solution} for the incompressible Euler equation if $u \in L^2_{loc}(M \times \R)$, 
  \begin{equation}\label{wc1}
\langle u , \grad f \rangle_{L^2}= 0 \quad \forall\, f \in \Cont^\infty(M)\,,
 \end{equation}
 and 
 \begin{equation}\label{wc2}
  \int_{-\infty}^{+\infty} \langle u, \partial_t v + \nabla_u v \rangle_{L^2}\,\diff t = 0 \quad \forall \, v \in \Cont^\infty_c(M \times \R) \mbox{ such that } \div v(., t_0) = 0 \,\, \forall \, t_0 \in \R.
  \end{equation}
 Here $L^2_{loc}$ stands for square integrability on compact subsets, while $ \Cont^\infty_c$ means $\Cont^\infty$ with compact support.
  \end{definition}

  There is a version of the definition for a finite time interval and adapted to an initial value problem.
%
%
%
   It is well-known that for $\Cont^\infty$ (in both space and time) vector fields $u$ these conditions are equivalent to the divergence-free condition on $u$ and the Euler equation~\eqref{idealEulerIntro}.
 For sufficiently regular solutions one can also restate condition \eqref{wc2} as follows:

    \begin{lemma}\label{firstWeakLemma}
  Assume that a time-dependent vector field $u \in L^2_{loc}(M \times \R)$ is continuous in $t$ with respect to the $L^2$-norm. Furthermore, assume that for any divergence-free $ v \in \Cont^\infty_c(M \times \R) $ the inner product $\langle u, v \rangle_{L^2}$ is a continuously differentiable function of $t$. Then \eqref{wc2} holds if and only if
\begin{equation}\label{wc2p}
\partial_t\langle u, w \rangle_{L^2} = \langle u, \nabla_u w \rangle_{L^2}  \quad \forall \, w \in \svect(M)\,.
  \end{equation}
  \end{lemma}
  \begin{proof}
  Assume that  \eqref{wc2} holds. Let
  $$
  F_v(t) := \partial_t\langle u, v \rangle_{L^2}  - \langle u, \partial_t v + \nabla_u v \rangle_{L^2}\,.
  $$
  Then, by \eqref{wc2}, for any  time-dependent  divergence-free smooth vector field $v$ with compact support we have
  $
  \int_{-\infty}^\infty F_vdt = 0.
  $
  Also notice that $F_{\phi v} = \phi F_v$ for any smooth scalar function $\phi = \phi(t)$. So, we get that
    $$
  \int_{-\infty}^\infty \phi F_vdt =   \int_{-\infty}^\infty F_{\phi v}dt = 0
  $$
  for any $\Cont^\infty$-smooth $\phi(t)$. Taking into account that $  F_v(t)$ is a continuous function, it follows that $F_v(t) \equiv 0$. Now, if $w \in \svect(M)$, we take any smooth function $\phi = \phi(t) \neq 0$ with compact support and consider $v := \phi w$. Then  $F_v \equiv 0$ implies \eqref{wc2p}.\par \smallskip
  Conversely, assume that  \eqref{wc2p} holds. 
 Take any  time-dependent  divergence-free smooth vector field $v$ with compact support and write it as $v = v(t_0) + (t-t_0) \tilde v$. Then $F_v = F_{v(t_0)} + (t-t_0)F_{\tilde v}$. Notice that the first summand vanishes by  \eqref{wc2p}, while the second summand vanishes for $t = t_0$. Since $t_0$ is arbitrary, it follows that $F_v = 0$ for any $t$. Integrating with respect to $t$, we get \eqref{wc2}, as desired.
  \end{proof}

  \begin{theorem}\label{thm:weak}
  Let $u(t)$ be a smooth curve in $\dsvect(M)$. Then the following is true.
\begin{enumerate}
\item If $u$ satisfies the algebroid Euler-Arnold equation~\eqref{twoPhaseEuler2} (with condition \eqref{presCont2}), then it is a weak solution of the incompressible Euler equation. 
\item If $u$ is a weak solution of the incompressible Euler equation, and, in addition, its jump $\tanjump(u)$ is non-zero almost everywhere on $\Sheet$ for every $t$, then $u$ satisfies the algebroid Euler-Arnold equation~\eqref{twoPhaseEuler2} and condition \eqref{presCont2}.
\end{enumerate}
  \end{theorem}
  \begin{remark}
The assumption on $\tanjump(u)$ in the second part of the theorem can not be omitted. For instance, if $u$ is a $\Cont^\infty$-smooth solution of the Euler equation, and $\Sheet =  \Sheet(t)$ is any smooth curve in $\VS(M)$, then the pair $(u,\Sheet)$, viewed as a curve in  $\dsvect(M)$ (here we regard $u$ as a vector field discontinuous across $\Sheet$, although there is no actual discontinuity), is a weak solution of the Euler equation, but does not, generally speaking, satisfy  the last equation in~\eqref{twoPhaseEuler2}.   \end{remark}
  
  The proof of Theorem \ref{thm:weak} is based on the following lemma.
  \begin{lemma}\label{weakLemma2}
  Let $M$ be a compact Riemannian manifold, and let $\Sheet \subset M$ be a closed hypersurface. Let also $v \in \vect(\Sheet)$ be a $\Cont^\infty$-smooth vector field on $\Sheet$ (tangent to $\Sheet$). Then there exists a $\Cont^\infty$-smooth divergence-free extension of $v$ on $M$ with an arbitrary small $L^2$-norm.
    \end{lemma}
    \begin{proof}
    We choose an identification of a neighborhood of $\Sheet$ in $M$ with $\Sheet \times \R$ such that the volume form on $M$ splits into a product of the volume form on $M$ and the standard volume form $\diff z$ on $\R$. Then, for any smooth function $\psi(z)$ with compact support such that $\psi(0) = 0$ and $\psi'(0) = 1$, the vector field
    $$
   u_n :=  \psi'(nz) v - \frac{1}{n}\psi(nz) \div v \frac{\partial}{\partial z}
    $$
    is a smooth divergence-free extension of $v$ supported in a small neighborhood of $\Sheet$. Furthermore, we have $||u_n||_{L^2} \to 0$ as $n \to \infty$, as desired.
    \end{proof}
    
%
  \begin{proof}[Proof of Theorem \ref{thm:weak}]
 Let $u$ be any smooth curve in $\dsvect(M)$ and let $w \in \svect(M)$. Then, by Lemma \ref{derIntDiscFnLemma}, we have
  $$
  \partial_t \langle u, w \rangle_{L^2} =   \langle \partial_t^\Reg u, w \rangle_{L^2}  + \int_{\Sheet} jump(u, w) {\partial_t\Sheet}{}\,.
  $$
At the same time, by Lemma \ref{intDerLemma}, we have
\begin{align*}
\int_{M} ( u,  \nabla_{u} w)\mu = \int_{M} &( -  ( \nabla^\Reg_{u}u, w) + \L_u^\Reg (u,w))\mu =   - \langle  \nabla^\Reg_{u}u, w \rangle_{L^2} +  \int_{\Sheet} jump(u,w)\#u\,.
\end{align*}
So, for a smooth curve in $\dsvect(M)$, condition \eqref{wc2p} is equivalent to
\begin{align}\label{wc2pp}
 \langle \partial_t^\Reg u +  \nabla^\Reg_{u}u, w \rangle_{L^2}  + \int_{\Sheet} jump(u, w) ({\partial_t\Sheet}{} - \#u) = 0 \quad \forall \, w \in \svect(M)\,.
\end{align}
  Now, notice that the latter holds for any solution of~\eqref{twoPhaseEuler2}. So, \eqref{wc2p} holds as well, and so does~\eqref{wc2} by Lemma \ref{firstWeakLemma}. Finally, notice that \eqref{wc1} is true for any curve in $\dsvect(M)$ due to the singular Hodge decomposition~\eqref{SHD2}. This proves the first statement of the theorem.\par
  To prove the second statement, we need to show that \eqref{wc2pp} implies equations~\eqref{twoPhaseEuler2}. To that end, take $w \in \svect(M)$ orthogonal to $\Sheet$ at every point. Then, for such $w$, we have $jump(u, w) = (\tanjump(u), w) = 0$, so the second summand in \eqref{wc2pp} vanishes. Therefore, the first summand vanishes as well. Since this holds for any  $w \in \svect(M)$ orthogonal to $\Sheet$, and such vector fields are $L^2$-dense in $\svect(M)$ be Lemma \ref{weakLemma2}, it follows that the first summand in  \eqref{wc2pp} vanishes for every $w \in \svect(M)$. But since $\svect(M)$ is $L^2$-dense in $\dsvect(M, \Sheet)$ (Corollary \ref{dense}), we get that
  $$
  \partial_t^\Reg u +  \nabla^\Reg_{u}u \in \dsvect(M, \Sheet)^\bot = \dgrad(M, \Sheet)\,,
  $$
  which is equivalent to the first two of equations~\eqref{twoPhaseEuler2} supplemented by condition \eqref{presCont2}.\par
  To derive the last equation, we use that the first summand in  \eqref{wc2pp} vanishes for every $w \in \svect(M)$, and thus so does the second summand. In particular, applying this for vector fields $w$ tangent to $\Sheet$, we get
  $$
  \int_{\Sheet} jump(u, w) ({}{\partial_t}\Sheet - \#u) = 0  \quad \forall \, w \in \vect(\Sheet)\,.
  $$
 (Here we use that any vector field on $\Sheet$ can be realized as a restriction of a divergence-free vector field.) 
 Taking $w:= \tanjump(u)$, we get
    $$
  \int_{\Sheet}(\tanjump(u), \tanjump(u)) (\partial_t\Sheet - \#u) = 0\,.
  $$
  Since $\tanjump(u) \neq 0$ almost everywhere, it follows that $\#u = \partial_t\Sheet$, ending the proof.
  \end{proof}
  \medskip
  \section{Appendix \ref{app:shapes}: Vorticity metric on shape spaces}
  \refstepcounter{AppCounter}
\label{app:shapes}
As we mentioned before, the metric $ \langle \,, \, \rangle_{\vs}$ on the space of vortex sheets $\VS(M)$ 
can be used to define a metric of hydrodynamical origin on the space of shapes in $M$.

Now we regard hypersurfaces $\Sheet\in \VS(M)$  as boundaries of $D_\Sheet:=D^+_\Sheet\subset M$ called shapes. Define the distance between two shapes of equal volume 
by means of the (weak) Riemannian metric  $\VS(M)$.
Note that this metric has an explicit description, which follows from the above consideration of submersion.
By definition, to find the square length $\langle v,v\rangle_{\vs}$ of a normal vector field $v$ attached to the boundary 
$\Sheet=\partial D_\Sheet$ (and of total zero flux through $\Sheet$) one needs to find infimum of 
$\int_M (u,u)\,\mu$ over all divergence-free
fields  $u =  \chiplus u^+ + \chimin u^-$ in $M$ whose normal component to $\Sheet$ is $v$.
The following is a corollary of the  submersion property.

\begin{corollary}\label{cor:horiz}
The length of vectors tangent to the manifold $\VS(M)$ of vortex sheets at $\Sheet$ can be found 
 as the solution of the Neumann problem: for a normal vector field $v = g\nu$  attached to $\Sheet$ (where $\nu$ is a fixed unit normal field)
 the infimum of $\int_M (u,u)\,\mu$ is attained on 
 gradient vector fields $u^\pm =\grad f^\pm$ and 
 $$
\langle v,v\rangle_{\vs}= \int_{D^+_\Sheet}(\grad f^+, \grad f^+)\,\mu 
+  \int_{D^-_\Sheet}(\grad f^-, \grad f^-)\,\mu\,,
$$
where $\Delta f^\pm=0$ in $D^\pm_\Sheet$ and $\partial f^\pm/\partial \nu|_\Sheet=   g$.

Equivalently, by applying the Stokes formula, one can rewrite this metric via the Neumann-to-Dirichlet operators
$\NTD^\pm$ on the domains  $D^\pm_\Sheet$ as
$$
\langle v,v\rangle_{\vs}=\langle  g, (\NTD^+ + \NTD^-)g \rangle_{L^2(\Sheet)}\,.
$$ 
\end{corollary}

This problem of reconstruction of a harmonic potential from the normal derivative data requires an integration 
of the fundamental solution in $M$ against the boundary data over $\Sheet$. 
So this metric is non-local in terms of $v$. 
Since the  Neumann-to-Dirichlet operators $\NTD^\pm$  have order $-1$ as pseudodifferential operators on the boundary $\Sheet$,
the corresponding metric is $H^{-1/2}$-like. It is interesting to compare it with other metrics
 on shape spaces  (see e.g. \cite{KLMP}), where (local) metrics of $H^s$-type with $s\ge 0 $ are usually used.
 
 Note that by regarding shapes $\Sheet=\partial D_\Sheet$ as measures $\mu_\Sheet$ supported on $D_\Sheet\subset M$
 one can define the Wasserstein distance between the shapes. Then Wasserstein's
 ${\rm Dist} (\mu_\Sheet, \mu_{\tilde\Sheet})$ is not greater than the distance between $\Sheet$ and $\tilde\Sheet$ 
 in the sense of the $\langle \,,\,\rangle_{\vs}$-metric, cf. \cite{Loesch}. This follows from the 
definition: in both cases one takes the $L^2$-norm of the vector fields moving the shape/mass, but 
in the Wasserstein distance one minimizes over all, 
not necessarily volume-preserving, diffeomorphisms of $M$.

\bibliographystyle{plain}
\bibliography{vs}

\end{document}